\documentclass[11pt,reqno]{amsart}  
\usepackage{color, amsmath,amssymb, amsfonts, amstext,amsthm, latexsym}
\usepackage[active]{srcltx}

\allowdisplaybreaks

\usepackage[normalem]{ulem}

\newcommand{\R}{{\mathbb R}}

\newcommand{\NN}{{\mathbb N}}
\newcommand{\RR}{{\mathbb R}}

\newcommand{\EX  }{{\mathbb E}}
\newcommand{\EE}{{\mathbb E}}
\newcommand{\WW}{{\mathcal W}}
\renewcommand{\SS}{{\mathcal S}}

\numberwithin{equation}{section}
\newtheorem{theorem}{Theorem}[section]

\newtheorem{lemma}[theorem]{Lemma}

\newtheorem{Th}[theorem]{Theorem}
\newtheorem{Prop}[theorem]{Proposition}

\usepackage{comment}
\title[Stochastic energy-critical NLS]
{The Energy-Critical Stochastic
nonlinear Schr\"odinger equation: Well-Posedness and Blow-up}

\author[A. Millet]{Annie Millet}
\address{SAMM, EA 4543,
Universit\'e Paris 1 Panth\'eon Sorbonne, 90 Rue de
Tolbiac, 75634 Paris Cedex France {\it and} LPSM, UMR 8001 
  Universit\'e de Paris et Sorbonne Universit\'e, France}
\email{annie.millet@univ-paris1.fr} 

\author[S. Roudenko]{Svetlana Roudenko}
\address{Department of Mathematics \& Statistics\\Florida International University,  Miami, FL, USA}
\curraddr{}
\email{sroudenko@fiu.edu}

\subjclass[2020]{60H15, 35R60, 35Q55}

\keywords{stochastic NLS equation, energy-critical, additive noise, multiplicative noise, Stratonovich integral, well-posedness, blow-up}

\begin{document}

\begin{abstract}
We investigate the focusing and defocusing energy-critical stochastic nonlinear Schr\"odinger equation, subject to random perturbations in the form of either additive or multiplicative (Stratonovich) noise. 
We establish local well-posedness for (random or deterministic) initial data $u_0$ in $\dot{H}^1(\mathbb{R}^n)$ or $H^1(\mathbb{R}^n)$, depending on the noise type. In the focusing case we provide quantitative estimates regarding the existence time and probability. Moreover, we derive blow-up criteria for solutions with positive energy in both cases of noise, provided that the noise intensity is sufficiently small, showing that blow-up occurs before a certain given positive time with positive probability, thus, extending deterministic results of Kenig-Merle \cite{KenMer} for the energy-critical NLS equation to the stochastic setting.
\end{abstract}

\maketitle

\tableofcontents

\section{Introduction}\label{s1}
We study the focusing energy-critical stochastic nonlinear Schr\"odinger (SNLS) equation  
subject to random perturbations of additive or multiplicative noise
\begin{align}\label{E:NLS}
\begin{cases} 
iu_t -\big(  \Delta u + \lambda \, |u|^{\frac4{n-2}} u \big) = \epsilon \, f(u), 
\quad (t,x)\in  [0,\infty) \times {\mathbb R}^n \, , \\
u(0,x)=u_0,
\end{cases}
\end{align}
where $\lambda=1$ stands for the focusing case and $\lambda=-1$ denotes the defocusing case, the initial data $u_0$ can be either deterministic or stochastic of $\dot H^1(\mathbb R^n)$ or $H^1(\mathbb R^n)$-type and the term $f(u)$ stands for a stochastic perturbation driven 
by the noise $W(dt,dx)$ white in time with some spatial regularity. 

In this paper we are interested in solutions 
in the {\it energy-critical} NLS equation (hence, the given nonlinearity) in the stochastic setting. The deterministic setting ($\epsilon = 0$) has been investigated in a breakthrough work of Kenig-Merle \cite{KenMer} in dimensions $n=3,4,5$, and that is the approach we take in this paper.

During their lifespans, solutions to the deterministic equation conserve the energy (or Hamiltonian) defined as
\begin{equation}\label{H}
H(u)=\frac{1}{2} \|\nabla u\|_{L^2(\RR^n)}^2 - \lambda\frac{n-2}{2n} \|u\|_{L^{\frac{2n}{n-2}}(\RR^n)}^{\frac{2n}{n-2}}.
\end{equation}
In some cases, when it is defined and finite, the mass (or $L^2$ norm) is also conserved, $M(u)=\|u\|_{L^2(\RR^n)}^2$, and so is the momentum
$\mathcal P (u) = \mbox{Im} \int \bar u(t,x) \, \nabla u(t,x) \, dx$. 

The deterministic NLS equation is invariant under the scaling: if $u(t,x)$ is a solution to \eqref{E:NLS} with $\epsilon=0$, 
then so is $u_\lambda(t,x) = \lambda^{\frac{n-2}2} u(\lambda^2 t, \lambda x)$. Under this scaling the Sobolev $\dot{H}^1_x$ norm of solutions is invariant, and that is why the NLS equation \eqref{E:NLS} with the given nonlinearity is called energy-critical.

Let 
\begin{equation}\label{E:Q} 
Q(x)= \Big( 1+\frac{|x|^2}{n(n-2)}\Big)^{-{\frac{n-2}{2}}},
\end{equation}
then $Q$ is a stationary solution to \eqref{E:NLS} with $\epsilon=0$, in other words, 
when the noise $W$ vanishes. Note that being time independent also means that $Q$ solves the non-linear elliptic equation 
$\Delta Q + Q^{\frac{n+2}{n-2}}=0.$ 
It is known that in this energy-critical case $Q$ is the only solution (up to symmetries) of  $\|Q\|_{L^{\frac{2n}{n-2}}(\mathbb R^n)} = C_n \|\nabla Q \|_{L^2(\mathbb R^n)}$, where $C_n$ is the sharp constant in the Sobolev inequality
\begin{equation}\label{E:Sobolev}
\|u\|_{L^{\frac{2n}{n-2}}(\RR^n)}\leq C_n \|\nabla u\|_{L^2{(\RR^n)}}, \quad  \forall\, u\in \dot{H}^1_x. 
\end{equation}
For the value of the constant $C_n$ and other details and references, e.g., see \cite[Sec. 5.1]{DR2015}.
This stationary solution $Q$ played an important role of a sharp threshold for solutions behavior in deterministic energy-critical NLS equation in a breakthrough work of Kenig-Merle \cite{KenMer}, which we recall next; the purpose of this paper is to extend this result to the stochastic setting with either additive or multiplicative noise. 

\begin{theorem}[\cite{KenMer}]\label{T:main-deter}
Let $u_0\in \dot H^1(\R^n)$, $n=3,4,5$. Suppose 
$H(u_0) < H(Q)$. 
\begin{itemize}
\item[{\bf (1)}] 
If 
$\|\nabla u_{0}\|_{L^2(\R^n)}<\| \nabla Q\|_{L^2(\R^n)}$,
then the solution $u(t)$ is global in both time directions and \\
$\|\nabla u(t)\|_{L^2(\R^n)} < \| \nabla Q\|_{L^2(\R^n)}$ for all times $t \in \mathbb R$.
Moreover, $u(t)$ scatters (approaches some linear NLS evolution in both time directions) in $\dot H^1(\mathbb R^n)$.

\item[{\bf (2)}] 
If 
$\|\nabla u_{0}\|_{L^2(\R^n)} > \|\nabla{Q}\|_{L^2(\R^n)}$,
then
$\|{\nabla u(t)}\|_{L^2(\R^n)} > \|{\nabla Q}\|_{L^2(\R^n)}
$ 
 for all $t \in I$ (the maximal time interval of existence).
Furthermore, if $|{x}| u_0\in L^2(\mathbb R^n)$ or $u$ is radial, then the solution blows up in finite time in both time directions. 
\end{itemize}
\end{theorem}

A similar result for the intercritical case (mass-supercritical and energy-subcritical) was developed by the second author and collaborators in \cite{Hol_Rou_2007, Hol_Rou_2008, DHR}, 
and some of it we extended to the stochastic setting in \cite{MilRou}. For that we relied on the local well-posedness result developed in the stochastic setting for the energy-subcritical NLS by de Bouard-Debussche in \cite{deB_Deb_CMP, deB_Deb_H1}.
In \cite{MilRou} we also obtained conditions for blow-up in finite time with positive  probability for positive energy (Hamiltonian) and a ``small" noise, thus, expanding on blow-up criteria for negative energy in \cite{deB_Deb_AnnProb, deB_Deb_PTRF}. In this paper besides obtaining local well-posedness, we also obtain a dichotomy {\it a la} Theorem \ref{T:main-deter}, including blow-up criteria for positive energy in the energy-critical case of the stochastic NLS equation. 
As in the deterministic case, to close a fixed-point argument (needed for local existence and uniqueness) in the energy-subcritical cases it is always possible to pull out some power of time $T$ via H\"older's inequality and then choose $T$ small enough to close the loop (for a contraction in the space considered for local well-posedness). 
In the energy-critical case, however, 
the power of $T$ in the contraction estimates vanishes due to the criticality of the nonlinearity. Consequently, local existence can not be established through a simple small-time argument, necessitating a more refined analysis in the spirit of \cite{KenMer}. We proceed in the same manner, however, the stochastic integral creates further challenges to control it. 
By carefully truncating the domain and utilizing uniform boundedness of noise spatial density and its smallness, we are able to handle the energy-critical case for the local 
well-posedness (see Theorems \ref{th_lwp-add} and \ref{th_lwp_mul}). 
We also prove that under the same conditions on $u_0$ as in the deterministic setting, which give global well-posedness (part (1) of Theorem \ref{T:main-deter}), 
in the stochastic setting we are able to provide information about the maximal existence time and show that when the strength of the noise goes to 0, this existence time goes to $\infty$ (see Theorems \ref{T*-add-genIC} and \ref{T*-multi-genIC}).
Under the conditions giving finite time blow-up in the deterministic case (part (2) of Theorem \ref{T:main-deter}), we prove that blow-up occurs with positive probability when the noise is sufficiently ``small" (see Theorems \ref{th_blowup-add} and \ref{th_blowup2}). 
Both results are consistent with what was proved in \cite{KenMer}.

Thus, this paper, in a sense, completes the dichotomy results for the stochastic NLS equation with additive and multiplicative (Stratonovich) noise. 
We mention that there are a few other works that examined the energy-critical NLS with forcing perturbations.
For example, in \cite{Oh_Oka} the authors obtain global well-posedness for the defocusing SNLS with additive noise and $\phi \in L^{0,1}_{2}$ (see other developments in the defocusing SNLS in \cite{Oh_Oka} and references therein); 
the SNLS with linear multiplicative noise via a rescaling approach is studied in \cite{BRZ2016} and  \cite{BRZ2017}, where 
the authors consider a {\it finite-dimensional} noise $W(t,x)=\sum_{j=1}^N \mu_j e_j(x) \beta_j(t)$ ($\mu_j$ are complex constants, the functions $e_j$ are smooth and bounded real-valued 
functions with derivatives with fast decay at infinity,  and $\beta_j$ are independent real-valued Brownian motions). 
In \cite{BRZ2016} local well-posedness is shown for nonlinearities that are $\dot H^1$-subcritical and $\dot H^1$-critical, and global well-posedness in the defocusing $\dot{H}^1$-subcritical case, and the focusing $L^2$-subcritical case. 
In \cite{BRZ2017}, the authors prove that in the $\dot H^1$-subcritical case (for more specific functions $e_j=f_j+c_j$ with $f_j\in C^\infty_b$ and $c_j\in {\mathbb R}$) in the non-conservative case, when  $Re(\mu_1)\neq 0$ (different from noise considered here and in works of de Bouard-Debussche), the probability of blow-up occurring 
either on $[0,T]$ or on  $[0,\infty)$, converges to 1 as $c_1\to \infty$ 
(depending on  the functions $f_j$), in some sense as the noise becomes very large. We remark that as in \cite{deB_Deb_AnnProb, deB_Deb_PTRF} 
the noise considered here is infinite-dimensional, conservative, and  sufficiently small, which allows us to show not only the well-posedness in the energy-critical 
case but also blow-up in finite time with positive probability for certain positive energy data. 
\smallskip

The paper is organized as follows: Section \ref{S:Prelim} reviews necessary notation,  interpolation inequalities, Strichartz estimates and it sets up the corresponding spaces for the well-posedness.   
Section \ref{S-lwp} establishes the local well-posedness of the energy-critical SNLS equation under both types of  
stochastic perturbations: specifically, assuming it is a Hilbert-Schmidt operator from $L^2$ to $H^1$ for the additive case, and a Radonifying operator from $L^2$ to
 $W^{1,\kappa}$ for the multiplicative case.  
In Section \ref{S:time}, for the focusing SNLS we provide  quantitative estimates for the time existence and its positive probability in both noise cases. 
Finally, in Section \ref{S:B} we derive the blow-up criteria for solutions with positive energy, specifying the noise smallness in both types of stochastic perturbation.

\subsection*{Acknowledgments} 
A.M.'s research has been conducted within the FP2M federation (CNRS FR 2036). S.R. was partially supported by the NSF grants DMS-2055130 and DMS-2452782.

\section{Preliminaries and notation}\label{S:Prelim}

We recall the Sobolev embedding of 
$\dot{H}^1(\RR^n)$ into 
$ L^{\frac{2n}{n-2}}(\RR^n)$, the inequality \eqref{E:Sobolev}, and state 
a slightly more general version, 
the embedding $\dot W^{1,\rho} (\mathbb{R}^n) \hookrightarrow L^{p}(\mathbb{R}^n)$, $\frac1{p} = \frac1{\rho} - \frac{1}{n}$, that we use in this paper, 
\begin{equation}\label{Sob-gen}
\|u\|_{L^p(\mathbb{R}^n)} \leq C \|\nabla u\|_{L^\rho(\mathbb{R}^n)}
\end{equation} 
for some (known) sharp constant $C$ (that depends on the dimension $n$ and integrability powers $p$ and $\rho$). 
\smallskip

For the arguments in the energy-critical setting that we study in this paper, following  \cite{KenMer}, it suffices to fix specific values of  $p$ and $\rho$ in \eqref{Sob}, thus, we define  
\begin{equation}    \label{gamma-rho}
p=\frac{2(n+2)}{n-2},  \quad \rho = \frac{2n(n+2)}{n^2+4}\quad \mbox{\rm  and }\quad \gamma = p.
\end{equation}
Then the pair $(\gamma, \rho)$ is called $L^2$-admissible and  $(\gamma,p)$ is an $\dot{H}^1$-admissible pair, if
$$
\frac{2}{\gamma} + \frac{n}{\rho} = \frac{n}{2} \quad \mbox{and} \quad  \frac{2}{\gamma} + \frac{n}{p} = \frac{n}{2}-1.
$$
For the given $p$ and $\rho$ in \eqref{gamma-rho}, we denote the sharp constant in Sobolev embedding \eqref{Sob-gen} by $C_{Sob}$, i.e., 
\begin{equation}\label{Sob}
\|u\|_{L^{\frac{2(n+2)}{n-2}}(\mathbb{R}^n)} \leq C_{Sob} \|\nabla u\|_{L^{\frac{2n(n+2)}{n^2+4}}(\mathbb{R}^n)}.
\end{equation}

Next, we recall the Schr\"odinger group $S(t) = e^{-it \Delta}$ (i.e., the propagator in the solution of the linear equation $i u_t = \Delta u$) and its time decay estimates. Let $\rho, \rho'$ be conjugate exponents, $\frac{1}{\rho} + \frac{1}{\rho'}=1$, then there exists a positive constant $C(\rho)$ such
that 
\begin{equation}		\label{decay}
\| S(t) u \|_{L^\rho_x} \leq C(\rho)\;  t^{-n(\frac{1}{2}-\frac{1}{\rho})}\;  \|u\|_{L^{\rho'}_x},\quad \forall u\in L^{\rho'}_x, ~~ \rho \geq 2.
\end{equation}

The following well-known Strichartz inequalities are satisfied by the Schr\"odinger group $S(t)$ for the pair $(\gamma, \rho)$ in \eqref{gamma-rho},
\begin{equation}\label{gamma-rho-2}
\| S(t) f\|_{L^\gamma_t L^\rho_x} \leq C_{Str} \|f\|_{L^2_x},  \end{equation}
\begin{equation}\label{Stri_int}
\Big\| \int_0^t S(t-s) g(s) ds \Big\|_{L^\gamma_{[0,T]} L^\rho_x} \leq C_{Str} \|g\|_{L^{q'}_{[0,T]}L^{r'}_x}		\end{equation}
for any $L^2$-admissible pair $(q,r)$ (not necessarily equal to $(\gamma, \rho)$).  
Note that $(\infty, 2)$ is an $L^2$-admissible pair. 
Given the Sobolev inequality \eqref{Sob} and $u(t,x) \in C_0^{\infty}(\mathbb R^{n+1})$, we also have 
$$
\|u\|_{L^\gamma_t  L^p_x} \leq C_{Sob} \|\nabla_x u\|_{L^\gamma_t  L^\rho_x},
$$
which motivates us to make the following notation. 

For any $T>0$,  let
\begin{equation} 	\label{def-S(T)-W(T)}
{\mathcal S}[T]= \{ u :  
\|u\|_{L^\gamma_{[0,T]}  L^p_x} <\infty \} \quad  \mbox{\rm and} \quad  {\mathcal W}[T]= \{ u  : 
\|  u \|_{L^\gamma_{[0,T]}  L^\rho_x}  <\infty\} ,
\end{equation} 
and for $u\in {\mathcal S}[T]$ (resp. $u\in {\mathcal W}[T]$), set 
$$
\qquad 
\|u\|_{{\mathcal S}[T]} := \|u\|_{L^\gamma_{[0,T]}  L^p_x} \qquad \mbox{(resp.} \quad \|u\|_{{\mathcal W}[T]} := \|u\|_{L^\gamma_{[0,T]} L^\rho_x}).
$$

As usual we denote by $C$ (resp. $C(c)$) a positive constant (resp. a positive constant depending on some parameter $c$) that can change from one line to the next. 

\section{Local well-posedness} \label{S-lwp}
We first have to prove local well-posedness in a stochastic setting. 

\subsection{Local well-posedness: additive  stochastic perturbation}\label{S-lwp-add}

Let $L^2(\RR^n)$ (resp. $H^s(\RR^n)$) denote the set of ${\mathbb C}$-valued  functions  of $L^2(\RR^n)$ (resp. $H^s(\RR^n)$). 
Let $(\Omega, {\mathcal F},P)$ be a probability space endowed with a filtration $({\mathcal F}_t)_{t\geq 0}$.
Let $(\beta_k)_{k\in \NN}$ be a sequence of independent real-valued Brownian motions on $[0,\infty)$ with respect to the filtration 
$({\mathcal F}_t)_{t\geq 0}$ and let $(e_k)_{k\in \NN}$ be an orthonormal basis of the set $L^2(\RR^n)$ of complex-valued square integrable
functions on $\RR^n$. We define a Wiener process on the space $L^2(\RR^n)$ as follows
\begin{equation} \label{R-noise-add}
W(t,\cdot,\omega)=\sum_{k=0}^\infty \beta_k(t,\omega) \phi \, e_k(\cdot),
\end{equation}
where $\phi$ is a bounded operator from $L^2(\RR^n)$ into itself.  

For any $s>0$, let $L^{0,s}_2$ denote the set of Hilbert-Schmidt operators from $L^2_x$ to $W^{s,2}_x$ and denote by
$\| \; .\; \|_{L^{0,s}_2}$ the corresponding norm.

In this section, we consider the following stochastic nonlinear Schr\"odinger equation in the energy-critical case  in dimension $n=3,4,5$, that is, for $u_0\in \dot{H}^1(\RR^n)$ a.s.  and $\lambda \in \{-1,1\}$, consider
\begin{equation}\label{NLS-additive}
i d_t u - (\Delta u + \lambda |u|^{\frac{4}{n-2}} u)  = dW_t, ~~ u(0) = u_0.
\end{equation}

Note that if $\phi$ is defined via a kernel ${\mathcal{K}}$, that is, for any square integrable function $u$, 
$\phi \,u(x)=\int_{\RR^n}{\mathcal K}(x,y) u(y) dy$, then the correlation function of the noise is 
\[ 
\qquad \qquad  \EX\Big( \frac{\partial W}{\partial t}(t,x) \frac{\partial W}{\partial t}(s,y)\Big) = c(x,y) \delta_{t-s}, \quad 
\mbox{\rm with} \quad c(x,y) = \int_{\RR^n} {\mathcal K}(x,z) {\mathcal K}(y,z) dz.\]

The following theorem provides a local well-posedness result in this case.
\begin{Th}      \label{th_lwp-add}
Let $n=3,4,5$. 
Define $\gamma, \rho, p$ as in \eqref{gamma-rho}, and let $u_0$ be ${\mathcal F}_0$-measurable such that
$\| u_0\|_{\dot{H}^1_x} \leq A$ for some $A>0$, and suppose that 
$\phi \in L^{0,1}_2$. 
Then  there exists a stopping time  $\tau(u_0) >0 $ a.s. such that the equation \eqref{NLS-additive} has a unique solution $u(t)$ on the interval $[0, \tau(u_0))$ and a.s. 
$u\in C([0,\tau(u_0)) ; \dot{H}^1(\RR^n))$;  furthermore, a.s. $u\in  L^\gamma(0,\tau(u_0) ; L^{p}(\RR^n))$ and 
$\nabla u \in  L^\gamma(\Omega, L^\gamma(0,\tau(u_0) ; L^\rho_x(\RR^n)) )$ 
\end{Th}

\begin{proof}
The Duhamel formula shows that we look for a process $u$ such that 
\begin{equation} 	\label{Duhamel-add}
u(t) = S(t) u_0 - i\lambda \int_0^t S(t-s)  |u(s)|^{\frac{4}{n-2}} u(s) ds 
 + i\int_0^t S(t-s) dW(s). 
\end{equation}

Let $a,b,T>0$ (to be chosen later) and, recalling \eqref{def-S(T)-W(T)}, we denote by ${\mathcal X}_{[T]}(a,b)$ the set of processes $u: \Omega \times [0,T] \times \RR^n \to {\mathbb C}$ such that 
$$ 
\|  u\|_{{\mathcal S}[T]} \leq a\; \mbox{\rm a.s.}  \quad \mbox{\rm and} \quad    \|\nabla u\|_{ {\mathcal W}[T]} \leq b\; \mbox{\rm a.s.}, 
$$ 
then for $ u \in {\mathcal X}_{[T]}(a,b)$ we define 
$$
\|u\|_{{\mathcal X}_{[T]}(a,b)} :=\| u\|_{{\mathcal S}[T]} +  \|\nabla u\|_{ {\mathcal W}[T]}.
$$

Let  ${\mathcal G}$ denote the map defined on ${\mathcal X}_{[T]} (a,b)$
by ${\mathcal G}(u)(t) = S(t) u_0 + {\mathcal T}(u)(t) + I(t)$, where
\begin{align*}
{\mathcal T}(u)(t) = &\;  -i \lambda \int_0^t 
S(t-s)  |u(s)|^{\frac{4}{n-2}} u(s) ds, \quad 
I(t)=\;  i \int_0^t S(t-s) dW(s). 
\end{align*}

{\bf Step 1.}  
We first upper estimate the term involving the initial condition. Using $\nabla S(t)=S(t) \nabla$ and $\| u_0\|_{\dot{H}^1_x} \leq A$ a.s.,  the Strichartz inequality \eqref{gamma-rho-2} implies  
\begin{equation}\label{upper-ICW}
\| \nabla S(\cdot) u_0\|_{{\mathcal W}{[T]}} \leq  C_{Str} \|\nabla u_0\|_{L^2_x} \leq  C_{Str} A \quad \mbox{\rm  a.s.   for any}\quad   T>0. 
\end{equation}

We now show that for any $\delta >0$ one can choose $T>0$ such that $\| S(\cdot) u_0\|_{{\mathcal S}{[T]}} \leq \delta$ a.s.
Indeed,  applying the Sobolev embedding \eqref{Sob}, we deduce that  for any $T>0$ we have $\| S(\cdot) u_0 \|_{{\mathcal S}{[T]}} \leq C_{Sob} \, C_{Str} A$ a.s.
Since for every $T>0$, the map $t\in [0,T] \mapsto    \| S(\cdot) u_0 \|_{{\mathcal S}{[t]}} $ is a.s. continuous, starts from 0 when $t=0$, and is upper bounded, we deduce that given any $\delta>0$ and almost every $\omega$, there exists a (random, ${\mathcal F}_0$-measurable) time $\tau_0(\delta)$ such that 
\begin{equation}\label{upper-ICS}
\| S(\cdot) u_0 (\omega) \|_{{\mathcal S}{[\tau_0(\delta)(\omega) ]}} \leq \delta\quad \mbox{\rm  a.s.}
\end{equation} 
\smallskip
     
We next prove an upper estimate for $\| \nabla I\|_{L^\gamma(\Omega; {\mathcal W}{[T]} ) }$ in terms of $\phi$ and $T$. 
  
First, observe that $\nabla I(t)=i \int_0^t S(t-s) d(\nabla W(s))$. Using the It\^o isometry and the fact that $S(\cdot)$ is an $L^2_x$-isometry,  we deduce that 
\[ \EE\big( \|\nabla  I(t,\cdot)\|_{L^2_x}^2 \big) = \sum_{k\geq 0}  \int_0^t \int_{\RR^n} |S(t-s) \nabla (\phi e_k))(x) |^2 dx ds= \sum_{k\geq 0}t  \|\nabla \phi e_k\|_{{L^2_x}}^2 =t \| \phi\|_{L^{0,1}_2}^2.\]

Let $\{ \phi_n\}_{n}$ be a sequence of elements of $L^{0,2}_2$ such that $\phi_n \to \phi$ in $L^{0,1}_2$ (with $\phi \in  L^{0,1}_2 $) 
 as $n\to \infty$, $\{W_n\}_n$ be the sequence of corresponding Brownian motions and $I_n(t)=\int_0^t S(t-s) dW_n(s)$. 
Then $\EE\big( \|\Delta  I_n(t,\cdot)\|_{L^2_x}^2 \big) = \int_0^t  \sum_{k\geq 0}\| \Delta \phi_n e_k\|_{{L^2_x}}^2  ds =t \| \phi_n\|_{L^{0,2}_2}^2$. Therefore, $\nabla I_n(t)$ 
is an $H^1_x$-valued Gaussian process, and
the embedding $\dot H^1_x\subset L^{\rho}_x$ implies that it is also an $L^\rho_x$-valued Gaussian process. Hence, the moments of $\| \nabla I_n(t)\|_{L^\rho_x}$ are comparable. 
Furthermore, the It\^o isometry implies that $I_n $ converges to $I$ in $L^\gamma(\Omega ; L^2_{[0,T]} H^1_x)$. 
We then proceed as in the proof of \cite[Theorem~3.1]{deB_Deb_H1}.    
The Fubini theorem implies the existence of  positive  constants $c_1$ depending on $\nu,\rho$ and $c_2$ depending on $\rho$ (and independent of $n$) such that for any $\nu\geq 1$ we get
\begin{align*}
\EE\Big( \int_0^T \|\nabla I_n(t)\|_{L^\rho_x}^\nu dt\Big) =  &
 \int_0^T \EE\big( \| \nabla I_n(t)\|_{L^\rho_x}^\nu\big) dt 
  \leq  \; c_1 \int_0^T \Big\{ \EE\big( \|\nabla I_n(t)\|_{L^\rho_x}^\rho\big) \Big\}^{\frac{\nu}{\rho}} \, dt\\
  \leq &\; c_1 \int_0^T \Big\{ \int_{\RR^n} \;  c_2  \big[  \EE\big( |\nabla I_n(t,x)|^2\big) \big]^{\frac{\rho}{2}} dx \Big\}^{\frac{\nu}{\rho}} \; dt, 
  \end{align*}
where in the last estimate we used once more the Fubini theorem and the fact that the random variable $\nabla I_n(t,x)$ is Gaussian (and thus, any power $\nu \geq 1$ can be used).

Let $\gamma, \rho$ be defined by \eqref{gamma-rho}. Then the Minkowski inequality applied to the $\|\cdot\|_{L^{{\rho}/{2}}_x}$ 
  (for the product of the counting measure on $\NN$  and $1_{[0,t]} ds$) yields 
\begin{align*}
\Big\{ \int_{\RR^n} \big[  \EE\big( |\nabla I_n(t,x)|^2\big) \big]^{\frac{\rho}{2}} dx \Big\}^{\frac{2}{\rho}}  \leq &\sum_{k\geq 0}  
  \int_0^t \Big\{ \int_{\RR^n}    |S(t-s) (\nabla \phi_n e_k)(x)|^\rho dx  \Big\}^{\frac{2}{\rho}} ds
 \leq \sum_{k\geq 0}  \| S(\cdot) (\nabla \phi_n e_k)\|_{L^2_{[0,t]} L^\rho_x}^2 \\
 \leq & \sum_{k\geq 0} \Big\{ \int_0^t \| S(\cdot)(\nabla \phi_n e_k)\|_{L^\rho_x}^\gamma ds \Big\}^{\frac{2}{\gamma}} t^{2-\frac{2}{\gamma}},
\end{align*} 
where the last inequality follows from H\"older's inequality and 
the fact that we use the $L^2$-admissible pair $(\gamma, \rho)$, and thus, $\gamma >2$. 
Strichartz's inequality \eqref{gamma-rho-2} implies $\|S(\cdot) (\nabla \phi_n e_k)\|_{L^\gamma_{[0,t]} L^\rho_x} \leq C_{Str} 
\| \nabla(\phi _ne_k)\|_{L^2_x}$  for every $k\geq 0$. Therefore, 
\begin{align*}	
\EE\Big( \int_0^T \| \nabla I_n(t)\|_{L^\rho_x}^\gamma dt\Big) & \leq c_1\, c_2^{\frac{\gamma}{\rho}} \, \int_0^T C_{Str}^\gamma  \| \phi_n\|_{L{_2^{0,1}}}^\gamma t^{\gamma-1} dt 
\leq C(\gamma, \rho) \, C_{Str}^\gamma \|\phi_n\|_{L^{0,1}_2}^\gamma T^\gamma.  
\end{align*} 
This estimate proves that the sequence $\{ \nabla I_n\}_n$ is bounded 
in  $L^\gamma(\Omega ; {\mathcal W}{[T]})$. 
Since it converges to $\nabla I$ in $L^2(\Omega ; L^\infty_{[0,T]} L^2_x)$, its {weak limit in $L^\gamma(\Omega ; {\mathcal W}{[T]})$}  is $\nabla I$; letting $n\to \infty$, we deduce 
\begin{equation}		\label{mom-IW}
 \| \nabla I\|_{L^\gamma(\Omega ; {\mathcal W}{[T]})} \leq C(\gamma, \rho) \, C_{Str} \|\phi\|_{L^{0,1}_2} T, 
\end{equation}
while the Sobolev embedding  \eqref{Sob} yields
\begin{equation}		\label{mom-IS}
\| I\|_{L^\gamma(\Omega ; {\mathcal S}[T])} \leq C(\gamma, \rho) C_{Sob} C_{Str} \|\phi\|_{L^{0,1}_2} T.
\end{equation}
Hence, $\|I\|_{{\mathcal S}[T]} + \| \nabla I\|_{{\mathcal W}[T]} < \infty$ a.s., thus, given positive constants $\alpha,\beta$, we may define the stopping time
\begin{equation}		\label{def_tau-ab}
\tau(\alpha,\beta) = \inf\{ t>0 \; : \; \|I\|_{{\mathcal S}[t]} \geq \alpha\} \wedge \inf\{ t>0\; : \; \| \nabla I\|_{{\mathcal W}[t]} \geq \beta\}.
\end{equation}
\smallskip

We next upper estimate the non-linear term ${\mathcal T}(u)$ and follow computations from \cite{KenMer}. Once more, using \eqref{Sob}, we only need to bound above 
$\| \nabla {\mathcal T}(u)\|_{{\mathcal W}[T]}$. 
Using \eqref{Stri_int} and setting $\sigma = \frac{2}{n-2}$, we deduce for every $L^2$-admissible pair $(q,r)$ that 
\[ 
\|\nabla {\mathcal T}(u)\| _{L^\gamma_{[0,T]} L^\rho_x} \leq C_{Str}  (2\sigma +1) 
\| |u(s)|^{2\sigma} \nabla u(s) \|_{L^{q'}_{[0,T]} L^{r'}_x} .\]
Take $q=2$ and $r=\frac{2n}{n-2}$, so $(q,r)$ is an $L^2$-admissible pair with $r'=\frac{2n}{n+2}$, which implies $\frac{1}{r'} = \frac{2\sigma}{p} + \frac{1}{\rho}$. Hence,  H\"older's inequality
implies that $\| |u|^{2\sigma} \nabla u\|_{L^{r'}_x} \leq \|u\|_{L^p_x}^{2\sigma}  \|\nabla u\|_{L^\rho_x}$. 

Furthermore, $\frac{2\sigma +1}{\gamma} = \frac{1}{2}$. Hence,  H\"older's inequality applied to the measure $ dt$ on the set $ [0,T]$  implies that for every $T>0$
\begin{align}\label{upper_T(u)-W}
\| \nabla {\mathcal T}(u) \|_{{\mathcal W}[T]}  \leq & \; C_{Str} (2\sigma +1) \|u\|_{ {\mathcal S}[T]}^{2\sigma }
\, \|\nabla u\|_{ {\mathcal W}[T]} ,   \\
\|  {\mathcal T}(u) \|_{ {\mathcal S}[T]} \leq  & \; C_{Sob} C_{Str} (2\sigma +1) \|u\|_{{\mathcal S}[T]}^{2\sigma }
\, \|\nabla u\|_{ {\mathcal W}[T]}. 		\label{upper_T(u)-S} 
\end{align} 
We next choose $a,b$ such that if some random $\tau$
is small enough, then ${\mathcal G}$ maps $ {\mathcal X}_{[{\tau}]}(a,b)$ into itself a.s. Collecting the estimates \eqref{upper-ICW}--\eqref{upper_T(u)-S}, we deduce that 
for positive constants $\alpha$, $\beta$, $\delta$, determined later, and setting 
$\tau:=\tau_0(\delta) \wedge \tau(\alpha,\beta)$, we have 
\begin{align*}
 \| \nabla {\mathcal G}(u)\|_{{\mathcal W}[\tau]} \leq &\;  C_{Str} \; \Big[ A + \frac{n+2}{n-2} \,  a^{\frac{4}{n-2} } \, b \Big] + \beta , \\ 
  \|  {\mathcal G}(u)\|_{ {\mathcal S}[\tau]} \leq & \; \delta +  C_{Sob}\,  C_{Str}  \;  \frac{n+2}{n-2} \,  a^{\frac{4}{n-2} } \, b  + \alpha. 
 \end{align*}
Now, set $b=3 C_{Str} A$ and $\beta = \frac{b}{3}$. Choose $a>0$  such that 
\[ C_{Str}\, \frac{n+2}{n-2} \,  a^{\frac{4}{n-2} } \leq \frac{1}{3} \quad \mbox{\rm  and } \quad 
C_{Sob} C_{Str}\, \frac{n+2}{n-2} \,  a^{\frac{6-n}{n-2} } \, b \leq \frac{1}{3},
\]
noting that for the last inequality we need $6-n>0$, therefore, the restriction on the dimension $n=3,4,5$. 
Let $\alpha = \frac{a}{3}$ and  $\delta = \frac{a}{3}$ (which defines $\tau$). We then obtain that ${\mathcal G}$ maps  $ {\mathcal X}_{[{\tau}]}(a,b)$ into itself a.s.

Note also that if we keep $b=3 C_{Str} A$ and $\beta = \frac{b}{3}$, but choose $\tilde{a}\leq a$ with $\tilde{\delta} = \tilde{\alpha}=   \frac{\tilde{a}}{3}$, then for $\tilde{\tau}:=\tau_0(\tilde{\delta}) \wedge
\tau(\tilde{a},{b})$ we obtain that ${\mathcal G}$ maps  $ {\mathcal X}_{[{\tilde{\tau}}]}(\tilde{a},b)$ into itself a.s.
\smallskip

{\bf Step 2.} We next prove that for the previous choice of  $a$,$b$, and $\tau$, the mapping ${\mathcal G}$ is a contraction of ${\mathcal X}_{[\tau]}(a,b)$ a.s. Observe that for $u_1, u_2$, we have
${\mathcal G}(u_1)- {\mathcal G}(u_2) = {\mathcal T}(u_1) -  {\mathcal T}(u_2) $. 
Note that
\[ \big| \nabla \big( |u_1|^{2\sigma} u_1 - |u_2|^{2\sigma} u_2\big) \big| \leq (2\sigma +1) \big[  |u_1|^{2\sigma} |\nabla (u_1-u_2)|  + 2\sigma (|u_1|^{2\sigma-1} 
+ |u_2|^{2\sigma -1}) |u_1-u_2| |\nabla u_2| \big]
.\]
Using once more \eqref{Stri_int} with the $L^2$-admissible pair $(q,r)$, where $q=2$ and $r=\frac{2n}{n-2}$ (which implies 
$r'=\frac{2n}{n+2}$ and $\frac{1}{2}=\frac{\sigma+1}{\gamma}$) and H\"older's inequality, we obtain
\begin{align*}
 \| \nabla {\mathcal T}(u_1) - \nabla {\mathcal T}(u_2)\|_{{\mathcal W}[\tau]} \leq &\; C_{Str} (2\sigma +1) \big[ \|u_1\|_{{\mathcal S}[\tau]}^{2\sigma} \| \nabla (u_1-u_2)\|_{{\mathcal W}[\tau]} \\
 &\qquad \quad + 2\sigma \big( \|u_1\|_{{\mathcal S}[\tau]}^{2\sigma -1} + \| u_2\|_{{\mathcal S}[\tau]}^{2\sigma -1} \big) \|u_1-u_2\|_{{\mathcal S}[\tau]} \| \nabla u_2\|_{{\mathcal W}[\tau]}
 \big] \\
 \leq &\; C_{Str} \; \frac{n+2}{n-2} \; \Big[ a^{\frac{4}{n-2}} \| \nabla(u_1-u_2)\|_{{\mathcal W}[\tau]} + \frac{8}{n-2} \; a^{\frac{6-n}{n-2}} \, b\, \|u_1-u_2\|_{{\mathcal S}[\tau]}\Big] \\
 \leq &\; C_{Str} \; \frac{n+2}{n-2} a^{\frac{6-n}{n-2}} \max \Big(a,\frac{8 b}{n-2}\Big) \, \big[ \|u_1-u_2\|_{{\mathcal S}[\tau]} + \| \nabla (u_1-u_2)\|_{{\mathcal W}[\tau]}\big] .
 \end{align*}
Since $\| {\mathcal T}(u_1)-{\mathcal T}(u_2)\|_{{\mathcal S}[\tau]} \leq C_{Sob} \| \nabla {\mathcal T}(u_1)-\nabla {\mathcal T}(u_2)\|_{  {\mathcal W}[\tau]} $, we deduce
\[ \| {\mathcal T}(u_1)-{\mathcal T}(u_2)\|_{{\mathcal X}_{[\tau]}(a,b)} \leq (1+C_{Sob}) C_{Str} \;  \frac{n+2}{n-2}\,  a^{\frac{6-n}{n-2}} \max \Big(a,\frac{8 b}{n-2}\Big) \, 
 \|u_1-u_2\|_{{\mathcal X}_{[\tau]}(a,b)} .\]
Choose  $b$ as in Step 1 and take $\tilde{a}$ so that  
\[ {(1+C_{Sob})} C_{Str}\, \frac{n+2}{n-2} \,  \tilde{a}^{\frac{4}{n-2} } \leq \frac{1}{3} \qquad  \mbox{\rm and}  \qquad 
{(1+C_{Sob}) } C_{Str}\, \frac{n+2}{n-2} \,  \tilde{a}^{\frac{6-n}{n-2} } \, \frac{8}{n-2} b \leq \frac{1}{3}.
\] 
Then set $\tilde{\delta}=\frac{\tilde{a}}{3}$. 
Hence, for $\tilde{\tau}$ defined as at the end of Step 1, we 
obtain
$\| {\mathcal G}(u_1)-{\mathcal G}(u_2)\|_{{\mathcal X}[\tilde{\tau}]} \leq \frac{2}{3} \|u_1-u_2\|_{{\mathcal X}[\tilde{\tau}]}$, thus, ${\mathcal G}$ is a contraction, and therefore,
 there exists a unique element in ${\mathcal X}_{[\tilde{\tau}]}(a,b)$ such that ${\mathcal G}(u)=u$, and hence, $u$ a unique solution to \eqref{Duhamel-add}.

Let $\tau(u_0)$ denote the supremum of all stopping times $\tilde{\tau}$ such that the solution of 
\eqref{Duhamel-add} exists and is unique on the random time interval $[0, \tilde{\tau}]$. Note that  the solution $u(t)$ is now defined on $[0,\tau(u_0))$.

We next prove that $u\in C([0,\tilde{\tau}); \dot{H}^1_x)$ a.s. for any $\tilde \tau < \tau(u_0)$.  Since $u_0\in \dot{H}^1_x$ a.s., the linear flow $S(\cdot) u_0 \in C([0,\tilde \tau] ; \dot{H}^1_x$) a.s.
Then recalling that $(\infty,2)$ and $(2,r)$ (with $r=\frac{2n}{n-2}$) are the $L^2$-admissible pairs, the Strichartz estimate
\eqref{Stri_int} and the computations proving \eqref{upper_T(u)-W}  yield 
\[  \|\nabla {\mathcal T}(u)\|_{L^\infty_{(0,\tilde \tau)} 
L^2_x} \leq C (2\sigma +1) \|u\|_{{\mathcal X}_{[\tilde \tau]}(a,b)}^{\frac{n+2}{n-2}} <\infty.\] 
Since $ \nabla {\mathcal T}(u)(\cdot)$ is a.s. continuous on the (random) time interval $[0, \tilde \tau]$,   
and $I(\cdot)\in C([0,\tilde \tau]; \dot{H}^1_x)$, 
the same holds on the interval $[0,\tau(u_0))$; then
the assumption $\phi \in L^{0,1}_2$ and the It\^o isometry complete the proof.
\end{proof} 

\subsection{Local well-posedness: multiplicative  stochastic perturbation}\label{S-lwp-multi}
Denoting $H=L^2_{\RR}(\RR^n):= L^2(\RR^n;\RR)$  and letting  $B$ to be a martingale type-2 Banach space  (such as Sobolev spaces $W^{k,p}_{\RR}(\RR^n)$), 
we  consider bounded operators from $H$ into $B$. In that case, 
Hilbert-Schmidt operators are replaced by $\gamma$-radonifying ones (see e.g. \cite{Brz_1997, Brz_Mil}). 
The set of Radonifying operators from $H$ to $B$ is denoted by $R(H,B)$ and endowed with the norm
$$ 
\|\phi\|_{R(H,B)}^2 = \tilde{\mathbb E} \Big( \big| \sum_{k\geq 1}\gamma_k \phi e_k\big|_B^2 \Big), 
$$
where $\{e_k\}_{k\geq 1} $ is any orthonormal basis of $H$ and $\{ \gamma_k\}_{k\geq 1}$ is any sequence of independent real-valued standard Gaussian random variables on some probability space 
$(\tilde{\Omega}, \tilde{\mathcal F}, \tilde{P})$. For conciseness, we also denote $F_\phi := \sum_{k\geq 0} (\phi e_k)^2$.
In this section while considering multiplicative noise, we suppose that the operator $\phi$ is Radonifying from $L^2$ to some Sobolev space $W^{1,\kappa}$, to be precise, 
\begin{equation}\label{E:Radon}
\qquad \phi \in  R\big(L^2(\RR^n), W^{1,\kappa} (\RR^n)\big),  ~~~~\kappa = \frac{n(n+2)}{n-2}.
\end{equation}
\smallskip

Let $u_0\in \dot{H}^1 (\mathbb R^n)$ a.s. and consider the stochastic nonlinear Sch\"odinger equation with Stratonovich multiplicative noise 
(in order to enforce mass conservation when the noise is real-valued). Note that to prove local well-posedness, since mass conservation is not needed, we
can assume that the noise is complex-valued.
\begin{equation}\label{E:SNLS-m} 
i d_t u - \big( \Delta u + \lambda |u|^{\frac{4}{n-2}} u\big) = u \circ dW_t,\quad u(0)=u_0. 
\end{equation}

Rewriting the equation \eqref{E:SNLS-m} in the integral form using It\^o's stochastic integral, we obtain 
\begin{equation}\label{NLS_Ito_critical}
\begin{aligned}
u(t) = &S(t) u_0 - i\lambda \int_0^t S(t-s) \big( |u(s)|^{\frac{4}{n-2}} u(s)\big) ds\\ 
&- i \int_0^t  S(t-s) \big[ u(s) dW(s)\big] - \frac{1}{2} \int_0^t S(t-s) \big[ u(s) F_\phi\big] ds.
\end{aligned}
\end{equation}

We are now ready to discuss the local well-posedness of the equation \eqref{NLS_Ito_critical}. We remark that unlike the additive noise framework (as in Section \ref{S-lwp-add}), where we only used
$u_0\in \dot{H}^1_x$, we require $u_0\in H^1_x$ due to the estimates of the gradient of stochastic integrals. 

\begin{Th}      \label{th_lwp_mul}
Let $n=3,4,5$. Set $p, \rho, \gamma$ as in \eqref{gamma-rho} 
and let $u_0\in H^1_x$ be ${\mathcal F}_0$-measurable such that $\|u_0\|_{{H}^1_x} \leq A$ a.s. Suppose that \eqref{E:Radon} holds, i.e., $\phi \in  R(L^2_x, W^{1,\kappa}_x)$, $\kappa=\frac{n(n+2)}{(n-2)}$. Then there exists a stopping time $\tau(u_0)>0$ a.s. such that the equation \eqref{NLS_Ito_critical} has a unique solution $u(t)$ on the interval $[0, \tau(u_0))$ and a.s. 
$u\in C([0,\tau(u_0)); {H}^1_x) $; furthermore, $\nabla u \in  L^\gamma(\Omega ; L^\gamma(0, \tau(u_0); \dot{W}^{1,\rho}_x))$, and consequently, $u\in  L^\gamma(\Omega; L^\gamma(0,\tau(u_0); L^{p}_x))$.
\end{Th}
\begin{proof}
To handle the nonlinearity, we introduce a truncation function $\theta : [0,\infty) \to [0,1]$, of class $C^\infty$, such that $\theta(x)=1$  for $x\in [0,1]$ and $\theta(x)=0$ for $x\in [2,\infty)$. Then, for any constant $c>0$, set
\begin{equation}\label{theta_c}
\theta_c(x)= \theta\Big( \frac{x}{c}\Big).
\end{equation}
We now consider the truncated equation for some positive constants $a$, $b$, $t>0$, and writing $\sigma = \frac{2}{n-2}$,
\begin{align}\label{eq_trunc}
v(t) =  S(t) u_0 - & i\lambda \int_0^t \theta_a\big(\|v(\cdot)\|_{\SS[s]} \big) \, \theta_b\big( \|  v(\cdot) \|_{\WW{[s]} }\big) S(t-s)\big[   |v(s)|^{2\sigma} v(s) \big] ds \nonumber \\
&\; - i\int_0^t S(t-s)\big[ v(s) dW(s)\big] - \frac{1}{2}\int_0^t  S(t-s)\big[ v(s) F_\phi\big] ds. 
\end{align}
For $c>0$ and $T>0$ (specified later), define the space  
\begin{equation}\label{E:Yspace}
{\mathcal Y}_{[T]} {(a,b,c) }=  \{ u:\Omega \times [0,T] \times \RR^n \to {\mathbb C} :  \|u\|_{L^\gamma(\Omega ; L^\infty_{[0,T]} H^1_x)}  \leq c \} \cap {\mathcal X}_{[T]}(a,b), 
\end{equation}
where for  $a,b>0$ (specified later) we also denote
\begin{equation}\label{E:Xspace} 
\begin{aligned}
{\mathcal X}_{[T]}(a,b) =\Big\{ u:\Omega \times [0,T] \times \RR^n  & \to {\mathbb C} : ~  \|u\|_{L^\gamma (\Omega ; {\mathcal S}{[T]})} \leq a \\ 
&\quad ~ \mbox{\rm and} ~~
\|  u\|_{L^\gamma (\Omega ; {\mathcal W}{[T]})}+ \| \nabla  u\|_{L^\gamma (\Omega ; {\mathcal W}{[T]})} \leq b \big\}.
\end{aligned}
\end{equation} 
We define the mapping $\mathcal G$ on {${\mathcal Y}_{[\tau]}(a,b,c)$ } by the truncated integral equation \eqref{eq_trunc} as
\begin{equation}\label{E:all}
{\mathcal G}(u) = S(t) u_0 + 
T_1(u) +T_2(u)+ T_3(u),
\end{equation}
where
\begin{align*}
T_1(u)(t) = &\;  -i \lambda \int_0^t \theta_a\big(\|u(\cdot)\|_{\SS_{[s]} } \big) \theta_b\big( \|  u(\cdot) \|_{\WW_{[s]}}\big) S(t-s) [|u(s)|^{2\sigma} u(s)] ds, \\
T_2(u)(t)=&\;  -i \int_0^t S(t-s)\big[ u(s) dW(s)\big] , \\
T_3(u)(t) =&\; - \frac{1}{2}\int_0^t  S(t-s)\big[ u(s) F_\phi\big] ds. 
\end{align*}

{\bf Step 1.}  We first show how to choose $a, b, c$ and $\tau(\omega)>0$ a.s. such that ${\mathcal G}$ maps {${\mathcal Y}_{[\tau]}(a,b,c)$ } into itself.  
The estimates \eqref{upper-ICW} and \eqref{upper-ICS} imply that for any $\delta >0$ there exists an ${\mathcal F}_0$-measurable random time $\tau(\delta)$ such that 
\begin{equation}\label{IC-multi}
\| S(\cdot)u_0\|_{L^\gamma(\Omega ; {\mathcal S}{[\tau(\delta)]})} \leq \delta \quad \mbox{\rm and} \quad \| S(\cdot) u_0\|_{L^\gamma (\Omega ; {\mathcal W}{[T]})} 
+ \|  \nabla S(\cdot)u_0\|_{L^\gamma(\Omega ; {\mathcal W}{[T]})} \leq 2 C_{Str} A,
\end{equation}
where the last upper estimate is true for any $T>0$. 
Furthermore, since $S(t)$ is an $L^2_x$ isometry, 
we have $\|S(t) u_0\|_{H^1_x} \leq \|u_0\|_{H^1_x} \leq A$ a.s. for every $t>0$, which implies
\begin{equation}\label{IC-H1}
\| S(\cdot) u_0\|_{L^\gamma(\Omega ; L^\infty_{[0,T]} H^1_x)} \leq A, \quad \mbox{for ~any} ~~ T>0.
\end{equation}

(1) We bound the nonlinear term $T_1(u)$. Recall that \eqref{Sob} implies $\|T_1(u)\|_{\SS[T]} \leq C_{Sob} \| \nabla T_1(u)\|_{{\mathcal W}{[T]}}$. 
Furthermore, \eqref{Stri_int} implies that for any $L^2$-admissible pair $(q,r)$, we have
\[ 
\big\|\nabla T_1(u) \big\|_{L^\gamma_{[0,T]} L^\rho_x} \leq C_{Str} (2\sigma +1) \big\| \theta_a\big( \|u\|_{\SS[T]}\big) \theta_b\big( \| \nabla u\|_{\WW[T]}\big) 
|u(s)|^{2\sigma} \nabla u(s) \big\|_{L^{q'}_{[0,T]} L^{r'}_x}.
\]
Take $q=2$ and $r=\frac{2n}{n-2}$, so that $(q,r)$ is an $L^2$-admissible pair such that $r'=\frac{2n}{n+2}$, which implies $\frac{1}{r'} = \frac{2\sigma}{p} + \frac{1}{\rho}$
 (with $p$ and $\rho$ as in \eqref{gamma-rho}). 
Then H\"older's inequality implies that $\| |u|^{2\sigma} \nabla u\|_{L^{r'}_x} \leq \|u\|_{L^p_x}^{2\sigma}  \|\nabla u\|_{L^\rho_x}$. 

Note that $\frac{2\sigma +1}{\gamma} = \frac{1}{2}$. 
Since $\| \theta\|_{L^\infty} \leq 1$,  H\"older's inequality applied to the measure $dP\times dt$ on the set $\Omega \times [0,T]$  implies that for every $T>0$ 
\begin{equation}
 \| \nabla T_1(u) \|_{L^\gamma(\Omega ; {\mathcal W}{[T]} )}  \leq C_{Str}  \, \frac{n+2}{n-2} \,  \|u\|_{L^\gamma(\Omega; {\mathcal S}{[T]})}^{\frac{4}{n-2} } \; 
\|\nabla u\|_{L^\gamma(\Omega; {\mathcal W}{[T]})}, 		
\label{upp_nablaT1-W}
\end{equation}
and a similar computation implies
\begin{equation}\label{upper_T1-W}
\|  T_1(u) \|_{L^\gamma(\Omega ; {\mathcal W}{[T]} )}  \leq C_{Str}   \,  \|u\|_{L^\gamma(\Omega; {\mathcal S}{[T]})}^{\frac{4}{n-2} } \; 
\|  u\|_{L^\gamma(\Omega; {\mathcal W}{[T]})}.
\end{equation} 
Since $(\infty,2)$ is an $L^2$-admissible pair, for the same values of $r,q$ we also have
\[  
\|\nabla T_1(u)\|_{L^\infty_{[0,T]} L^2_x} \leq C_{Str} (2\sigma +1) \big\| \theta_a\big( \|u\|_{\SS[T] }\big) \theta_b\big( \nabla u\|_{\WW [T] }\big) 
|u(s)|^{2\sigma} \nabla u(s) \|_{L^{q'}_{[0,T]} L^{r'}_x}, 
\]
which yields  
$$ \| \nabla T_1(u) \|_{L^\gamma(\Omega ; L^\infty_{[0,T]} L^2_x)} \leq  C_{Str} \, \frac{n+2}{n-2} \,  \|u\|_{L^\gamma(\Omega; {\mathcal S}{[T]})}^{\frac{4}{n-2} } \, 
\|\nabla u\|_{L^\gamma(\Omega; {\mathcal W}{[T]})}.
$$

In the same manner, we also have 
$$  \| T_1(u) \|_{L^\gamma(\Omega ; L^\infty_{[0,T]} L^2_x)} \leq     C_{Str} \, 
\|u\|_{L^\gamma(\Omega; {\mathcal S}{[T]})}^{\frac{4}{n-2} } \| u\|_{L^\gamma(\Omega ; {\mathcal W}{[T]})}.
$$
Therefore,
\begin{equation}		\label{upper_T1_H1} 
\| T_1(u)\|_{L^\gamma(\Omega ; L^\infty_{[0,T]} H^1_x)} \leq    C_{Str} \frac{2n}{n-2} \|u\|_{L^\gamma(\Omega; {\mathcal S}{[T]} )}^{\frac{4}{n-2} }  
\big[ \|\nabla u\|_{L^\gamma(\Omega; {\mathcal W}{[T]})} + \|u\|_{L^\gamma(\Omega; {\mathcal W}{[T]})} \big].
\end{equation} 
\smallskip

(2) Our next task is to bound the stochastic integral $T_2(u)$. Similarly to the above, by \eqref{Sob},  we have
 $\| T_2(u)\|_{L^\gamma(\Omega ; \SS{[T]}) } \leq C_{Sob} \|\nabla T_2(u)\|_{L^\gamma(\Omega ; {\mathcal W}{[T]})}$.
We split $i \nabla T_2(u)= T_{2,1}(\nabla u) + T_{2,2}(u)$, where
\[ T_{2,1}(\nabla u) = \int_0^t S(t-s) \nabla u(s) dW(s) \quad \mbox{\rm and} \quad T_{2,2}(u) = \int_0^t S(t-s) u(s) d\nabla W(s).\]

Given a separable Hilbert space $H$, two Banach spaces  $E,F$,  and two  operators $K\in R(H;E)$ and
$L\in {\mathcal L}(E;F)$, recalling the proof of \cite[Lemma 2.1]{deB_Deb_CMP}, we obtain that the operator $LK \in R(H;F)$ and
\begin{equation} 	\label{Radon_composition}
\| LK\|_{R(H,F)} \leq \|L\|_{{\mathcal L}(E,F)} \cdot  \|K\|_{R(H;E)}. 
\end{equation} 

Given $t, t_0\in [0,T]$, we set
\[  {\mathcal I} \nabla u(t_0,t) = \int_0^{t_0} S(t-s) \nabla u(s) dW(s) = \sum_{k\geq 0} \int_0^{t_0} S(t-s) \big(\nabla u(s) \phi e_k\big)  d\beta_k(s).\]

Let $B$ be a martingale type-2 Banach space, and  
$X$ be a $R(L^2(\RR^n ;\RR); B)$-valued progressively measurable process. Recall 
that for any $p\in [2,\infty)$, we have
\begin{equation}    \label{BDG}
\EX\Big( \sup_{t\in [0,T]} \Big| \int_0^t X(s) dW(s) \Big|_B^p\Big) \leq C_p(B) \, \EX\Big( \Big| \int_0^T \|X(s)\|_{R(L^2(\RR^n;\RR); B)}^2 ds \Big|^{\frac{p}{2}}
\Big). 
\end{equation}
Then the Burkholder-Davis-Gundy inequality (e.g., see \cite{Brz_1997}) implies that
\[ \EX \Big( \sup_{t_0 \in [0,T]} \|  {\mathcal I}\nabla u(t_0,t)\|_{L^\rho_x}^\gamma \Big) \leq 
C_{\gamma}(\rho) \EX \Big( \Big| \int_0^T \|S(t-s) \big( \nabla u(s) \phi \big) \|^2_{R(L^2_x; L^\rho_x)} ds\Big|^{\frac{\gamma}{2}} \Big).
\] 
For $H=L^2_x$, $E= L^{n\gamma/2}_x$ and $F=L^\gamma_x$, if $L_{\nabla u}$ is the map from $E$ into $F$ defined by $L_{\nabla u}(v)=S(t-s) ( \nabla u(s) v)$, the upper bound
\eqref{Radon_composition} implies that
\[ \|S(t-s) \big( \nabla u(s) \phi) \|_{R(L^2_x; L^\rho_x)} \leq \|L_{\nabla u}\|_{{\mathcal L}(L_x^{n\gamma/2} ; L^\rho_x)} \|\phi\|_{R(L^2_x; L^{n\gamma/2}_x)}.\]
Let $v\in L^{n\gamma/2}(\mathbb R^n)$. Using the time decay of the Schr\"odinger group \eqref{decay}, we deduce
\[ \|L_{\nabla u} v\|_{L^\rho_x} = \| S(t-s) (\nabla u(s)  v)\|_{L^\rho_x} \leq C(\rho)  |t-s|^{-n ( \frac{1}{2}-\frac{1}{\rho})} \| \nabla u(s) v \|_{L^{\rho'}_x}.\]
Since $(\gamma, \rho)$ is an $L^2$-admissible pair, $n( \frac{1}{2}-\frac{1}{\rho}\big) = \frac{2}{\gamma}$, and since $\frac{1}{\rho'} = 1-\frac{1}{\rho} = \frac{1}{2}+ \frac{2}{n\gamma}$, we deduce
\[ \|L_{\nabla u} v\|_{L^\rho_x} \leq C(\rho) |t-s|^{-\frac{2}{\gamma}} \|\nabla u(s)\|_{L^2_x} \|v\|_{L^{n\gamma/2}_x}, \]
so that 
\[ \EX \Big( \sup_{t_0\in [0,T]} \Big| \int_0^{t_0} \!\! S(t-s) \big( \nabla u(s) dW(s)\big)\Big|^\gamma\Big) \leq 
C(\gamma,\rho)\EX \Big( \Big| \int_0^T \!\! (t-s)^{-\frac{4}{\gamma}} \| \nabla u(s)\|_{L^2_x}^2 ds \Big|^{\frac{\gamma}{2}} \Big)
\|\phi\|_{R(L^2_x; L^{n\gamma/2}_x)}^\gamma.\]
Since $\gamma = \frac{2(n+2)}{n-2}$, it implies that $\gamma >4$ for $n<6$. Therefore, 
\[ 
\EX \Big( \sup_{t_0\in [0,T]}  \Big| \int_0^{t_0} \!\! S(t-s) \big( \nabla u(s) dW(s) \big) \Big|^\gamma\Big) \leq 
C(\gamma, \rho) T^{\,\gamma (\frac{1}{2} - \frac{2}{\gamma})} \|\phi\|_{R(L^2_x; L^{n\gamma/2}_x)}^\gamma \EX \big( \|\nabla u\|_{L^\infty_{[0,T]} L^2_x}^\gamma \big).
\] 
Thus, if $J\nabla u(t):={\mathcal I}\nabla u(t,t)$, we deduce that for every $t\in [0,T]$ 
\[ \EX \big( \|J\nabla u(t)\|_{L^\rho_x}^\gamma\big)  \leq C(\gamma,\rho)\,  t^{\gamma (\frac{1}{2} - \frac{2}{\gamma})} \|\phi\|_{R(L^2_x; L^{n\gamma/2}_x)}^\gamma \EX\big( \|\nabla u\|_{L^\infty_{[0,T]} L^2_x}^\gamma \big).
\]
Using that $\frac{n\gamma}{2} = \frac{2\rho}{\rho-2}$,  we conclude 
\begin{align*} 
\|T_{2,1}(\nabla u)\|_{L^\gamma(\Omega ; {\mathcal W}{[T]})}^\gamma 
=  \EE\Big( \int_0^T  \|Ju(t)\|_{L^\rho_x}^\gamma dt\Big) 
\leq C(\gamma,\rho)  T^{1+\gamma ( \frac{1}{2} - \frac{2}{\gamma})} \|\phi\|_{R(L^2_x;L^{2\rho/(\rho-2)}_x)}^\gamma \EE\big(  \|\nabla u\|_{L^\infty_{[0,T]} L^2_x}^\gamma \big).
\end{align*}
Since $T_2(u)=-i T_{2,1}(u)$, a similar computation implies  
\begin{equation}		\label{upper_T21_1bis}
\|T_{2}( u)\|_{L^\gamma(\Omega ; {\mathcal W}{[T]})} 
\leq C(\gamma,\rho)  T^{\frac{1}{2}-\frac{1}{\gamma}} \|\phi\|_{R(L^2_x;L^{2\rho/(\rho-2)}_x)}  \| u\|_{L^\gamma(\Omega ; L^\infty_{[0,T]} L^2_x )}.
\end{equation} 
Using \cite[Theorem 6.10]{DaP-Zab}, we deduce
\[ \EX  \Big( \sup_{t\in [0,T]} \Big\| \int_0^t S(t-s) (\nabla u(s) dW(s))\Big\|_{L^2_x}^\gamma\Big) \leq C(\gamma) \EX  \Big(\Big| \int_0^T \|\nabla u(s) \phi \|_{L^{0,0}_2}^2 ds 
\Big|^{\frac{\gamma}{2}} \Big).\]
Since $\frac{1}{2}=\frac{1}{\rho} + \frac{2}{n\gamma}$, using again \eqref{Radon_composition} and H\"older's inequality for the time integral,  we deduce
\begin{align*}
\EX \Big( \sup_{t\in [0,T]} \|J u(t)\|_{L^2_x}^\gamma\Big) \leq& \;  C(\gamma)  \|\phi\|_{R(L^2_x; L^{n\gamma/2}_x)}^\gamma \EX \Big( \Big| \int_0^T \| \nabla u(s)\|_{L^\rho_x}^2 ds \Big|^{\frac{\gamma}{2}} \Big)\\
\leq& \; C(\gamma)
T^{\frac{\gamma}{2}-1}  \|\phi\|_{R(L^2_x; L^{n\gamma/2}_x)}^\gamma \EX \Big( \int_0^T \|\nabla u(s)\|_{L^\rho_x}^\gamma ds\Big),
\end{align*} 
which implies 
\begin{align}		
\|T_{2,1}(\nabla u)\|_{L^\gamma(\Omega ; L^\infty_{[0,T]} L^2_x)} 
\leq &\; C(\gamma)  T^{\frac{1}{2}-\frac{1}{\gamma} } \|\phi\|_{R(L^2_x;L^{2\rho/(\rho-2)}_x)}   \|\nabla u\|_{L^\gamma(\Omega; {\mathcal W}{[T]} )} \nonumber \\
\|T_{2}( u)\|_{L^\gamma(\Omega ; L^\infty_{[0,T]} L^2_x)} 
\leq &\; C(\gamma)  T^{\frac{1}{2}-\frac{1}{\gamma}}  \|\phi\|_{R(L^2_x;L^{2\rho/(\rho-2)}_x)}   \| u\|_{L^\gamma(\Omega; {\mathcal W}{[T]} )}.		\label{upper_T21_2bis}
\end{align} 
Since $T_{2,2}(u)$ is deduced from $T_{2,1} (\nabla u)$ by replacing $\nabla u$ with $u$ and $\phi$ with $\nabla \phi$,   we have 
\begin{align*}		
\|T_{2,2}( u)\|_{L^\gamma(\Omega ; {\mathcal W}{[T]})} \leq & \; C(\gamma,\rho)\,   T^{\frac{1}{2}-\frac{1}{\gamma}} \|\phi\|_{R(L^2_x;\dot{W}^{1, 2\rho/(\rho-2)}_x)} 
 \| u\|_{L^\gamma(\Omega ; L^\infty_{[0,T]} L^2_x)},
\\
\|T_{2,2}( u)\|_{L^\gamma(\Omega ; L^\infty_{[0,T]} L^2_x)} 
\leq &\; C(\gamma) \,   T^{\frac{1}{2}-\frac{1}{\gamma}} \|\phi\|_{R(L^2_x;\dot{W}^{1,2\rho/(\rho-2)}_x)} \| u\|_{L^\gamma(\Omega; {\mathcal W}{[T]} )}. 
\end{align*}
Collecting the estimates of $\| T_{2,1}(\nabla u)\|_{L^\gamma(\Omega , {\mathcal W}{[T]})}$ and  $\| T_{2,2}( u)\|_{L^\gamma(\Omega , {\mathcal W}{[T]})}$, we obtain
\begin{equation}		\label{upp_nablaT2}
\| \nabla T_2(u)\|_{L^\gamma(\Omega ; {\mathcal W}{[T]})} \leq \; C(\gamma,\rho) \|\phi\|_{R(L^2_x; W^{1, 2\rho/(\rho-2)}_x)} T^{\frac{1}{2}-\frac{1}{\gamma}} 
\|u\|_{L^\gamma(\Omega; L^\infty_{[0,T]} H^1_x)}.
\end{equation}
Similar estimates for $\| T_{2,1}(\nabla u)\|_{L^\gamma(\Omega ; L^\infty_{[0,T]} L^2_x)}$ and $\| T_{2,2}( u)\|_{L^\gamma(\Omega ; L^\infty_{[0,T]} L^2_x)}$ yield
\begin{equation}		\label{upp_nablaT2_2}
\| \nabla T_2(u)\|_{L^\gamma(\Omega ; L^\infty_{[0,T]} L^2_x)} \leq \; C(\gamma) \|\phi\|_{R(L^2_x; W^{1, 2\rho/(\rho-2)}_x)} T^{\frac{1}{2}-\frac{1}{\gamma}} 
\big[ \|u\|_{L^\gamma(\Omega; {\mathcal W}{[T]})} +  \|\nabla u\|_{L^\gamma(\Omega; {\mathcal W}{[T]})} \big].
\end{equation} 
\smallskip

(3) We finally bound the term $T_3(u)$. The 
embedding \eqref{Sob} implies that $\|T_3(u)\|_{L^p_x}$ $ \leq C_{Sob} \| \nabla T_3(u)\|_{L^\rho_x}$, so we proceed with bounding the last quantity.
The inequality \eqref{Stri_int} implies 
\[ \| \nabla T_3(t)\|_{{\mathcal W}{[T]}} \leq C_{Str} \| \nabla\big( u(s) F_\phi\big) \|_{L^{\gamma'}_{[0,T]} L^{\rho'}_x} \leq C_{Str} \big[ T_{3,1}(\nabla u) + T_{3,2}(u)\big],\]
where
\[ T_{3,1}(\nabla u) = \| \nabla u  F_\phi\|_{L^{\gamma'}_{[0,T]} L^{\rho'}_x} \quad \mbox{\rm and} \quad T_{3,2}(u) = \| u \nabla F_\phi \|_{L^{\gamma'}_{[0,T]} L^{\rho'}_x}.\]
Since $\frac{1}{\rho'} = \frac{1}{\rho} + \frac{\rho-2}{\rho}$ and $\frac{1}{\gamma'} = \frac{1}{\gamma} + \frac{\gamma-2}{\gamma}$, H\"older's inequality implies
\[ T_{3,1}(\nabla u) \leq \| \nabla u\|_{{\mathcal W}{[T]}} \|F_\phi\|_{L^{\rho/(\rho-2)}_x} T^{\frac{\gamma-2}{\gamma}} \quad \mbox{\rm and} \quad 
T_{3,2}(u) \leq \|u\|_{{\mathcal W}{[T]}} \|\nabla F_\phi\|_{L^{\rho/(\rho-2)}_x} T^{\frac{\gamma-2}{\gamma}}.
\]
Computations made in \cite[p.~173]{deB_Deb_CMP} imply the existence of a constant $\kappa>0$ such that  $\|F_\phi\|_{L^{\rho/(\rho-2)}_x} \leq \kappa \|\phi\|_{R(L^2_x ; L^{2\rho/(\rho-2)}_x)}^2$, 
and similarly, $\| \nabla F_\phi\|_{L^{\rho/(\rho-2)}} \leq \kappa \|\phi\|_{R(L^2_x; \dot{W}^{1, 2\rho/(\rho-2)}_x)}^2$. 
Finally, since $T_3(u)$ is similar to $T_{3,1}(\nabla u)$ (replacing $\nabla u$ by $u$),  we deduce that
\begin{align}		\label{upp_nablaT3_W}
 \|\nabla T_3(u)\|_{L^\gamma(\Omega ; {\mathcal W}_{[T]})}
 \leq&\,  \kappa C_{Str} T^{1-\frac{2}{\gamma}} \|\phi\|_{R(L^2_x; W^{1, 2\rho/(\rho-2)}_x)}^2 
\big[ \|u\|_{L^\gamma( \Omega ; {\mathcal W}{[T]})}  + \| \nabla u\|_{L^\gamma( \Omega ; {\mathcal W}{[T]} )}\big], \\
\|  T_3(u)\|_{L^\gamma(\Omega ; {\mathcal W}{[T]})}
 \leq &\, \kappa C_{Str} T^{1-\frac{2}{\gamma}} \|\phi\|_{R(L^2_x; L^{2\rho/(\rho-2)}_x)}^2 
 \|u\|_{L^\gamma( \Omega ; {\mathcal W}{[T]})} . 	\label{upp_T3_W}
\end{align} 
Since $(\infty,2)$ is an $L^2$-admissible pair, similar bounds are valid for $\| \nabla T_3(u)\|_{L^\infty_{[0,T]} L^2_x}$ and $\| T_3(u)\|_{L^\gamma (\Omega ;  L^\infty_{[0,T]} L^2_x)}$,
which imply
\begin{equation}		\label{upp_T3_H1}
\|T_3(u)\|_{L^\gamma(\Omega ; L^\infty_{[0,T]} H^1_x)} \leq \kappa C_{Str} \|\phi\|_{R(L^2_x ; W^{1, 2\rho/(\rho-2)})}^2 T^{1-\frac{2}{\gamma}} 
\big[2 \| u\|_{L^\gamma(\Omega ; {\mathcal W}{[T]})} +  \| \nabla u\|_{L^\gamma(\Omega ; {\mathcal W}{[T]})} \big].
\end{equation} 
\smallskip

Now that we obtained bounds for all terms in \eqref{E:all}, we show how to choose $a, b, c, \tau$ so that ${\mathcal G}$ maps ${\mathcal Y}_{a,b,c}(\tau)$ into itself. 

Let $u\in {\mathcal Y}_{a,b,c}(\tau)$ and $\tau \leq \tau(\delta)\wedge T$. Recall the definition of the parameters $a,b,c$ for norm bounds in the definition of spaces in \eqref{E:Yspace} and \eqref{E:Xspace}. Then
the inequalities \eqref{IC-H1}, \eqref{upper_T1_H1},  \eqref{upper_T21_2bis}, \eqref{upp_nablaT2_2} and \eqref{upp_T3_H1} imply
\begin{align*}
 \| {\mathcal G}(u)\|_{L^\gamma(\Omega ; L^\infty_{[0,\tau]} H^1_x)} \leq &\;  A + 2 C_{Str}\,  \frac{2n}{n-2}\,  a^{\frac{4}{n-2}}\,  b \\
&{+ 2\, b\,  T^{\frac{1}{2} -\frac{2}{\gamma}} \, 
\big[ C(\gamma) T^{\frac{1}{\gamma}} 
\| \phi\|_{R(L^2_x ; W^{1, 2\rho/(\rho-2)}_x)} 
+ \kappa C_{Str} T^{\frac{1}{2}} 
\|\phi\|_{R(L^2_x ; W^{1, 2\rho/(\rho-2)}_x)}^2\big]. }
\end{align*}
The inequalities in \eqref{IC-multi}, \eqref{upp_nablaT1-W}, \eqref{upper_T1-W}, \eqref{upper_T21_1bis}, \eqref{upp_nablaT2}, \eqref{upp_nablaT3_W} and \eqref{upp_T3_W}
imply
\begin{align*}
\| {\mathcal G}(u)\|_{L^\gamma(\Omega ; {\mathcal W}{[T]})} & + \| \nabla {\mathcal G}(u)\|_{L\gamma(\Omega ; {\mathcal W}{[T]})} \leq 
2  C_{Str} A + C_{Str} \frac{2n}{n-2} a^{\frac{4}{n-2}}  b \\
&{+ 2      T^{\frac{1}{2}-\frac{2}{\gamma}}  \big[ C(\gamma,\rho)\,     T^{\frac{1}{\gamma}}  \| \phi\|_{R(L^2_x ; W^{1, 2\rho/(\rho-2)}_x)} \; {c }
+ \kappa C_{Str}  \,  T^{\frac{1}{2}} \| \phi\|_{R(L^2_x ; W^{1, 2\rho/(\rho-2)}_x)}^2  \; b \big] . }
\end{align*}
Furthermore, by \eqref{Sob} we have $\| T_j(u)\|_{L^\gamma(\Omega ; {\mathcal S}{[T]})} \leq C_{Sob} \| \nabla T_j(u)\|_{L^\gamma(\Omega ; {\mathcal W}{[T]})}$, and thus, the upper estimates
\eqref{IC-multi}, \eqref{upp_nablaT1-W}, \eqref{upp_nablaT2} and \eqref{upp_nablaT3_W} imply
\begin{align*}
\| {\mathcal G}(u)\|_{L^\gamma(\Omega; {\mathcal S}{[\tau]})} \leq& \;  \delta + C_{Sob} C_{Str} \frac{n+2}{n-2} \, a^{\frac{4}{n-2}} \, b \\
&{+ C_{Sob} T^{\frac{1}{2}-\frac{2}{\gamma}}
\big[ C(\gamma, \rho)  , T^{\frac{1}{\gamma}} \|\phi\|_{R(L^2_x; W^{1, 2\rho/(\rho-2)}_x)}\; c   + \kappa C_{Str} T^{\frac{1}{2}} \|\phi\|_{R(L^2_x; W^{1, 2\rho/(\rho-2)}_x)}^2\;  b\big]. }
\end{align*} 

Set $c=3A$ and $b=6 C_{Str} A$. Choose  $a>0$ such that 
\[ 2C_{Str} \, \frac{2n}{n-2} a^{\frac{4}{n-2}}  \, b\leq A, \quad \frac{n}{n-2} a^{\frac{4}{n-2}} b \leq A \quad \mbox{\rm and} \quad C_{Sob} C_{Str} a^{\frac{6-n}{n-2}} \, b \leq \frac{1}{3}.
\] 

Let $\delta = \frac{a}{3}$ and choose $T>0$ such that 
{
\begin{align*}
 2b\, T^{\frac{1}{2}-\frac{2}{\gamma}} \big[ C(\gamma) T^{\frac{1}{\gamma}}   \|\phi\|_{R(L^2_x; W^{1, 2\rho/(\rho-2)}_x)} 
 + \kappa C_{Str} T^{\frac{1}{2}}   \|\phi\|_{R(L^2_x; W^{1, 2\rho/(\rho-2)}_x)}^2 \big] \leq &\; A,\\
 2\, T^{\frac{1}{2}-\frac{2}{\gamma}}  \big[ C(\gamma,\rho) T^{\frac{1}{\gamma}}  \|\phi\|_{R(L^2_x; W^{1, 2\rho/(\rho-2)}_x)} \; c 
  + \kappa C_{Str} T^{\frac{1}{2}} \|\phi\|_{R(L^2_x; W^{1, 2\rho/(\rho-2)}_x)}^2 \; b \big] \leq& \;  {2} \,C_{Str} A, \\
 C_{Sob} T^{\frac{1}{2}-\frac{2}{\gamma}}  \big[ C(\gamma,\rho)  T^{\frac{1}{\gamma}} \|\phi\|_{R(L^2_x; W^{1, 2\rho/(\rho-2)}_x)} \; c
  + \kappa C_{Str} T^{\frac{1}{2}} \|\phi\|_{R(L^2_x; W^{1, 2\rho/(\rho-2)}_x)}^2\; b  \big] \leq & \; \frac{a}{3}.
 \end{align*} }
Then for $\tau = \tau(\delta) \wedge T$, ${\mathcal G}$ maps ${\mathcal Y}_{a,b,c}(\tau)$ into itself. Note that for the same values of $b,c$, if $\tilde{a}\leq a$, and $\tilde{\delta}$, $\tilde{T}>0$ are chosen to satisfy the last three bounds with $\tilde{\delta}$ and $\tilde{a}$ instead of $\delta$ and $a$, respectively, then for $\tilde{\tau} = \tau(\tilde{\delta}) \wedge \tilde{T}$,
${\mathcal G}$ maps ${\mathcal Y}_{\tilde{a},b,c}(\tilde{\tau})$ into itself.
\smallskip

{\bf Step 2.}  We  give constraints on $a,b,c$ and $\tau$ so that ${\mathcal G}$ is a contraction on ${\mathcal Y}_{a,b,c}(\tau)$ into itself with the norm
\[ \|u\|_{{\mathcal Y}_{a,b,c}(\tau)} =  \|u\|_{L^\gamma(\Omega; L^\infty_{[0,\tau]} H^1_x)}+ \|u\|_{L^\gamma(\Omega ; {\mathcal S}{[\tau]})} + 
\|\nabla u\|_{L^\gamma(\Omega ; {\mathcal W}{[\tau]})} . 
\]

Let $u_1$ and $u_2$ belong to ${\mathcal Y}_{a,b,c}(\tau)$. Then ${\mathcal G}(u_1) - {\mathcal G}(u_2) = \sum_{j=1}^3 [T_j(u_1) - T_j(u_2)\big]$.
The Sobolev embedding \eqref{Sob} implies that for each $j=1, ...,3$, we have
\begin{equation}		\label{SS-WW}
\|T_j(u_1)- T_j(u_2)\|_{\SS{[\tau]}} \leq C_{Sob}\|\nabla \big[ T_j(u_1)- T_j(u_2)\big]\|_{{\mathcal W}{[\tau]}}. 
 \end{equation}

We first deal with the nonlinear term $T_1$. For $j=1,2$, define the stopping times $\tau_1$ and  $\tau_2$  related to the truncation functions $\theta_a$ and $\theta_b$, respectively, that is,  $\tau_j=\tau^{\SS,a}_j \wedge \tau^{\WW,b}_j$, $j=1,2$, where 
 \[ \tau^{\SS,a}_j = \inf\big\{ t\geq 0 : \|u_j\|_{\SS{[t]} }\geq 2a\big\} \wedge \tau \quad \mbox{\rm and} \quad  \tau^{\WW,b}_j = \inf\big\{ t\geq 0 : \|u_j\|_{\WW{[t]} }\vee \|\nabla u\|_{\WW{[t]}} \geq 2b\big \} 
 \wedge \tau.  \]
Without loss of generality, we may assume that $\tau_1 \leq \tau_2$. Note that the terms $T_1(u_1)$ and $T_1(u_2)$ involve time integration on the time interval $[0, \tau_2 ] = [0, \tau_1] \cup
(\tau_1, \tau_2]$. 

This yields  $ T_1(u_1) -   T_1(u_2)=\sum_{j=1}^5 T_{1,j}(u_1,u_2)$, where (recall that $2\sigma = \frac{4}{n-2}$), 
\begin{align*}
T_{1,1}(u_1,u_2)  = &\; \int_0^{\tau_1} S(t-s) {\mathcal T}_1(s) ds, \\
\mbox{\rm for} \;  {\mathcal T}_1(s) = &\; \big[ \theta_a\big(\|u_1\|_{\SS{[s]}} \big) - \theta_a\big(\|u_2\|_{\SS{[s]}}\big) \big] \theta_b(\big\|u_1\|_{\WW{[s]} }\big) 
 |u_1(s)|^{2\sigma} u_1(s) , \\
 T_{1,2}(u_1,u_2)=&\; \int_0^{\tau_1} S(t-s) {\mathcal T}_2(s) ds, \\
 \mbox{\rm for} \;  {\mathcal T}_2(s)=&\,   \theta_a\big(\|u_2\|_{\SS{[s]}} \big) \big[  \theta_b\big(\|u_1\|_{\WW{[s]}} \big) - \theta_b\big(\|u_2\|_{\WW{[s]}}\big) \big]
 |u_1(s)|^{2\sigma} u_1(s) , \\
T_{1,3}(u_1,u_2) = &\; \int_0^{\tau_1} S(t-s) {\mathcal T}_{3}(s) ds, \\
\mbox{\rm for} \; {\mathcal T}_{3}(s) =&\;   \theta_a\big(\|u_2\|_{\SS{[s]}}\big)   \theta_b\big(\|u_2\|_{\WW{[s]}} \big)
\big[ |u_1(s)|^{2\sigma} u_1(s) - |u_2(s)|^{2\sigma} u_2(s)\big] ds,\\
T_{1,4}(u_1,u_2) =&\; \int_{\tau_1}^{\tau_2} 1_{\big\{ \tau_1^{\SS,a}  \leq \tau_1^{\WW,b}\big\}}  S(t-s) {\mathcal T}_{4}(s)  ds, \\
\mbox{for} \; {\mathcal T}_{4} (s)=&\;  \theta_a\big(\|u_2\|_{\SS{[s]} }\big) \theta_b  \big( \|u_2\|_{\WW{[s]}  }\big)|u_2(s)|^{2\sigma} u_2(s), \\
T_{1,5}(u_1,u_2) =& \; \int_{\tau_1}^{\tau_2} 1_{\big\{ \tau_1^{\SS,a} > \tau_1^{\WW,b} \big\}}  S(t-s) {\mathcal T}_{4}(s)  ds.
\end{align*}
The Strichartz estimate \eqref{Stri_int} implies that for $r=\frac{2n}{n-2}$ (such that $(2,r)$ is an $L^2$-admissible pair), we have  
\begin{align*}	
\| \nabla T_{1,1}(u_1,u_2) \|_{L^\gamma_{[0,\tau]} L^{\rho}_x}  \leq & C_{Str} \| \nabla {\mathcal T}_{1}\|_{L^2_{[0, \tau_1]} L^{r'}_x} \\
\leq &  C_{Str} \frac{\|\theta'\|_{L^\infty}}{a} \|u_1-u_2\|_{\SS{[\tau_1]} }  \| \nabla \big( |u_1|^{2\sigma} u_1\big) \|_{L^2_{[0, \tau_1]} L^{r'}_x}\\
\leq & C_{Str}   \frac{\|\theta'\|_{L^\infty}}{a} (2\sigma +1) \|u_1-u_2\|_{\SS{[\tau_1]}} \| |u_1 |^{2\sigma }\nabla u_1 \|_{L^2_{[0, \tau_1]} L^{r'}_x}\\
\leq &  C_{Str}   \frac{\|\theta'\|_{L^\infty}}{a} (2\sigma +1) \|u_1-u_2\|_{\SS{[\tau_1]} }  \| u_1\|_{L^\gamma_{[0,\tau_1]} L^p_x}^{2\sigma} \| \nabla u_1\|_{L^\gamma_{[0, \tau_1]} L^\rho_x},
\end{align*}
where the last estimate is a consequence of H\"older's inequality and of the identities $\frac{1}{r'} = \frac{2\sigma}{p} + \frac{1}{\rho}$ and $\frac{1}{2} = \frac{2\sigma +1}{\gamma}$. 
By definition, we deduce that
\[ \| \nabla T_{1,1}(u_1,u_2)\|_{{\mathcal W}_{[\tau]}} \leq  C_{Str} \,  \|\theta'\|_{L^\infty}\,  \frac{n+2}{n-2} \,  2^{\frac{n+2}{n-2}}\,  a^{\frac{6-n}{n-2}} \, b \, \|u_1-u_2\|_{\SS_{[\tau_1]}} .
\] 
Since $(\infty,2)$ is an $L^2$-admissible pair, the same upper bound is valid for $\| \nabla T_{1,1}(u_1,u_2) \|_{L^\infty_{[0,\tau]} L^2_x}$. 

A similar computation shows that
\begin{align*}		
 \|  T_{1,1}(u_1,u_2)\|_{L^\gamma_{[0,\tau]} L^{\rho}_x} \leq  &\;
 C_{Str}   \frac{\|\theta'\|_{L^\infty}}{a}  \|u_1-u_2\|_{\SS{[\tau_1]}}  \| u_1\|_{L^\gamma_{[0,\tau_1]} L^p_x}^{2\sigma} \| u\|_{L^\gamma_{[0, \tau_1]} L^\rho_x} \\ 
 \leq &\; C_{Str}  \|\theta'\|_{L^\infty}  2^{\frac{n+2}{n-2} } a^{\frac{6-n}{n-2} } b \|u_1-u_2\|_{\SS{[\tau_1]}}, 
\end{align*}
and again that the same upper bound holds for $\|  T_{1,1}(u_1,u_2)\|_{L^\infty_{[0,\tau]} L^{2}_x}$. 
Therefore, we deduce
\begin{equation}		\label{upp_T11} 
\| T_{1,1}(u_1,u_2)\|_{{\mathcal Y}_{a,b,c}(\tau)} \leq C_{Str} \| \theta'\|_{\infty}  2^{\frac{n+2}{n-2}} 
  a^{\frac{6-n}{n-2} }\,  b  \Big[ \frac{n+2}{n-2} \big( 2 +  C_{Sob}\big) + 2 \Big]
 \|u_1-u_2\|_{L^\gamma(\Omega; \SS{[\tau_1]})}. 
\end{equation}
The upper bounds for $T_{1,2}(u_1,u_2)$ are obtained in a similar manner. Indeed,
\begin{align*}	
 \| \nabla T_{1,2}(u_1,u_2)\|_{L^\gamma_{[0,\tau]} L^{\rho}_x} \leq & \;  C_{Str} { \| \nabla {\mathcal T}_{2}\|_{L^2_{[0, \tau_1]} L^{r'}_x}  }\\
  \leq&  \; C_{Str} \frac{\|\theta'\|_{L^\infty}}{b}  \|u_1-u_2\|_{{\mathcal W}{[\tau_1]} } \| \nabla \big( |u_1|^{2\sigma} u_1\big) \|_{L^2_{[0, \tau_1]} L^{r'}_x}\\
\leq &\;  C_{Str}   \frac{\|\theta'\|_{L^\infty}}{b} (2\sigma +1) \|u_1-u_2\|_{{\mathcal W}{[\tau_1]}} \| |u_1 |^{2\sigma }\nabla u_1 \|_{L^2_{[0, \tau_1]} L^{r'}_x}\\
  \leq &\;  C_{Str}   \frac{\|\theta'\|_{L^\infty}}{b} (2\sigma +1) \|u_1-u_2\|_{{\mathcal W}{[\tau_1]} }  \| u_1\|_{L^\gamma_{{\mathcal S}{[\tau_1]} }}^{2\sigma} \| \nabla u_1\|_{{\mathcal W}{[\tau_1]}},
\end{align*}
which implies
\[ \| \nabla T_{1,2}(u_1,u_2)\|_{{\mathcal W}{[\tau]}} \leq C_{Str} \, \frac{n+2}{n-2} \,  \|\theta'\|_{L^\infty}  \, 2^{\frac{n+2}{n-2}}  \; a^{\frac{4}{n-2} }\;   \|u_1-u_2\|_{{\mathcal W}{[\tau_1]}}  ,
\] 
and the same estimates hold for $\| \nabla T_{1,2}(u_1,u_2)\|_{L^\infty_{[0,\tau_2]} L^{2}_x}$,
while similar computations imply
\begin{align*}
 \|T_{1,2}(u_1,u_2)\|_{L^\gamma_{[0,\tau]}L^\rho_x} +  \|T_{1,2}(u_1,u_2)\|_{L^\infty_{[0,\tau_1]}L^2_x}
 \leq &\; 2 C_{Str}   \frac{\|\theta'\|_{L^\infty}}{b} \|u_1-u_2\|_{{\mathcal W}{[\tau_1]}}  \| u_1\|_{L^\gamma_{[0,\tau_1]} L^p_x}^{2\sigma} \|  u_1\|_{L^\gamma_{[0, \tau_1]} L^\rho_x},
 \\
 \leq &\; 2 \,  C_{Str}  \,  \|\theta'\|_{L^\infty} \, 2^{\frac{n+2}{n-2}} \; a^{\frac{4}{n-2} } \;  \|u_1-u_2\|_{{\mathcal W}{[\tau_1]}} . 
 \end{align*}
Putting together estimates for $T_{1,2}$, we get
\begin{equation}	\label{upp_T12} 
\|T_{1,2}(u_1,u_2)\|_{{\mathcal Y}_{a,b,c}(\tau)}  \leq C_{Str} \|\theta'\|_{L^\infty} 2^{\frac{n+2}{n-2}} a^{\frac{4}{n-2}  }  \Big[ \frac{n+2}{n-2}  (2+C_{Sob})  + 2 \Big]
 \|u_1-u_2\|_{L^\gamma(\Omega; {\mathcal W}{[\tau_1]})}.  
\end{equation}
Using \cite[p. 648]{KenMer}, since $2\sigma -1 = \frac{6-n}{n-2}>0$,  we have
\begin{align}
\big| |u_1|^{2\sigma} u_1 - |u_2|^{2\sigma} u_2\big| \leq & (2\sigma +1) \big[ |u_1|^{2\sigma} + |u_2|^{2\sigma} \big] |u_1-u_2|, \label{E:20a}\\
| \nabla ( |u_1|^{2\sigma} u_1) - |u_2|^{2\sigma} u_2)| \leq &(2\sigma +1)\big[  |u_1|^{2\sigma} |\nabla(u_1-u_2)| 
+ 2\sigma |u_1-u_2| (|u_1|^{2\sigma -1} + |u_2|^{2\sigma -1}) \nabla u_2  \big]. \label{E:20b}
\end{align}
Then \eqref{Stri_int}, \eqref{E:20a}, \eqref{E:20b} and H\"older's inequality based on $\frac{1}{r'} = \frac{2\sigma}{p} + \frac{1}{\rho}$ and $\frac{1}{2} = \frac{2\sigma +1}{\rho} $
imply
\begin{align*}
\|\nabla T_{1,3}(u_1,u_2)\|_{L^\gamma_{[0,\tau]} L^\rho_x} \leq &  C_{Str} \| \nabla (|u_1|^{2\sigma} u_1 - |u_2|^{2\sigma} u_2)\|_{L^2_{[0, \tau_1]} L^{r'}_x}\\
\leq & C_{Str} (2\sigma +1) \Big[  \|u_1\|_{L^\gamma_{[0, \tau_1]} L^p_x}^{2\sigma} \|\nabla(u_1-u_2)\|_{L^\gamma_{[0, T]}L^\rho_x} \\
&\quad 
+2\sigma \big( \|u_1\|_{L^\gamma_{[0, \tau_1]}L^p_x}^{2\sigma -1} +  \|u_2\|_{L^\gamma_{[0, \tau_1]}L^p_x}^{2\sigma -1} \big) \|\nabla u_2\|_{L^\gamma_{[0, \tau_1]}L^\rho_x} 
\|u_1-u_2\|_{L^\gamma_{[0,T]}L^p_x}\Big]\\
\leq & C_{Str} 2^{\frac{4}{n-2} }   \frac{n+2}{n-2} \Big[ a^{\frac{4}{n-2} }\,  \|\nabla( u_1-u_2)\|_{{\mathcal W}{[\tau_1]}}  + \frac{8}{n-2}\,   a^{\frac{6-n}{n-2}} \, b \|u_1-u_2\|_{{\mathcal S}{[\tau_1]}}\Big]. 
\end{align*}
The same bound is valid for $ \|\nabla T_{1,3}(u_1,u_2)\|_{L^\infty_{[0,\tau]}L^2_x}$. 
A similar computation implies
\begin{align*}
\| T_{1,3}(u_1,u_2)\|_{L^\gamma_{[0,T]} L^\rho_x}& + \|T_{1,3}(u_1,u_2)\|_{L^\infty_{[0,T]} L^2_x}\leq \, 2 C_{Str} \| |u_1|^{2\sigma} u_1 - |u_2|^{2\sigma} u_2\|_{L^2_{[0, \tau_1]} L^{r'}_x} \\
\leq & \, 2\,  C_{Str} (2\sigma +1) \big( \|u_1\|_{L^\gamma_{[0, \tau_1]} L^p_x}^{2\sigma} + \|u_2\|_{L^\gamma_{[0, \tau_1]} L^p_x}^{2\sigma}\big) \|u_1-u_2\|_{L^\gamma_{[0,\tau_1]} L^\rho_x} \\
\leq &\,  4  \, C_{Str} \, \frac{n+2}{n-2} \,  2^{\frac{4}{n-2} }\,  a^{\frac{4}{n-2} } \|u_1-u_2\|_{{\mathcal W}{[\tau_1]}}.
\end{align*}
Similar bounds hold for $\| T_{1,3}(u_1,u_2)\|_{L^\infty_{[0, \tau]} L^2_x}$. 
Putting together the estimates for $T_{1,2}$ we get
\begin{align} 	\label{upp_T13}
\| T_{1,3}(u_1,u_2)\|_{{\mathcal Y}_{a,b,c}(\tau)} & \leq C_{Str}  2^{\frac{4}{n-2} } { \frac{n+2}{n-2} \Big[ } a^{\frac{4}{n-2} } (2+C_{Sob})
 \| \nabla (u_1-u_2)\|_{L^\gamma(\Omega ; {\mathcal  W}{[\tau_1]})}  \nonumber \\
& +{2  
a^{\frac{4}{n-2}} } \, \| u_1-u_2\|_{L^\gamma(\Omega ; {\mathcal W}{[\tau_1]} )}+ \frac{8}{n-2} \, (2+C_{Sob})\,  a^{\frac{6-n}{n-2}} \, b\, 
 \|u_1-u_2\|_{L^\gamma(\Omega ; {\mathcal S}{[\tau_1]})}.
\end{align}

Next, on the set $\{ \tau_1= \tau^{\SS,a}_1 \leq \tau^{\WW ,b}_1\}$, we have $\theta_a(\|u_1\|_{{\mathcal S}_{[s]}}) =0$ for $s\in (\tau_1, \tau_2]$. Hence, 
\[ {\mathcal T}_{4}(s) = \big[  \theta_a\big(\|u_2\|_{{\mathcal S}_{[s]}}\big) - \theta_a\big(\|u_1\|_{{\mathcal S}_{[s]}}\big)  \big]   
 \theta_b \big( \|u_2\|_{\WW_{[s]}}\big)\; |u_2(s)|^{2\sigma} u_2(s). 
\]
Therefore, the proof of \eqref{upp_T11} implies
\begin{equation}		\label{upp_T14}
\| T_{1,4}(u_1,u_2)\|_{{\mathcal Y}_{a,b,c}(\tau)} \leq C_{Str} \, \|\theta'\|_{L^\infty} \, 2^{\frac{n+2}{n-2} } \, a^{\frac{6-n}{n-2} }\, b \,  \Big[ \frac{n+2}{n-2} \, (2+  C_{Sob}) +2 \Big]
 \|u_1-u_2\|_{L^\gamma(\Omega; {\mathcal S}{[\tau_2]})}. 
\end{equation} 
Similarly, on the set $\{\tau_1=\tau^{\WW,b}_1 < \tau^{\SS,a}_1 \}$, we have  $\theta_b(\|u_1\|_{\WW{[s]}}) =0$ for $s\in (\tau_1, \tau_2]$. Hence, the proof of 
\eqref{upp_T12} implies 
\begin{equation}		\label{upp_T15}
\| T_{1,5}(u_1,u_2)\|_{{\mathcal Y}_{a,b,c}(\tau)} \leq  
C_{Str} \|\theta'\|_{L^\infty} 2^{\frac{n+2}{n-2}} a^{\frac{4}{n-2}  }  \Big[ \frac{n+2}{n-2}  (2+C_{Sob})  + 2 \Big]
 \|u_1-u_2\|_{L^\gamma(\Omega; {\mathcal W}{[\tau_2]})}.  
 \end{equation}

Finally, observe that  for $j=2,3$, $T_j(u_1)-T_j(u_2) = T_j(u_1-u_2)$. Therefore, the upper estimates \eqref{Sob} and  \eqref{upper_T21_1bis}--\eqref{upp_nablaT2_2} imply for $\tau = \tau(\delta)
\wedge T$ 
\begin{align} 	\label{diff_T2}
\| T_2(u_1)-T_2(u_2)& \|_{{\mathcal Y}_{a,b,c}(\tau)} \leq  \|\phi\|_{R(L^2_x ; W^{1, 2\rho/(\rho-2)}_x)} 
T^{\frac{1}{2}-\frac{1}{\gamma}} \Big[ C(\gamma,\rho) (2+C_{Sob})  \|u_1-u_2\|_{L^\gamma(\Omega ; L^\infty_{[0,T]} H^1_x)}  \nonumber \\
& +
C(\gamma) \big( (2 + C_{Sob}) \|u_1-u_2\|_{L^\gamma(\Omega ; {\mathcal W}{[T]})}  + (1+C_{Sob})
 \|\nabla (u_1-u_2)\|_{L^\gamma(\Omega ; {\mathcal W}{[T]})}\big) \Big] . 
\end{align}
Similarly, \eqref{Sob} and  \eqref{upp_nablaT3_W}-- \eqref{upp_T3_H1} imply 
\begin{align}	\label{upp_diff_T3}
\| T_3(u_1)- T_3(u_2) \|_{{\mathcal Y}_{a,b,c}(\tau)}  \leq  &\;  \kappa C_{Str} \|\phi\|_{R(L^2_x ; W^{1,2\rho/(\rho-2)}_x)}^2   T^{1-\frac{2}{\gamma}}  
\big[ (4+C_{Sob}) \|u_1-u_2\|_{L^\gamma (\Omega ; {\mathcal W}{[T]} )} \nonumber \\
&\qquad\qquad + (2+C_{Sob}) \| \nabla (u_1-u_2)\|_{L^\gamma(\Omega ; {\mathcal W}{[T]})} \big] .
\end{align}

Collecting the estimates \eqref{upp_T11} --  \eqref{upp_diff_T3}, we obtain
\begin{align}		\label{diff-G}
\| {\mathcal G}(u_1)& - {\mathcal G}(u_2) \|_{{\mathcal Y}_{a,b,c}(\tau)} \leq  \| u_1-u_2\|_{L^\gamma(\Omega ; L^\infty_{[0,\tau]} H^1_x)}
\|\phi\|_{R(L^2_x; W^{1,2\rho/(\rho-2)}_x )} T^{\frac{1}{2}-\frac{1}{\gamma}} C(\gamma,\rho) (2+C_{Sob})  \nonumber \\
&+  \| u_1-u_2\|_{L^\gamma(\Omega ; {\mathcal W}{[\tau]})}  \Big\{  2 \, C_{Str}\, 2^{\frac{n+2}{n-2}} \, a^{\frac{4}{n-2}} \Big(   \|\theta'\|_{L^\infty}  \,  \Big[ \frac{n+2}{n-2} (2+C_{Sob})+2\Big]  + 
 \frac{n+2}{n-2}\,  \Big)  \nonumber \\
&\qquad \quad 
+ T^{\frac{1}{2}-\frac{2}{\gamma}} \Big( (2+C_{Sob})  T^{\frac{1}{\gamma}}  C(\gamma) \|\phi\|_{R(L^2_x ; W^{1,2\rho/(\rho-2)})} + \kappa C_{Str} (4+C_{Sob}) 
T^{\frac{1}{2}}  \|\phi\|_{R(L^2_x ; W^{1,2\rho/(\rho-2)})}^2 \Big) \Big\} 
\nonumber \\
&+ \| \nabla(u_1-u_2)\|_{L^\gamma(\Omega ; {\mathcal W}{[\tau]})} \Big\{ C_{Str} (2+C_{Sob}) \,  2^{\frac{4}{n-2}} \, \frac{n+2}{n-2} \; a^{\frac{4}{n-2}} \nonumber \\
& \qquad \quad + T^{\frac{1}{2}-\frac{2}{\gamma}} \Big( (2+C_{Sob}) T^{\frac{1}{\gamma}}   C(\gamma)  \|\phi\|_{R(L^2_x ; W^{1,2\rho/(\rho-2)})} 
+ \kappa C_{Str} (2+C_{Sob}) T^{\frac{1}{2}}  \|\phi\|_{R(L^2_x ; W^{1,2\rho/(\rho-2)})}^2\Big) \Big\} 
\nonumber \\
&+ \|u_1-u_2\|_{L^\gamma(\Omega ; {\mathcal S}{[\tau]} )} \, a^{\frac{6-n}{n-2}} \, b\, \Big\{ 2 C_{Str} \|\theta'\|_{L^\infty} 2^{\frac{n+2}{n-2}} \Big[ \frac{n+2}{n-2} (2+C_{Sob}) +2\Big] \nonumber \\
&\qquad \quad  + C_{Str} 2^{\frac{4}{n-2}} \frac{8}{n-2} \, (2+C_{Sob})\Big\}. 
\end{align} 

Set $c=3A$ and $b=6 C_{Str}A$. Let $a,\delta$ and $T$ be defined as in Step 1. Let $\tilde{a}\leq a$ satisfy the following inequalities:
\begin{align*}
 \tilde{a}^{\frac{4}{n-2}} & \; 2\, C_{Str} \, 2^{\frac{n+2}{n-2}} \, \Big( \| \theta'\|_{L^\infty}\, \Big[ \frac{n+2}{n-2} (2+C_{Sob}) +2\Big] + \frac{n+2}{n-2}\Big) < \frac{1}{2}, \\
\tilde{a}^{\frac{4}{n-2}} & \;  C_{Str} (2+C_{Sob}) \, 2^{\frac{4}{n-2}}\, \frac{n+2}{n-2} < \frac{1}{2}, \\
\tilde{a}^{\frac{6-n}{n-2}} \, b &\; \Big\{ 2 C_{Str} \|\theta'\|_{L^\infty} 2^{\frac{n+2}{n-2}} \Big[ \frac{n+2}{n-2} (2+C_{Sob}) +2\Big] 
+ C_{Str} 2^{\frac{4}{n-2}} \frac{8}{n-2} \, (2+C_{Sob})\Big\}<1.
\end{align*} 

Take $\tilde{\delta} = \frac{\tilde{a}}{3}$, and let $\tilde{T}\leq T$ satisfy the inequalities:
\begin{align*}
\tilde{T}^{\frac{1}{2}-\frac{1}{\gamma}} & \; C(\gamma,\rho) (2+C_{Sob})  \|\phi\|_{R(L^2_x ; W^{1,2\rho/(\rho-2)})}  <1, \\
\tilde{T}^{\frac{1}{2}-\frac{2}{\gamma}}  & \;  \big[ (2+C_{Sob} T^{\frac{1}{\gamma}}  C(\gamma) \|\phi\|_{R(L^2_x ; W^{1,2\rho/(\rho-2)})} 
 + \kappa C_{Str} (4+C_{Sob}) T^{\frac{1}{2}} \|\phi\|_{R(L^2_x ; W^{1,2\rho/(\rho-2)})}^2\big] < \frac{1}{2}. 
\end{align*}
Then, replacing $a,T$ by $\tilde{a}$ and $\tilde{T}$,  all coefficients of the various norms of $u_1-u_2$ in the right-hand side of \eqref{diff-G} are strictly less than 1, which implies
that ${\mathcal G}$ is a contraction of ${\mathcal Y}_{\tilde{a},b,c}(\tilde{\tau})$, where $\tilde{\tau} = \tau(\tilde{\delta}) \wedge \tilde{T}$. 

Therefore, the truncated equation \eqref{eq_trunc} with $\tilde{a}$ instead of $a$ has a unique solution $v$ on the random time interval $[0, \tilde{\tau}]$. 
Set 
\[ \tilde{\tau}(u_0) = \inf\{ s\geq 0 \, : \, \|v(\cdot)\|_{{\mathcal S}{[s]}} \geq \tilde{a}\} \wedge \inf\{ s\geq 0\, :\, \|v(\cdot)\|_{{\mathcal W}{[s]}} \geq b\} \wedge \tilde{\tau}.\]
Then $\tilde{\tau}(u_0)>0$ a.s., 
and $v$ is the unique solution to \eqref{NLS_Ito_critical}.

Finally, let $\tau(u_0)$ denote the supremum of all stopping times such that there exists a unique solution to \eqref{NLS_Ito_critical} on the time interval $[0, \tilde{\tau}(u_0))$,
 with the required properties, thus, completing the proof.
\end{proof}

\section{Maximal existence time in the focusing case} \label{S:time}
In this section we consider the focusing ($\lambda = 1$) case in \eqref{E:NLS} and recalling the local existence time 
$\tau(u_0)$ in each case of the noise from Theorems \ref{th_lwp-add} and \ref{th_lwp_mul}, we prove bounds on $P(\tau(u_0) >T)$ for some given time $T>0$ and give estimates on $\EE(\tau(u_0))$, provided that the energy and kinetic energy
of the initial condition are bounded by similar quantities of the ground state $Q$ (from \eqref{E:Q}) as in the deterministic case of Theorem \ref{T:main-deter} (i).

From the explicit formula \eqref{E:Q} it easily follows that $Q$ is positive, decaying polynomially fast (and does not belong to $L^2(\RR^n)$ in dimensions $n=3,4$, but it does in dimension $5$), $Q\in \dot{H}^1(\RR^n)$, and 
\begin{equation} \label{nabla-H-Q}
   \|\nabla Q\|_{L^2(\RR^n)} = C_n^{\,-{n}/{2}}, \quad 
   \quad H(Q)= n^{-1} C_n^{-n}.
\end{equation}

We now review the energy trapping argument, similar to that proved in the deterministic case (see \cite[Theorem~3.9]{KenMer}).
Recalling the energy from \eqref{H} and the Sobolev inequality \eqref{E:Sobolev}, we bound the energy as
$$
H(u) \leq \frac12 \|\nabla u\|_{L^2(\RR^n)}^2  - \frac{n-2}{2n}  C_n^{\frac{2n}{n-2}} \|\nabla u\|^{\frac{2n}{n-2} }_{L^2(\RR^n)}.
$$
Solving for the gradient, we have 
\begin{align} \label{H-nabla}
\|\nabla u\|_{L^2(\RR^n)}^2  
    \leq     2 H(u) +     \frac{n-2}{n} C_n^{\frac{2n}{n-2} } \|\nabla u\|_{L^2(\RR^n)}^{\frac{2n}{n-2}}. 
\end{align}
Thinking of $\|\nabla u\|^2_{L^2(\RR^n)}$ as a variable $x$, for $x\geq 0$, we set 
\begin{equation} 		\label{def-f} 
f(x) = \frac{1}{2}x - \frac{n-2}{2n}  C_n^{\frac{2n}{n-2} } x^{\frac{n}{n-2} }. 
\end{equation} 
Then $f'(x)=0$ if and only if $x=0$ or $x=C_n^{-n}$, which by \eqref{nabla-H-Q} can be denoted as $x = x_c:= \|\nabla Q\|_{L^2(\RR^n)}^2$. 
Moreover, $f$ is strictly increasing on the interval $[0, x_c]$ and decreasing on $[x_c,\infty)$. Substituting $x_c$ into the function $f$, we get $f(x_c)= n^{-1} C_n^{-n} = H(Q)$, by \eqref{nabla-H-Q}. 
From \eqref{H-nabla} we have $f(\|\nabla u\|_{L^2(\RR^n)}^2) \leq H(u)$. 

Let $u$ be the solution to \eqref{E:SNLS-m}  or \eqref{NLS-additive}. 
Suppose  that $\|\nabla u_0\|_{L^2(\RR^n)}^2 < x_c$ a.s.
and that for some $\beta \in (0,1)$ we have either 
$H(u_0) = \beta  f(x_c)$ if $u_0$ is deterministic or $H(u_0)\leq \beta H(Q)$ a.s. if $u_0$ is random. 

For $\delta \in (\beta,1)$, let $ \tau_\delta = \inf\{ s\geq 0\; : \; H(u(s)) \geq \delta H(Q)\} \wedge \tau(u_0)$.
Then by a.s. continuity of the random variable $\|\nabla u(\cdot)\|_{L^2(\RR^n)}^2$ on the (random) time interval $[0, \tau_\delta]$,
we deduce that 
\begin{equation}  \label{trapping}
\| \nabla u(s)\|_{L^2(\RR^n)} \leq \sqrt{x_c} = \|\nabla Q\|_{L^2(\RR^n)} \quad \mbox{\rm a.s. for} \quad s\in [0, \tau_\delta]. 
\end{equation}

We now proceed with each type of noise separately.

\subsection{Additive noise}
We suppose that the conditions of Theorem \ref{th_lwp-add} are satisfied, and we find constraints on $u_0$ and 
$T$ such that $P(\tau(u_0) >T) >0$, we also estimate a lower bound on $\EE\big(\tau(u_0)\big)$. 
We recall the upper bounds for the Hamiltonian derived in the energy-subcritical setting in \cite[(3.10)]{deB_Deb_H1} and \cite[(3.31)]{MilRou}. 
The same argument applies in the energy-critical setting to obtain a similar bound 
for $t < \tau(u_0)$, i.e.,
\begin{align}\label{H-add}
    H(u(t)) \leq \; 
    H(u_0) + & \frac{1}{2} \tau \|\phi\|_{L_2^{0,1}}^2  + \mbox{\rm Im } \Big(\sum_{k\geq 0} \int_0^\tau \!\! \int_{\RR^n} 
\nabla \overline{u(s,x)} \nabla (\phi e_k)(x) dx d\beta_k(s) \Big)\nonumber \\
&- \mbox{\rm Im } \Big( \sum_{k\geq 0}  \int_0^\tau \!\! \int_{\RR^n}  
| u(s,x)|^{\frac{4}{n-2}} \overline{u(s,x)} (\phi e_k)(x) dx d\beta_k(s) \Big). 
\end{align}

The following lemma describes upper bounds on the expected value of the energy uniformly on the time interval $[0,\tau]$.

\begin{lemma}\label{lem_Hadd}
Let the assumptions of Theorem \ref{th_lwp-add} be satisfied. Then for every stopping time $\tau \in \big(0, \tau(u_0)\big)$ we have 
\begin{align}		\label{E_sup_H_add}
\EX \Big( \sup_{s\leq \tau} H\big(u(s)\big)\Big) \leq &\; \EX \big(H(u_0) \big)
+ \frac{1}{2} \|\phi\|_{L^{0,1}_2}^2 \EX (\tau) + 3 \|\phi\|_{L^{0,1}_2} \EX \Big(
\sqrt{\tau}  \sup_{s\leq \tau} 
\|\nabla u(s)\|_{L^2_x}  \Big) \nonumber \\
&\; + 3 C_n^{\frac{2n}{n-2}} \| \phi\|_{L_2^{0,1} }\EX \Big( \sqrt{\tau} \sup_{s\leq \tau} 
 \|\nabla u(s)\|_{L^2_x}^{\frac{n+2}{n-2}} \Big),
\end{align}
where $C_n$ is defined in \eqref{nabla-H-Q}.
\end{lemma}
\begin{proof}
Let $\tau<  \tau(u_0)$  be a stopping time. 
Using the identity \eqref{H-add}, the Davis inequality, the local property of stochastic integrals and the Cauchy-Schwarz and H\"older's inequalities, we obtain
\begin{align*}
\EX \Big( \sup_{s\leq \tau}  H\big(u(s)\big)\Big) &\leq \;\EX \big( H(u_0) \big) + \frac{1}{2}  \|\phi\|_{L^{0,1}_2}^2 \EX (\tau)  
+ 3\EX \Big( \Big\{ \int_0^\tau  \sum_{k\geq 0} \Big( \int_{\RR^n}\! \! |\nabla u(s,x)| 
 |\nabla (\phi e_k)(x)| dx \Big)^2 ds \Big\}^{\frac{1}{2}} \Big) \\
&\hspace{3cm}  + 3 \EX  \Big( \Big\{\int_0^\tau   \sum_{k\geq 0} \Big( \int_{\RR^n} |u(s,x)|^{\frac{n+2}{n-2} } |\phi e_k(x)| \, dx\Big)^2 ds 
 \Big\}^{\frac{1}{2}} \Big) 	\label{E_sup_H_add_v1} \\
& \leq \; \EX \big( H(u_0) \big) + \frac{1}{2}  \|\phi\|_{L^{0,1}_2}^2 \EX (\tau)
 +3  \|\phi\|_{L_2^{0,1}} \EX \Big(   \sqrt{\tau} \sup_{s\leq \tau } \|\nabla u(s)\|_{L^2_x}  \Big)     \\
&\hspace{2cm} +3 \EX \Big( \Big\{ \int_0^\tau  \sum_{k\geq 0}  \Big( \int_{\RR^n} 
|u(s,x)|^{\frac{2n}{n-2}} dx \Big)^{\frac{n+2}{n}} 
 \Big( \int_{\RR^n} |\phi e_k(x)|^{\frac{2n}{n-2}} dx \Big)^{\frac{n-2}{n}} ds \Big\}^{\frac{1}{2}} \Big). 
\end{align*}
Using \eqref{E:Sobolev}, we deduce
\[ \sum_{k\geq 0} \|\phi e_k \|_{L_x^{\frac{2n}{n-2}}}^2 \leq C_n^2 \|\phi\|_{L_2^{0,1}}^2 \quad 
\mbox{\rm and} \quad   \|u(s)\|_{L_x^{\frac{2n}{n-2}}} \leq C_n \|\nabla u(s)\|_{L^2_x}. \]
Therefore,
\begin{align*}
    3 \EX \Big( \Big\{ \int_0^\tau &\sum_{k\geq 0}  \Big( \int_{\RR^n} 
|u(s,x)|^{\frac{2n}{n-2} } dx \Big)^{\frac{n+2}{n} }
 \Big( \int_{\RR^n} |\phi e_k(x)|^{\frac{2n}{n-2}} dx \Big)^{\frac{n-2}{n}} ds \Big\}^{\frac{1}{2}} \Big) \\
 & \leq 3 C_n^{\frac{2n}{n-2}} \|\phi\|_{L^{0,1}_2} \EE\Big( \sqrt{\tau} \sup_{s\leq \tau } \|\nabla u(s)\|_{L^2_x}^{\frac{n+2}{n-2}}\Big),
\end{align*}
which completes the proof of this lemma. 
\end{proof} 

\begin{Th}\label{T*-add-genIC}
Let $u_0$ be an $\dot{H}^1(\RR^n)$-valued random variable   such that  
$\|\nabla u_0\|_{L^2_x} <  \|\nabla Q\|_{L^2_x}$ a.s. and $H(u_0)\leq \beta H(Q)$ a.s.  for some
constant $\beta \in  (0,1)$. 
Suppose that $\phi \in L^{0,1}_2$. 
Let $u(t)$ be the solution to  \eqref{NLS-additive}. 

Then 
$P(\tau(u_0)>T) >0$ for every $T<T^*$, where
\begin{equation}        \label{T*-gere}
T^* = \frac{ 36 }{C_n^n \, \|\phi\|_{L^{0,1}_2}^2 
} \Big( \big(1+\tfrac{1}{18n}(1 - \beta)\big)^{\frac12} -1\Big)^2 .
\end{equation}

Furthermore,  
\[
\EE\big(\tau(u_0)  \big) \geq 
\frac{ 36 }{C_n^n \, \|\phi\|_{L^{0,1}_2}^2 } \Big( \big(1+\tfrac{1}{18n}(1 - \beta)\big)^{\frac12} -1\Big)^2.
\]
\end{Th}
\begin{proof}
Let $\delta \in (\beta,1)$ and recalling the stopping time $\tau(u_0)$ from the local theory (Theorem \ref{th_lwp-add})  set 
\[ \tau_\delta = \inf\{ t\geq 0 : H(u(t)) \geq \delta H(Q) \} \wedge {\tau(u_0)} .\]
Then for $T>0$, using \eqref{E_sup_H_add}, the fact that $\phi \in {L^{0,1}_2}$ and the Markov inequality, we deduce 
\begin{align*}
P( \tau_\delta \leq  T)&  = P\Big( \sup_{s\leq  T\wedge \tau_\delta} H(u(s)) \geq \delta H(Q) \Big)\\
\leq &\;  \frac{\beta}{\delta} +   \frac{1}{2 \delta H(Q) } \|\phi\|_{L_2^{0,1}}^2 T  
+ \frac{3}{\delta H(Q)} \sqrt{T} \|\phi\|_{L_2^{0,1}} \EE\Big( \sup_{s\leq  \tau_\delta }
 \|\nabla u(s)\|_{L^2_x} + C_n^{\frac{2n}{n-2}} \sup_{s\leq \tau_\delta } \|\nabla u(s)\|_{L^2_x}^{\frac{n+2}{n-2}}\Big) .
\end{align*}
Given $\gamma \in (\frac{\beta}{\delta},1)$ we look for a condition on $T$ to ensure that $P(\tau_\delta \leq T) \leq \gamma$. 
The trapping property \eqref{trapping} shows that for $s\in [0,\tau_\delta]$ we have $\|\nabla u(s)\|_{L^2_x} \leq \|\nabla Q\|_{L^2_x}$.
Therefore, we strengthen the requirement to read
\[ \frac{\beta}{\delta} + \frac{1}{2\delta H(Q)} \|\phi\|_{L^{0,1}_2}^2 T + \frac{3}{\delta H(Q)} \sqrt{T} \|\phi\|_{L^{0,1}_2}
\Big( \|\nabla Q\|_{L^2_x} + C_n^{\frac{2n}{n-2} }
\|\nabla Q\|_{L^2_x}^{\frac{n+2}{n-2}}\Big) \leq \gamma. 
\]

Set $X=\sqrt{T} \|\phi\|_{L^{0,1}_2}$. Using the identities in \eqref{nabla-H-Q}, we deduce the inequality
$$ 
\qquad a X^2 +b X -c \leq 0, \quad
\mbox{where} \quad
a=\tfrac{n}{2\delta} C_n^n, ~~b=\tfrac{6n}{\delta} C_n^{\frac{n}{2}}, ~~c=\gamma - \tfrac{\beta}{\delta}.
$$
The discriminant of the second degree polynomial on the left is $D=b^2+4ac$ and the inequality is thus satisfied if $X\in [0, X_2(\delta,\gamma))$, where
$X_2(\delta,\gamma)=\frac{1}{2a} \big( \sqrt{b^2+4ac}-b \big)$, that is, 
$$ 
X_2(\delta,\gamma)= 6 C_n^{-\frac{n}{2}} \big( \sqrt{1+(\gamma\delta - \beta)/(18n)}-1\big).
$$ 
Let $\delta \to 1$ and $\gamma \to 1$. Since $X=\sqrt{T} \|\phi\|_{L^{0,1}_2}$, we deduce the value of $T^*$ in \eqref{T*-gere}. 
 
To obtain a lower bound on $\EE\big(\tau(u_0)\big)$, we deduce an estimate for $\EE(\tau_\delta)$ and may suppose that $\tau_\delta <\infty$ a.s. By a.s. continuity of $H(u(\cdot))$, we deduce $\EE\big( \sup_{s\leq \tau_\delta} H(u(s))\big) = \delta H(Q)$. Plugging this in \eqref{E_sup_H_add} and using the Cauchy-Schwarz inequality, we deduce 
\begin{align*}
\delta H(Q) = &\; \beta H(Q) + \frac{1}{2} \|\phi\|_{l^{0,1}_2}^2 \EE(\tau_\delta) + 3 \|\phi\|_{L_2^{0,1}} \sqrt{\EE(\tau_\delta)} \big( \|\nabla Q\|_{L^2_x} + 
 C_n^{\frac{2n}{n-2}}  \|\nabla Q\|_{L^2_x}^{\frac{n+2}{n-2}} \big).
\end{align*}
Let $Z=\|\phi\|_{L_2^{0,1}} \sqrt{\EE(\tau_\delta)}$. Then for $a,b$ as above and $\tilde{c}=1 - \frac{\beta}{\delta}$, we solve $aZ^2+bZ-\tilde{c}=0$. This inequality is satisfied if $Z\geq Z_2(\delta)$, where $Z_2(\delta) = \frac{b+\sqrt{b^2+4a\tilde{c}}}{2a}$. As $\delta \to 1$
we deduce that $\EE\big(\tau(u_0)\big) \geq T^*$.
\end{proof}

\subsection{Multiplicative noise}
In this section we suppose that $W$ is an $\RR$-valued Brownian motion 
$$
W(t,\cdot,\omega)=\sum_{k\geq 0} \phi e_k(\cdot) \beta_k(t,\omega),
$$
where the processes $\{ \beta_k\}_k$ are i.i.d one dimensional
Brownian motions, $\{e_k\}_k$ is an orthonormal basis of $L^2(\RR^n;\RR)$ and $\phi$ is an operator on $L^2(\RR^n;\RR)$. 

We suppose that the operator $\phi$ is Radonifying (notice the difference with the requirement in the additive noise \eqref{E:Radon}, since here we impose that the noise is {\it real-valued}), namely, 
\begin{equation}\label{E:Radon-m}
\qquad \phi \in  R\big(L^2_{\RR}(\RR^n), W_{\RR}^{1,\kappa} (\RR^n)\big), ~~\kappa = \frac{n(n+2)}{n-2}.
\end{equation}

The argument used to prove in \cite[Proposition 4.4]{deB_Deb_H1} the mass conservation in the $H^1$-subcritical multiplicative case can clearly be used to deduce that the mass of the local 
solution is a.s. constant on the local existence time interval $[0, \tau(u_0))$. 
Furthermore, the arguments used to prove \cite[Proposition 4.5]{deB_Deb_H1} imply that the identity satisfied by $H(u(t))$ for $t<\tau(u_0)$ is valid in  our energy-critical
framework.

We first consider the equation \eqref{NLS_Ito_critical} and provide information about the maximal existence time and probability  in the  $H^1$-critical case  in a random setting 
similar to that proven  in \cite[Theorem~2.8]{MilRou} for the intercritical cases.
Besides being Radonifying, we also assume that $\phi$ satisfies  
\begin{equation} \label{M_phi}
    M_\phi : = \|f^1_\phi\|_{L^\infty_{\RR}(\RR^n)} = 
    \sup_{x\in \RR^n} \sum_{k\geq 0} |\nabla (\phi _k)(x)|^2 <\infty. 
\end{equation}
Since the driving noise is real-valued, we know that the mass is a.s. conserved. Furthermore (see, e.g., \cite[Lemma~3.4]{MilRou} with $\alpha =0$), we have 
for any stopping time {$\tau \in (0, \tau(u_0))$.}
\begin{align} \label{E_max_H_multi}
\EE\Big( \sup_{s\leq \tau} H(u(s))\Big) \leq E\big(H(u_0)\big) + \frac12{M_\phi}  \EE(M(u_0) \tau) + 3 \sqrt{M_\phi}
\EE\Big( \sqrt{\tau M(u_0)}  \sup_{s\leq \tau } \|\nabla u(s)\|_{L^2_x} \Big).
\end{align}
\begin{Th}\label{T*-multi-genIC}
Let $u_0$ be an $H^1(\RR^n)$-valued random variable   such that $\EE(M(u_0)) <\infty$, $\|\nabla u_0\|_{L^2_x} <  \|\nabla Q\|_{L^2_x}$ a.s. and 
$H(u_0)\leq \beta H(Q)$ a.s. for some
constant $\beta \in  (0,1)$. 
Let $\phi$ satisfy \eqref{E:Radon-m} and \eqref{M_phi}.
Let $u(t)$ be the solution to \eqref{NLS_Ito_critical} on the random time interval $[0, \tau(u_0))$.

Then 
$P(\tau(u_0)>T) >0$ for every $T<T^*$, where
\begin{equation}\label{T*-gene}
T^* = \frac{ 9 }{C_n^n \, M_{\phi} \, \EE(M(u_0))} \Big( \big(1+\tfrac{2}{9n}(1 - \beta)\big)^{\frac12} -1\Big)^2 .
\end{equation}

Furthermore,  
\begin{equation}\label{E:Etau}
\EE\big(\tau(u_0)  M(u_0) \big) \geq 
\frac{ 9 }{C_n^n \,  M_\phi } \Big( \big(1+\tfrac{2}{9n}(1 - \beta)\big)^{\frac12} -1\Big)^2.
\end{equation}

If $u_0$ is deterministic, then $T_{det}^* = \frac{ 9 }{C_n^n \, M_{\phi} \, M(u_0)} \big( \big(1+\tfrac{2}{9n}(1 - \beta)\big)^{\frac12} -1\big)^2$ and \eqref{E:Etau} 
can be reformulated as $\EE( \tau(u_0)) \geq T_{det}^*$. 
\end{Th}

\begin{proof}
Let $\delta \in (\beta,1)$. Recalling the stopping time $\tau(u_0)$ from the local theory (Theorem \ref{th_lwp_mul}), set 
\[ \tau_{\delta} := \inf\{ s : H(u(s)) \geq \delta H(Q)\} \wedge \tau(u_0) \quad \mbox{\rm and}\quad  X(\delta) := \sup_{s\in [0, \tau_\delta)}
 \|\nabla u(s)\|_{L^2_x}^2.\]
By assumption, $X_0 = \|\nabla u_0\|_{L^2(\RR^n)}^2 < x_c$. 
Fix $T>0$ and  consider $\phi$ such that $M_\phi$ ensures that $P(\tau_\delta > T)>0$. 
The Markov inequality, the estimate \eqref{E_max_H_multi}, the assumption on the noise \eqref{M_phi}, and the trapping property \eqref{trapping}, which
is a consequence of the bounds imposed on the initial condition, and also the Cauchy-Schwarz inequality imply
\begin{align}       
\label{upper_P}
P(\tau_\delta \leq T) = &P\Big( \sup_{s\leq \tau_\delta \wedge T} H(u(s)) \geq \delta H(Q)\Big) \nonumber \\
&\leq \frac{1}{\delta H(Q)} \Big[ \EE\big( H(u_0) \big) + \frac12{M_\phi} \EE\big( M(u_0)\big)  T + 3\sqrt{M_\phi  T } \, 
 \|\nabla Q\|_{L^2_x}  \EE\big( \sqrt{M(u_0)} \,\big) \Big] \nonumber \\
 &\leq \frac{\beta}{\delta} + \frac{1}{2} \frac{M_\phi \EE\big(M(u_0) \big) T}{\delta H(Q)} +  3\sqrt{M_\phi  \EE\big(M(u_0) \big) T }\,  
 \frac{\|\nabla Q\|_{L^2_x}}{\delta H(Q)}.
\end{align}
The right-hand side of \eqref{upper_P} is quadratic in $Y=\sqrt{M_\phi \EE\big(M(u_0) \big)T}$. Given $\gamma \in (\frac{\beta}{\delta}, 1)$, we want to find
$\epsilon(\gamma)>0$ such that for $Y\leq \epsilon(\gamma)$ we have $P(\tau_\delta \leq T) \leq \gamma$. 
Then, as $\gamma\to 1$ and $\delta \to 1$, we will deduce the value of $T^*$ such that the solution exists 
on the time interval $[0,T^*]$ with positive probability by Theorem \ref{th_lwp_mul}. 
Set $a= \frac{n}{2\delta }  C_n^n$, $b=\frac{3n}{\delta} C_n^{\frac{n}{2}} $ and $c=\gamma- \frac{\beta}{ \delta } >0$. Then \eqref{upper_P} becomes $aY^2+bY-c\leq 0$. 
The discriminant of the second degree polynomial on the left-hand side is $D= {C_n^n}{\delta^{-2}} \big( 9n^2 + 2n  (\gamma \delta -\beta)\big)$
 and the roots are $Y_1=\frac{-b-\sqrt{D}}{2a}$
and $Y_2=\frac{-b+\sqrt{D}}{2a}$. For $Y\geq 0$ we deduce that $aY^2+bY-c \leq 0$ if and only if $Y\in [0, Y_2]$, where 
\[ Y_2 = 3 C_n^{-\frac{n}{2}} \big( \sqrt{1+ 2(\gamma \delta- \beta)/(9n)}-1\big),
\]
which implies (as $\gamma \to 1$ and $\delta \to 1$) that $P(T<\tau^*(u_0)) >0$ for
\[ T<T^*=  \frac{9 n^2}{ C_n^n  M_\phi M(u_0)}   \big[ \sqrt{1 + 2 (1 - \beta)/(9n)} - 1\big]^2, 
\]
proving \eqref{T*-gene}. Note that $T^*$ is a decreasing  function of the intensity $M_\phi$ of the forcing noise, of the parameter $\beta$ and of the mass $M(u_0)$ of the initial
condition. 
\smallskip

The proof of the lower bound of $\EE(\tau(u_0)) $ is similar. Indeed, if $\tau_\delta < \infty $ a.s., we have
\begin{align*}
 \delta H(Q) \leq \beta H(Q) + \frac1{2}{M_\phi} \EE \big( M(u_0) \tau_\delta\big) 
 + 3 \sqrt{M_\phi} \|\nabla Q\|_{L^2_x} \EE\big( \sqrt{M(u_0) \tau_\delta}\big).  
\end{align*}

Then, for any constant $C$ the lower bound $\EE(\tau(u_0) M(u_0) )\geq C$ is obvious if $P(\tau_\delta <\infty)<1$. 
Indeed, in that case $\EE(\tau_\delta)=+\infty$, so that 
$\EE(\tau(u_0)M(u_0))\geq \EE(\tau_\delta M(u_0)\big)= +\infty$.
Since $\tau(u_0) \geq \tau_\delta$, we deduce a lower bound for $\EE(\tau(u_0) M(u_0))$.

We next  suppose that $\tau_\delta < \infty$ a.s. 
By a.s. continuity of $H(u(\cdot))$ we deduce that $\EE\big(\sup_{s\leq \tau_\delta} H(u(s)) \big) = \delta H(Q)$. Plugging this in \eqref{E_max_H_multi}
and using the Cauchy-Schwarz inequality, we deduce 
\begin{align}
\delta H(Q) = &\beta H(Q) + \frac12{M_\phi} \, \EE\big( M(u_0) \tau_\delta \big) + 3 \sqrt{M_\phi} \, \EE\big(\sqrt{M(u_0) \tau_\delta }\big)  \| \nabla Q\|_{L^2_x}
\nonumber \\
\leq & \beta H(Q) + \frac12 {M_\phi \, \EE\big(M(u_0) \tau_\delta\big)} + 3 \sqrt{M_\phi  \, \EE\big(M(u_0)  \tau_\delta\big)} \| \nabla Q\|_{L^2_x}.
\end{align}
Set $Z=\sqrt{M_\phi\,\EE(M(u_0) \tau_\delta)}$. Then the above inequality can be written as $aZ^2+bZ-\tilde{c} \geq 0$ for $\tilde{c}=\delta - \beta >0$. 
Replacing $c$ by $\tilde{c}$ in the above computations, we deduce that $M_\phi \, \EE\big( M(u_0) \tau_\delta\big) \geq Z_2^2$,
where
\[ 
Z_2 = 3 C_n^{-\frac{n}{2}} \big( \sqrt{1+ 2( \delta- \beta)/(9n)}-1\big). 
\]
Since $\tau(u_0) \geq \tau_\delta$, for every $\delta <1$, as $\delta \to 1$, we deduce that 
$M_{\phi} \EE(\tau(u_0) M(u_0)) \geq 
9 C_n^{-n} \big( \sqrt{1+ 2(1- \beta)/(9n)}-1\big)
$, which implies \eqref{E:Etau}.

When $u_0$ is deterministic, we may divide the previous inequality by $M(u_0)$ as well as defining $T_{det}^*$ from $T^\ast$ in a similar manner, 
we deduce $\EE(\tau(u_0)) \geq T^*_{det}$, completing the proof.
\end{proof}

\section{Blow-up in finite time}\label{S:B}
In this section we show that solutions to \eqref{NLS-additive} and \eqref{NLS_Ito_critical}  blow up before some time $T>0$ with positive probability in both additive and multiplicative  cases, 
given that the initial condition has its energy positive but ``small" (less than the energy of the ground state) and sufficiently large $L^2$-norm of the gradient 
(as in the deterministic case in \cite[page~656]{KenMer}). Choosing the noise intensity sufficiently small (which we specify in each case), we show that the probability that the stopping time $\tau(u_0)$ from the local theory being finite is positive, and as before we treat each type of noise separately. 
\smallskip

As in the deterministic case, the argument is based on the variance $V(u) $ defined by   
\begin{equation}\label{def_V}
V(u(t))=\int_{\RR^n} |x|^2 |u(t,x)|^2 dx 
\quad \mbox{\rm on the set} \quad \Sigma = \Big\{ u_0\in \dot{H}^1(\RR^n)\, : \int_{\RR^n}  |x|^2 |u_0(x)|^2 dx < \infty  \Big\}, 
\end{equation}
what is typically referred to as a {\it finite} variance. Since in the deterministic case the following identities hold (e.g., see \cite{DR2015})
$$
\frac{d}{dt} V(u(t)) = 4 {\rm Im} \int_{\RR^n} u(x) \, x\!\cdot\!\nabla \bar{u}(x) \, dx, \qquad \frac{d^2}{dt^2} V(u(t)) = \frac{16 n}{n-2} H(u) - \frac{16}{n-2}\| \nabla u (t)\|_{L^2_x}^2, 
$$
an easy convexity argument on the variance shows that blow-up occurs at some finite time if the energy is negative (and one can give an upper bound for it). 
In the stochastic setting this is much more involved in the $L^2$-critical and supercritical cases, as it was first shown by {de Bouard - Debussche in \cite{ deB_Deb_AnnProb} for an
initial condition with negative energy in case of a multiplicative noise.}  In \cite{MilRou} we showed how to incorporate the variance and its time derivatives in the case of positive energy 
(and other conditions) in order to show that blow-up occurs in finite time with positive probability. 
Here, we develop a criterion for blow-up in finite time (with positive probability) in the energy-critical case with positive energy (under the soliton energy threshold and sufficiently large enough $\dot{H}^1$ norm). 

For $u \in \Sigma$ we denote 
\begin{equation}    \label{def_G}
G(u(t)) = {\rm Im}\int_{\RR^n} u(t,x) \, x\!\cdot\!\nabla \bar{u}(t, x) \, dx, .
\end{equation}
and first derive time dependence formulas for $G(u(t))$ and $V(u(t))$. After that, we adapt the convexity argument to our stochastic settings.

As in previous sections, we study the case of an additive stochastic perturbation in \S \ref{S:A-bup}, and then of the multiplicative one  
in \S \ref{S:M-bup}.

\subsection{Additive stochastic perturbation} \label{S:A-bup}
We consider solutions to \eqref{NLS-additive} 
with an additive stochastic perturbation and make the following assumption on the {\it complex-valued} noise. 
\medskip

\noindent$\bullet$  \underline{\bf Condition (H)} ({\it Restrictions on the additive noise.}) 
\begin{itemize}
\item[(i)] 
The operator $\phi$ is Hilbert-Schmidt from $L^2(\RR^n)$ to $H^1(\RR^n)$, i.e., $\phi \in L^{0,1}_2$, and also from $L^2(\RR^n)$ to $\Sigma$, i.e., 
\begin{equation}\label{C-phi-sigma}
{C_\phi^\Sigma := \sum_{k\geq 0}  
\int_{\mathbb R^n} |x \phi e_k(x) |^2 \, dx < \infty.}
\end{equation} 
\item[(ii)] 
The operator $\phi$ is $\gamma$-Radonifying from $L^2(\RR^n)$ to $L^{\frac{2n}{n-2} }(\RR^n)$, that is, 
\begin{equation}\label{C-phi}
C(\phi):=\sum_{k\geq 0}  \|\phi e_k\|_{L^{\frac{2n}{n-2}}_x}^2<\infty.
\end{equation} 
\end{itemize}

We will also use the constants $C^{(1)}_\phi$ and $C^{(2)}_\phi$ defined by 
$$
C^{(1)}_\phi: = \sum_{k\geq 0} \|\nabla(\phi e_k)\|_{L^2_x}^2 
\quad \mbox{\rm and} \quad C^{(2)}_\phi: = {\rm Im }\sum_{k\geq 0}  \int_{\RR^n} \phi e_k(x)\  x\!\cdot\!\nabla(\overline{\phi e_k})(x) dx.
$$
Note that $C^{(1)}_\phi \leq \|\phi\|^2_{L^{0,1}_2}$, and that the Cauchy-Schwarz inequality implies that under 
the  condition {{\bf (H)}} on $\phi$, we have 
$$
C^{(2)}_\phi \leq \sqrt{ C^{(1)}_\phi C^\Sigma_\phi} \leq  \|\phi\|_{L^{0,1}_2} \sqrt{ C^\Sigma_\phi}.
$$

The following lemma is a rewriting  of 
\cite[Lemma~2.2]{deB_Deb_PTRF} adapted to our setting (see also \cite[Lemma~5.6]{MilRou} for the proof in the case $\sigma=\frac{2}{n-2}$). 
\begin{lemma}   \label{lem_v-G-add}
Let $u$ be the solution to \eqref{NLS-additive} satisfying the local theory from Theorem \ref{th_lwp-add} with the stopping time $\tau(u_0)$ and assume that $u_0\in \Sigma$ a.s.  
Let $\phi$ satisfy the condition {\bf (H)}.
Then for any stopping time $\tau<\tau(u_0)$ a.s., we have a.s. 
\begin{equation}    \label{V(u)-add-1}
V(u(\tau)) = V(u_0) +4 \int_0^\tau G(u(s)) ds +2 {\rm Im} \int_0^\tau \sum_{k\geq 0} \int_{\RR^n} |x|^2 \overline{{u}(s,x)} \phi e_k(x) dx d\beta_k(s) + \tau C_\phi^\Sigma, 
\end{equation}
and
\begin{align} \label{G(u)-add}
G(u(\tau)) = & G(u_0) + \frac{4n}{n-2} \int_0^\tau H(u(s)) ds - \frac{4}{n-2}  \int_0^\tau \|\nabla u(s)\|_{L^2_x}^2 ds \nonumber \\
&\qquad + {\rm Re}\int_0^\tau  \sum_{l\geq 0}  \int_{\RR^n} \overline{u(s,x)} \big[ 2\, x\!\cdot\!\nabla(\phi e_l)(x) +n \phi e_l(x)\big] dx d\beta_l(s) 
+ \tau C_\phi^{(2)}. 
\end{align}
\end{lemma}
\medskip

Let $u_0, \sigma$ and $\phi$ be as in Lemma~\ref{lem_v-G-add}.
We next state the time dependence of energy.
Recalling the energy $H(u(\tau))$ from \cite[Proposition 3.3]{deB_Deb_H1},  we deduce for any stopping time $\tau < \tau(u_0)$ a.s.  
\begin{align}\label{H_add}
H\big( u(\tau)\big) &= H(u_0) + \tfrac{1}{2}  C^{(1)}_\phi \tau  
- \mbox{\rm Im } \Big( \int_0^\tau \!\!\int_{\RR^n} \Big[    \Delta \overline{u(s,x)}
- | u(s,x)|^{\frac{4}{n-2}} \overline{u(s,x)}  \Big] \,dx dW(s) \Big) 
\nonumber \\
& \;  +\frac{1}{2}  \sum_{k\geq 0} \int_0^\tau \!\!\int_{\RR^n} \Big[ | u(s,x)|^{\frac{4}{n-2}} |\phi e_k(x)|^2 + \frac{4}{n-2} {|u(s,x)|^{\frac{2(4-n)}{n-2}} } 
\Big( \mbox{\rm Re }\big(\, \overline{u(s,x)}
\phi e_k(x) \big)\Big)^2 
\Big] 
dx ds. 
\end{align}

Then plugging \eqref{H_add} and \eqref{G(u)-add} into \eqref{V(u)-add-1}, we deduce 
that for any stopping time $\tau<\tau(u_0)$ a.s., we have  
\begin{align}		\label{V(u)-add-2}
V\big( u(\tau)\big) = & \; V(u_0) + 4\tau G(u_0) + \frac{8n}{n-2}  H(u_0) \tau^2 - \frac{16}{n-2}  \int_0^\tau ds \int_0^s \| \nabla u(r)\|_{L^2_x}^2 dr \nonumber \\
& \; + 2 C_\phi^{(2)} \tau^2 + \frac{4}{3} n\sigma  C^{(1)}_\phi  \tau^3 +  C_\phi^{\Sigma} \tau + \sum_{j=1}^4 T_j(\tau)\quad \mbox{\rm a.s.},
\end{align}
where
\begin{align*}
T_1(\tau) = 
&\;   \frac{16 n}{n-2} \! \! \int_0^\tau \!\! ds\!\!  \int_0^s \!\! dr \!\! \int_0^r \! \sum_{k\geq 0} {\rm Im} \int_{\RR^n} \!\! 
 \big[ \nabla \overline{{u}(r_1,x)} \nabla (\phi e_k)(x) \\
 &\qquad \qquad \qquad \qquad + |u(r_1,x)|^{\frac{4}{n-2} } \overline{{u}(r_1,x)} \phi e_k(x) \big] dx d\beta_k(r_1) , \\
T_2(\tau) =&\; -  \frac{8n}{n-2}   \!  \int_0^\tau \!\! ds\!  \int_0^s \!\! dr \! \int_0^r \!\! dr_1 \sum_{k\geq 0} \int_{\RR^n} \big[ 
|u(r_1,x)|^{\frac{4}{n-2}} |\phi e_k(x)|^2  \\
&\qquad \qquad \qquad \qquad  
+ \frac{4}{n-2}  |u(r_1,x)|^{\frac{2(4-n)}{n-2} } \big( {\rm Re} (\overline{{u}(r_1,x)} \phi e_k(x)) \big)^2\big]  dx , \\
T_3(\tau) =&\; 4 {\rm Re} \int_0^\tau ds \int_0^s \sum_{k\geq 0}  \int_{\RR^n} \overline{{u}(r,x)}   \big[  2 \,x\!\cdot\! \nabla(\phi e_k)(x) +n \phi e_k(x)  \big] \, dx d\beta_k(r),  \\
T_4(\tau) =&\; 2 {\rm Im} \int_0^\tau \sum_{k\geq 0}  \int_{\RR^n} |x|^2 \overline{{u}(s,x) }\phi e_k(x) \, dx d\beta_k(s). 
\end{align*} 

The following statement describes a sufficient condition on the initial data and on some deterministic positive time, by which blow-up occurs with positive probability. 

\begin{theorem}\label{th_blowup-add}
Let $u_0\in H^1_x\cap \Sigma$ a.s. be ${\mathcal F}_0$-measurable such that $\EE(H(u_0))<\infty$, $\EE(M(u_0)^2)<\infty$, $\EE(V(u_0))<\infty$, and $\EE(G(u_0)^2)<\infty$.  
Let $\phi$ satisfy the condition {\bf(H)}.
Suppose  that for some positive constants $\beta_0$ and $\delta_0$ such that $\beta_0 < 1 < \delta_0$, we have 
\begin{equation} \label{cond-u_0-Q-add}
H(u_0) \leq  \beta_0 H(Q) (\equiv \beta_0 {n^{-1} C_n^{-n}}) \; \mbox{\rm a.s.}  \quad \mbox{\rm  and}\quad  \| \nabla u_0\|_{L^2} 
\geq  \delta_0 \|\nabla Q\|_{L^2} (\equiv \delta_0 C_n^{- \frac{n}{2}}) \; \mbox{\rm a.s.}
\end{equation}
Then for 
$\|\phi\|_{L^{0,0}_2}, \|\phi\|_{L^{0,1}_2}, C(\phi)$, and $C^\Sigma_\phi$ 
small enough, we have $P\big(\tau(u_0) < \infty \big) >0$, where $\tau(u_0)$ is defined in Theorem \ref{th_lwp-add}.
\end{theorem}

Before proving this theorem, we first show a result similar to \cite[Theorem~4.1]{deB_Deb_AnnProb}, however, in our case for positive Hamiltonian, $H(u_0) > 0$, 
and the energy-critical nonlinearity. 
\smallskip

\begin{Prop}\label{Prop-blow-add}
Let $u_0$ satisfy the assumptions of Theorem \ref{th_blowup-add}. 
Given $\epsilon>0$, $\mathfrak M>0$  and $t>0$, suppose that $\phi$  satisfies the following conditions:
\begin{align}
C_\phi^\Sigma + C_\phi^\Sigma t +  \frac{16 n}{3}  \|\phi\|_{L_2^{0,0}} \EE\big(M(u_0)\big)^{\frac{1}{2}}  t^{\frac{3}{2}}+ \big( 2 C_\phi^{(2)} + 32 C_\phi^{(1)} 
 + 4 \sqrt{38}n  \|\phi\|_{L^{0,0}_2}^2 \big)  t^2  +
    \frac{4}{3} n\sigma C_{\phi}^{(1)} t^3 &
   \leq \epsilon ,     \label{cond_1-add}
   \\
  \frac{64 n}{15(n-2) }   \mathfrak M  \sqrt{C_{\phi}^{(1)}}   t^{\frac{5}{2} }
  & \leq \epsilon,    \label{cond_2-add}  
   \\
  \frac{64 n}{15 (n-2) }  C_n^{\frac{n+2}{n-2}} \mathfrak M^{\frac{n+2}{n-2}}   
 \;\sqrt{C(\phi)}  t^{\frac{5}{2} } 
   & \leq \epsilon,  \label{cond_3-add}\\
 \frac{4n(n+2)}{3(n-2)^2}   C_n^{\frac{4}{n-2}} \mathfrak M^{\frac{4}{n-2}}  C(\phi)  t^3 
&  \leq \epsilon.  \label{cond_4-add} 
  \end{align}
For $\delta \in (1, \delta_0]$ and the given above $\mathfrak M>0$, set 
\begin{align}
    \tau_\delta~ &= \inf\{ s\geq 0 :   \|\nabla u(s)\|_{L^2}    \leq \delta \|\nabla Q\|_{L^2} \}\wedge \tau(u_0)    \label{def-tau_delta},\\
    ~\tilde{\tau}_{\mathfrak M}
& = \inf\{ t\geq 0 : \|\nabla u(s)\|_{L^2} \geq \mathfrak M \} \wedge \tau(u_0).
\label{def-tildetau}
\end{align}
Suppose that for some $t>0$ and the constants $\epsilon$, $\mathfrak M>0$  chosen above,  we have  for  $\tau = \tau_\delta \wedge \tilde{\tau}_{\mathfrak M}$
\begin{equation}    \label{cond-3-add-Bis}
  \EE(   V(u_0) )   +4\epsilon  +4  \big\{\EE\big(  G(u_0)^2\big)\big\}^{\frac{1}{2}} \big\{ \EE \big( (t\wedge \tau)^2 \big)\big\}^{\frac{1}{2}}  \nonumber \\
    - \frac{8}{n-2} C_n^{-n} (\delta^2-\beta_0) \EE\big( (t\wedge \tau)^2\big) <0. 
\end{equation}
Then $P(\tau(u_0)\leq t)>0$. 
\end{Prop}
\begin{proof}
Suppose that   $t<\tau(u_0)$ a.s. for some $t$. Taking expected values in  \eqref{V(u)-add-2}, we obtain 
\begin{align}		\label{EV-add-deter}
 \EE\big( V\big(u(t\wedge \tau)\big)   = &\,  \EE\big( V(u_0) \big) 
+ 4 \EE\big(  G(u_0) (t\wedge \tau) \big) + \frac{8n}{n-2} \EE\big( H(u_0)  (t\wedge \tau)^2\big)  \nonumber \\
& - \frac{16 }{n-2}   \EE\Big(  \int_0^{t\wedge \tau} ds \int_0^s \|\nabla u(r)\|_{L^2}^2 dr\Big)  +  C_\phi^{\Sigma}  \EE\big( (t\wedge \tau)\big)  \nonumber \\
&
+ 2 C^{(2)}_\phi  \EE\big((t\wedge \tau)\big)^2  + \frac{4}{3} n\sigma C^{(1)}_\phi \EE\big(  (t\wedge \tau)^3\big) + \sum_{j=1}^4 \tilde{T}_j(t),
\end{align} 
where for $j=1, ..., 4$ we set $\tilde{T}_j(t)= \EE\big( T_i(t\wedge \tau)\big)$, with the terms $T_j(t\wedge \tau)$ defined 
in the statement of \eqref{V(u)-add-2}. 
By assumption we have $ H(u_0 )\leq \beta_0 n^{-1} {C_n^{-n}}$ a.s. and by definition of $\tau_\delta$,
for $0\leq r\leq \tau \leq \tau_\delta$, we have $ \|\nabla u(r)\|_{L^2}^2
\geq \delta^2 C_n^{-n}$. Hence,
\begin{align}		\label{Hyp-Q-add}
\frac{8n}{n-2} 
&\EE\big(   H(u_0)  (t\wedge \tau)^2\big) -\frac{8}{n-2}   
\EE\Big(  \int_0^{t\wedge \tau} ds \int_0^s \| \nabla u(r)\|_{L^2}^2 dr \Big) \nonumber \\
&\leq \Big[ \frac{8n}{n-2} \beta_0 {n^{-1} C_n^{-n}} - \frac{8}{n-2} \delta^2 C_n^{-n}   \Big] \EE\big( (t\wedge \tau)^2\big) \nonumber \\
&\leq - \frac{8}{n-2} C_n^{-n}  \big( \delta^2-\beta_0) 
\, \EE\big( (t\wedge \tau)^2 \big). 
\end{align}

We next give upper estimates of $|\tilde{T}_j(t)| $, $j=1, ..., 4$.  All these upper estimates contain either a multiplicative factor that depends on the strength of the noise 
or a time integral of $ \EE(V(\cdot \wedge \tau))$. We sketch the arguments and send to the proof of \cite[Proposition~5.10]{MilRou} for more details. 
First, note that  $|\tilde{T}_1(t) | \leq  \tilde{T}_{1,1}(t) + \tilde{T}_{1,2}(t)$, where
\begin{align*}
\tilde{T}_{1,1}(t)=
&\;  \frac{16n}{n-2}   \Big| \EE\Big( \int_0^{t\wedge \tau} \!\!\! ds \int_0^s \!dr \int_0^r \sum_{k\geq 0} \int_{\RR^n} \nabla \overline{{u}(r_1,x)} \nabla (\phi e_k)(x) dx 
d\beta_k(r_1)\Big)\Big|, \\
\tilde{T}_{1,2}(t) = 
&\; \frac{16n}{n-2}   \Big| \EE\Big( \int_0^{t\wedge \tau}\! \!\! ds \int_0^s \! dr \int_0^r \sum_{k\geq 0}  \int_{\RR^n} 
|u(r_1,x)|^{\frac{4}{n-2}} \overline{u(r_1,x)} \phi e_k(x) dx d\beta_k(r_1)\Big) \Big|.
\end{align*}
The Cauchy-Schwarz inequality (with respect to $dP$ and then to $dx$) and the It\^o isometry imply
\begin{align}\label{upp-tildeT_11}
 \tilde{T}_{1,1}(t) \leq
 &\;  \frac{16n}{n-2}  \int_0^t\! \! ds \int_0^s \! dr \Big\{ \EE\Big( \int_0^{r\wedge \tau}  \sum_{k\geq 0}  \Big( \int_{\RR^n} 
 \nabla \bar{u}(r_1,x) \nabla (\phi e_k)(x) dx\Big)^2 dr_1 \Big) \Big\}^{\frac{1}{2}} \nonumber \\
 \leq 
 & \; \frac{16n}{n-2} \;   \mathfrak M  \,   \sqrt{C_\phi^{(1)}} \, \frac{4}{15} t^{\frac{5}{2}},
\end{align}
where in the last estimate we used the fact that  $r\leq \tau \leq \tilde{\tau}_{\mathfrak M}$. 
A similar computation using H\"older's inequality implies that
\begin{align}
 \tilde{T}_{1,2}(t) \leq & \; \frac{16 n}{n-2}  \int_0^t\!\! ds \int_0^s\!\! dr \Big\{ \EE\Big(\int_0^{r\wedge \tau} \sum_{k\geq 0} \Big( \int_{\RR^n} |(u(r_1,x)|^{\frac{n+2}{n-2} }
  |\phi e_k(x)| dx \Big)^2 dr_1 \Big) \Big\}^{\frac{1}{2}} \notag \\
\leq &\;  \frac{16 n}{n-2}  \int_0^t\!\! ds \int_0^s\!\! dr \Big\{ \EE\Big(\int_0^{r\wedge \tau} \sum_{k\geq 0}  \|u(r_1)\|_{L_x^{\frac{2n}{n-2} }}^{ \frac{2(n+2)}{n-2} }
\| \phi e_k\|_{L_x^{\frac{2n}{n-2}}}^2  dr_1 \Big) \Big\}^{\frac{1}{2}}. \label{E:sub1}
\end{align}

Furthermore, the Sobolev embedding \eqref{Sob} implies that
$ \|u(r_1)\|_{L_x^{2n/(n-2)}} \leq C_n    \|\nabla u(r_1)\|_{L^2_x}$.

Since $C(\phi)=\sum_k \| \phi e_k\|_{L_x^{\frac{2n}{n-2}}}^2 $, we deduce (recalling the definition of  $\tilde{\tau}_{\mathfrak M}$  from 
\eqref{def-tildetau}, which implies  the bound $\mathfrak M$ 
on the gradient up to the stopping  time 
$\tilde{\tau}_{\mathfrak M}$) 
\begin{equation}	\label{upp-tildeT_12}
\tilde{T}_{1,2}(t) \leq \frac{16 n}{n-2}   C_n^{\frac{n+2}{n-2}} \;  \sqrt{C(\phi)} \;  {\mathfrak M}^{\frac{n+2}{n-2}}\; 
\frac{4}{15} t^{\frac{5}{2}}. 
\end{equation}
To bound $\tilde T_2(t)$, we use the Cauchy-Schwarz and H\"older inequalities, and then \eqref{Sob} and the definition of $C(\phi)$ to obtain 
\begin{align} \label{upp-tilde-T2}
|\tilde{T}_2(t)| \leq &\;  \frac{4n}{n-2}  \int_0^t\!\! ds \int_0^s\! \! dr \, \EE\Big( \int_0^{r\wedge \tau} 
\frac{n+2}{n-2} \sum_{k\geq 0}  \int_{\RR^n} |u(r_1,x)|^{\frac{4}{n-2}} |\phi e_k(x)|^2
 dx dr_1  \Big) 
 \nonumber  \\
\leq &\; \frac{4n(n+2)}{(n-2)^2}    \int_0^t\!\! ds \int_0^s\! \! dr \, \EE\Big( \int_0^{r\wedge \tau} 
\sum_{k\geq 0}   \|u(r_1)\|_{L_x^{\frac{2n}{n-2}}}^{\frac{4}{n-2}}  \; \|\phi e_k\|_{L_x^{\frac{2n}{n-2}}}^2 dr_1 \Big) \nonumber \\
\leq &  \;  \frac{4n(n+2)}{(n-2)^2} \; 
C_n^{\frac{4}{n-2}}  {\mathfrak M}^{\frac{4}{n-2}}  \; 
 C(\phi) \; \frac{t^3}{6}.
\end{align}

We next bound the term $\tilde{T}_3(t)$. The computations made to prove 
\cite[(5.61)]{MilRou} imply 
\begin{align}		\label{upper-tilde-T3}
\big| \tilde{T}_3(t)\big| 
\leq &\;  32 \, C_\phi^{(1)}  t^2 +  \frac{16}{3}\,  {n} \|\phi\|_{L^{0,0}_2} \EE\big(M(u_0)\big)^{\frac{1}{2}}  \, t^{\frac{3}{2}} +  4 \sqrt{38}\, n \|\phi\|_{L^{0,0}_2}^2 t^2 
 + \int_0^t  \EE\big(  V(u(s\wedge \tau))\big) ds . 
\end{align}

Finally, a similar argument implies (see the proof of \cite[(5.62)]{MilRou})
\begin{equation}		\label{upper-tilde-T4}
\big| \tilde{T}_4(t)\big|  
\leq   C_\phi^\Sigma  + \int_0^t \EE\big( V(u(s\wedge \tau))\big) ds. 
\end{equation} 

Collecting all the estimates \eqref{EV-add-deter}--\eqref{upper-tilde-T4}, we obtain
\begin{align*}
\EE\big( & V(u(t\wedge \tau))\big) \leq \; \EE\big(  V(u_0) \big) 
+ 4  \EE\big(   G(u_0) (t\wedge \tau) \big)   -\frac{8 }{n-2}  C_n^{-n} (\delta^2-\beta_0) \EE\big( (t\wedge \tau)^2\big)   \\
& +  \big[ C_\phi^\Sigma t 
 + 2  C_\phi^{(2)} t^2 + \frac{4}{3} n\sigma C_{\phi}^{(1)} t^3 \big]  +  \frac{64 n}{15(n-2) }  \,\mathfrak M \sqrt{C_\phi^{(1)}}  t^{\frac{5}{2}} \\\
&+ \frac{64n}{15(n-2) }   \,   C_n^{\frac{n+2}{n-2}} \, 
{\mathfrak M}^{\frac{n+2}{n-2}}\,  \sqrt{C(\phi)}\,  t^{\frac{5}{2}}  
+ \frac{4n(n+2)}{3(n-2)^2} \,  C_n^{\frac{4}{n-2}} \, {\mathfrak M}^{\frac{4}{n-2}}  \,  t^3  \\
&+ 32  C^{(1)}_\phi t^2 + \frac{16 }{3} n  \| \phi\|_{L^{0,0}_2} \sqrt{M(u_0)}  t^{\frac{3}{2}} + 4 \sqrt{38} n \|\phi\|_{L^{0,0}_2}^2 t^2 
+  C_\phi^\Sigma 
+2 \int_0^t  \EE\big(  V(u(s\wedge \tau))\big) ds.
\end{align*}

Since $\| \nabla u(s)\|_{L^2}^2 \leq {\mathfrak M}^2
<\infty$ for $ s\leq t\wedge \tau \leq t\wedge \tilde{\tau}_{\mathfrak M}$, the argument used in
\cite[page 85]{deB_Deb_PTRF}  implies that $\EE\big( \sup_{s\in [0,t]} V(u(s))\big) <\infty$. Hence, Gronwall's lemma and the bounds
\eqref{cond_1-add}--\eqref{cond_4-add} imply  
\begin{align}   \label{form_V_Q}
\EE\big( &  V(u(t\wedge \tau)) \big) \leq  \;  e^{2t} 
 \Big[ \EE\big(  V(u_0)\big)  + 4 \EE\big( G(u_0)   (t\wedge \tau) \big)  - \frac{8}{n-2} C_n^{-n} (\delta^2-\beta_0)  \EE\big(  (t\wedge \tau)^2\big) 
+  { 4\epsilon }  \Big].
\end{align} 

The hypothesis  \eqref{cond-3-add-Bis}  implies that $ \EE\big( V(u(t\wedge \tau))\big) <0$, which brings a contradiction. Hence, $P( \tau(u_0)<t) >0$, concluding the proof. 
\end{proof}

We are now ready to embark onto the proof of Theorem \ref{th_blowup-add}.
\smallskip

{\it Proof of Theorem \ref{th_blowup-add}.} We assume that $\tau(u_0)=\infty$ a.s. and look for a contradiction.

Let $\gamma \in (\beta,1)$ and let $\epsilon_0\in \big( 0, \frac{1}{4}\big)$ satisfy $\gamma (1-4\epsilon_0) > \beta_0$. 
Recall the monotonicity assumptions on the function $f$ defined by \eqref{def-f}. 
Then, given $\gamma \in (\beta,1)$, we deduce the existence of $\delta >1$ such that
\begin{equation}\label{E:IFF}
\big(  x>x_c \quad \mbox{\rm and}\quad   f(x)\leq \gamma f(x_c) \big) \quad   \Longrightarrow
\quad  x\geq \delta x_c. 
\end{equation}

Since $\|\nabla u_0\|_{L^2_x} > x_c$ a.s., $\|\nabla u_0\|_{L^2_x} $ is located on the decreasing side of the graph.
 Therefore, by a.s. continuity of $\|\nabla u(s)\|_{L^2_x}$ and $\|u(s)\|_{L^2_x}$ in $s$, we deduce that given  any  stopping time $\tau^*_0$ and $s\leq \tau^*_0$ 
such that $H(u(s)) \leq  \gamma f(x_c)$ and
$\|\nabla u(s)\|_{L^2_x}  > x_c$, we have 
$\|\nabla u(s)\|_{L^2_x} \geq  \delta x_c$, 
in other words, 
\begin{equation}     \label{gamma-delta-A}
    \sup_{s\in [0,\tau_0^\ast) }H(u(s))  \leq \gamma H(Q)  \; \Longrightarrow 
\inf_{s\in [0,\tau_0^\ast)}\|\nabla u(s)\|_{L^2_x}  \geq \delta \|\nabla Q\|_{L^2_x} .
\end{equation} 

For $\gamma$ and $\delta$ chosen above,  set 
\begin{equation} \label{sigma_gamma2}
\tilde{\sigma}_\gamma = \inf\big\{ s\geq 0 :    H(u(s))  \geq \gamma H(Q) \big\}, 
\end{equation}
and
\begin{equation} \label{tau_delta2}
\tau_\delta = \inf\big\{ s\geq 0 : \| \nabla u(s)\|_{L^2_x}   \leq  \delta
 \|\nabla Q\|_{L^2_x} \big\} .
\end{equation} 
Then by \eqref{gamma-delta-A}, we have $\tilde{\sigma}_\gamma \leq \tau_{\delta}$. Furthermore, 
given any stopping time $\tau \leq \tau_\delta$
\[  \int_0^{T\wedge \tau}  ds \int_0^s dr \|\nabla u(r)\|_{L^2}^2 \geq 
\delta^2  \int_0^{T \wedge \tau} ds \int_0^s dr
\|\nabla Q\|_{L^2} ^2 .\]
Using once more the identities in \eqref{nabla-H-Q}, we deduce that
\begin{align} \label{compensation_add}
       \frac{8n}{n-2}  &  \EE\big( H(u_0)    (T\wedge \tau)^2 \big) -  \frac{16 }{n-2} \EE\Big( 
 \int_0^{T\wedge \tau } ds \int_0^s \|\nabla u(r)|_{L^2}^2 dr \Big)  \nonumber \\
 & \leq - \frac{8}{n-2} C_n^{-n}  (\delta^2-\gamma)  \EE\big(  (T\wedge \tau)^2 \big).
\end{align}

The upper estimate \eqref{EV-add-deter} and the argument leading to 
\eqref{form_V_Q} imply that for ~$\tau = \tau_\delta 
\wedge \tilde{\tau}_N$, where ~
$\tilde{\tau}_N = \inf\{ s\geq 0: \|\nabla u(s)\|_{L^2_x} \geq N\} $, 
we have  
\begin{align*}      
& \EE\big( V(u(T\wedge \tau)) \big) \leq  \;   \EE\big(  V(u_0) \big) + 4  \EE\big(  G(u_0) (T\wedge \tau) \big)  - \frac{8}{n-2} C_n^{-n}  (\delta^2-\gamma)   \EE\big( (T\wedge \tau)^2\big)
\nonumber \\
&\quad + \Big[ C_\phi^\Sigma T + 2 C^{(2)}_\phi T^2  { +\frac{4}{3} n\sigma C_\phi^{(1)} T^3 } \Big]
+\sum_{j=1}^4 \tilde{T}_j(T).
\end{align*}
For $\|\phi\|_{L^{0,0}_2}, \| \phi\|_{L^{0,1}_2}, C(\phi)$ and $C^\Sigma_\phi$ small enough (which implies that {$C^{(1)}_\phi$ and} $C^{(2)}_\phi$ are small),
the upper estimates of the terms $\tilde{T}_j, j=1, ...4$, yield the upper estimates {\eqref{cond_1-add}--\eqref{cond_4-add} }. Then by the Cauchy-Schwarz inequality we deduce that 
for $\tau \leq \tau_\delta$
\begin{align}   \label{Gronwall-V-add}
 \EE\big(  V(u(T\wedge \tau)) \big) \leq & \;  e^{2T} 
 \Big[ \EE\big( V(u_0) \big)  + 4  \big\{ \EE\big( G(u_0)^2\big) \big\}^{\frac{1}{2}} \big\{ \EE\big(   (T\wedge \tau)^2\big)  \big\}^{\frac{1}{2}}  
\nonumber \\
&\qquad  
- \frac{8 C_n^{-n} }{n-2}  (\delta^2-\gamma)   \EE\big( (T\wedge \tau)^2\big)
+ { 4\epsilon } \Big] .
\end{align}

We then proceed as in the proof of Theorem \cite[Theorem~5.9]{MilRou} and first fix a large value of $T>0$. As $T\to \infty$, $T\wedge \tau_\delta \to \tau_\delta$ 
and the monotone convergence
theorem implies that $ \EE\big( (T\wedge \tau_\delta)^2 \big)\to \EE(\tau_\delta^2)$. 
\smallskip

$\bullet$  If $\EE(\tau_\delta^2)=+\infty$, for any fixed $M$ (to be chosen later) we have
$\EE\big( (T\wedge \tau_\delta)^2\big) \geq M^2$ for some large $T$. 

$\bullet$  If $\EE(\tau_\delta^2)<\infty$, then for any $\mu\in (0,1)$ and close enough to 1, 
for $T$ large enough we have
$\EE\big( (T\wedge \tau_\delta)^2 \big) \geq \mu^2 \EE(\tau_\delta^2)$ and $P\big(  T<\tilde{\sigma}_\gamma \big) < {\epsilon}_0$, where
the last inequality follows from the fact that $\tilde{\sigma}_\gamma \leq \tau_\delta < \infty$ a.s.
\smallskip

As $M\to \infty$ the sequence ${\tilde \tau}_M \to \infty$ a.s. Indeed, if we had $\lim_M {\tilde \tau}_M=\tau $ such that $P(\tau<\infty)>0$, then we would deduce that
$P(\sup_{t\leq \tau} \| \nabla u(t)\|_{L^2} =\infty)>0$, so that $P(\tau(u_0) \leq \tau)>0$, which contradicts the assumption $\tau(u_0)=\infty$ a.s.
We deduce that 
for $\tilde{\tau}_{M_0} = \inf\{  s\geq 0 : \|\nabla u(s)\|_{L^2} \geq M_0\}$,
for some $M_0$ large enough, if $\bar{\tau}_0 = \tau_\delta \wedge \tilde{\tau}_{M_0}$, 
we have  
$\EE\big( (T\wedge \bar{\tau}_0)^2 \big) = \EE\big( (T\wedge \tau_\delta \wedge \tilde{\tau}_{M_0})^2\big) \geq \mu^2 
\EE\big( ( T\wedge\tau_\delta)^2 \big)$.
Hence, either $\EE\big( (T\wedge \bar{\tau}_0)^2 \big) \geq \mu^2 M_0^2$
or $\EE\big( (T\wedge \bar{\tau}_0)^2 \big) \geq \mu^4 \EE\big(\tau_\delta^2\big)$. 

For this choice of $T$, $M$ and $M_0$,  
set $X= \big\{\EE\big((T\wedge {\tau}_\delta)^2 \big)\big\}^{\frac{1}{2}}$ and consider the polynomial 
$-ax^2+bx+c$, with $a=\frac{8}{n-2} C_n^{-n} (\delta^2-\gamma^2) $, $b=4 \big\{ \EE\big( G(u_0) \big) \big\}^{\frac{1}{2}} $ and
$c= \EE\big( V(u_0) \big)  + {4\epsilon}$. The Cauchy-Schwarz inequality and \eqref{Gronwall-V-add} imply 
\[  \EE\big( V(T\wedge \bar{\tau}_0) \big)\leq e^{2T} \big[ -a X^2 +bX+c\big] .\]
Let $X_1<X_2$ be the roots of $-aX^2+bX+c$. Then for $X>X_2$, we can conclude that $\EE\big(  V(T\wedge \bar{\tau}_0)\big) <0$,
which proves that
$P( \tau(u_0)<T)>0$.
\smallskip

We next  claim that when the noise is ``small enough", we have 
$\EE\big( (T\wedge \tau_\delta)^2 \big) >2 X_2^2 $.
Choosing $\mu$ close enough to 1, we deduce that for ``small" noise, 
$\EE\big(  (T\wedge \bar{\tau}_{0})^2 \big) >  X_2^2 $.

$\bullet$  If $\EE(\tau_\delta^2)=\infty$,  choosing $M>\sqrt{2} X_2$, we deduce that  
$\EE\big( (T\wedge \tau_\delta)^2\big) > 2 X_2^2$ for $T$ large enough, completing the claim.

$\bullet$ If $\EE(\tau_\delta^2)<\infty$, 
it remains to prove that  $\EE\big( (\tau_\delta \wedge T)^2\big) \geq 2 X_2^2$. 

Since $\EE(\tau_\delta^2)<\infty$ and 
$\tilde{\sigma}_\gamma \leq \tau_\delta$, by a.s. continuity of $H(u(\cdot))$ on $[0,T]$, we deduce that 
$H(u(\tilde{\sigma}_\gamma)) = \gamma H(Q)$ a.s.
The upper bound \eqref{H-add} implies that for any stopping time $\tau \leq \tau_\delta$ 
\begin{align*}
   H\big( u(\tau)\big) 
\leq &  H(u_0) + \frac{1}{2}   \|\phi\|_{L^{0,1}_2}^2 \tau  +    I_1(\tau) + I_2(\tau), 
\end{align*}
where 
\begin{align}		\label{I1-I2}
I_1(\tau)= &\;  - {\rm Im} \Big( \int_0^{\tau} \sum_k \int_{\RR^n}
|u(s,x)|^{\frac{4}{n-2}} \overline{u(s,x) } (\phi e_k)(x) dx d\beta_k(s) \Big),  \nonumber \\
I_2(\tau)=  
&\; {\rm Im} \Big( \int_0^{\tau} \sum_k \int_{\RR^n} \nabla \overline {u(s,x)} \nabla(\phi e_k)(x) dx d\beta_k(s) \Big). 
\end{align} 
For a fixed positive $\mathfrak L$, 
set $\tilde{\tau}_{\mathfrak L} = \inf\{ s\geq 0: \|\nabla u(s)\|_{L^2_x} \geq \mathfrak L\}$. 
Then H\"older's inequality with conjugate exponents $\frac{2n}{n-2}$ and $\frac{2n}{n+2}$ implies that for any stopping time $\tau \leq \tau_\delta$, we have
\begin{align}		\label{upper-I1}
\EE(\big| I_1(\tau \wedge T  \wedge \tilde{\tau}_{\mathfrak L} ) |^2 \big)\leq   & 
\; \EE\Big( \int_0^{\tau \wedge T \wedge  \tilde{\tau}_{\mathfrak L}} 
\sum_k \Big|   \int_{\RR^n}
|u(s,x)|^{\frac{n+2}{n-2} }  (\phi e_k)(x) dx \Big|^2  ds\Big)   \nonumber \\
\; \leq &\; \EE\Big( \int_0^{\tau \wedge T \wedge \tilde{\tau}_{\mathfrak L}}  \sum_k \|u(s)\|_{L_x^{\frac{2n}{n-2} }}^{\frac{2(n+2)}{n-2} }
\|\phi e_k\|_{L_x^{\frac{2n}{n-2} }}^2  ds \Big)  \nonumber \\
\leq &\; C(\phi) C_n^{\frac{2(n+2)}{n-2} }   {\mathfrak L}^{\frac{2(n+2)}{n-2} }  T<  \infty,  
\end{align}
where  in the last two estimates we used \eqref{Sob}. 
Therefore,  
$\EE\big( I_1(\tau_\wedge T \wedge \tilde{\tau}_{\mathfrak L} ) \big)=0$. 

A similar computation yields
\begin{align*}
 \EE(    I_2(\tau  \wedge T \wedge \tilde{\tau}_{\mathfrak L})|^2) \leq & \; \EE\Big( \int_0^{\tau \wedge T  \wedge \bar{\tau}_{\mathfrak L}}
 \sum_k \Big( \int_{\RR^n} \nabla \bar{u}(s,x) \nabla(\phi e_k)(x) dx \Big)^2 ds \Big)  \\
 \leq & \; \EE\Big( \int_0^{\tau \wedge T  \wedge \tilde{\tau}_{\mathfrak L}} \sum_k \| \nabla (\phi e_k)\|_{L^2_x}^2 \| \nabla u(s)\|_{L^2_x}^2 ds \Big)  \leq \|\phi\|_{L^{0,1}_2}^2 {\mathfrak L}^2 T <\infty. 
\end{align*}
Hence, $\EE(I_2(\tau \wedge T \wedge \tilde{\tau}_{\mathfrak L}))=0$, and we deduce that for every ${\mathfrak L}>0$ and
$\tau \leq \tau_\delta$, we have
\[   \EE(  H\big( u(\tau \wedge T \wedge \tilde{\tau}_{\mathfrak L})\big) \leq  \EE\big( H(u_0)  \big) 
 + \frac{1}{2}   \|\phi\|_{L^{0,0}_2}^2 \EE\big((\tau \wedge T)
 \big). \]
We next let ${\mathfrak L}\to \infty$. 
 Let $\sigma_0= \inf\{ t\geq 0 \; : \; H(u(t))\leq 0\}$. 
  Since $\tilde{\sigma}_\gamma \leq \tau_\delta $, choosing $\tau = \tilde{\sigma}_\gamma \wedge \sigma_0$, we deduce  
\[ \sup_{\mathfrak L} H(u(\tilde{\sigma}_\gamma \wedge \sigma_0 \wedge T\wedge \tilde{\tau}_{\mathfrak L})) \leq 
 \sup_{s\leq \tilde{\sigma}_\gamma  }  H(u(s)) 
\leq  \gamma H(Q) + \frac1{2} {T}\| \phi\|_{L^{0,0}_2}^2  <\infty\quad \mbox{\rm a.s.}
\]
Furthermore, by definition of $\sigma_0$, the sequence $ \big\{ H(u(\tilde{\sigma}_\gamma \wedge \sigma_0 \wedge T\wedge \tilde{\tau}_{\mathfrak L}))\big\}_
{\mathfrak{L}} $
is nonnegative. 
Therefore,  as ${\mathfrak L} \to \infty$, the dominated convergence theorem  implies
$   \EE\big( H(u(\tilde{\sigma}_\gamma \wedge \sigma_0 \wedge T \wedge \tilde{\tau}_{\mathfrak L})) \big) \to 
\EE\big(  H(u(\tilde{\sigma}_\gamma \wedge \sigma_0 \wedge T )) \big)$.
 
The definition of $\tilde{\sigma}_\gamma$ in \eqref{sigma_gamma2} and the a.s. continuity of $H(u(\cdot))$ on $(0, \infty)$ imply $H(u(\tilde{\sigma}_\gamma\wedge \sigma_0 \wedge T)) 
= \gamma H(Q)$ a.s.  on $\{ \tilde{\sigma}_\gamma \leq \sigma_0 \wedge T\}$.
 Thus, using once more the inequality $H(u(s))(\omega) \geq 0$ for every $s\leq \sigma_0(\omega)$ and neglecting $ \EE\big(  H(u(\sigma_0\wedge T))
 1_{\{\sigma_0\wedge T<\tilde{\sigma}_\gamma \}}\big)\geq 0 $, we obtain
\begin{align} 	\label{upp-gamma-beta} 
\gamma H(Q)  P ( \tilde{\sigma}_\gamma    \leq \sigma_0 \wedge T) 
 \leq &\;  \EE\big(  H(u_0) \big)  +  
\frac{1}{2}  \|\phi\|_{L^{0,1}_2}^2  
\EE\big(   (\tilde{\sigma}_\gamma \wedge T) \big) 
\nonumber \\
\leq & \;   \beta_0  H(Q)  +  
\frac{1}{2} 
\|\phi\|_{L^{0,1}_2}^2
\big\{\EE( (\tau_\delta \wedge T)^2)\big\}^{\frac{1}{2}} ,
\end{align} 
where in the last upper estimate we used  the upper estimate $\tilde{\sigma}_\gamma \leq \tau_\delta$, 
the assumption 
\eqref{cond-u_0-Q-add} and the Cauchy-Schwarz inequality. 
Furthermore,
\begin{align*}  P(\tilde{\sigma}_\gamma \leq T)  & = P(  \tilde{\sigma}_\gamma \leq \sigma_0\wedge T) +
P(\sigma_0 < \tilde{\sigma}_\gamma \leq T )
\leq  P( \tilde{\sigma}_\gamma \leq \sigma_0\wedge T) +  P( \sigma_0 < T).
\end{align*} 
On $\{ \sigma_0 <T\}$  we have  
\[ 0=H(u(\sigma_0)) \geq H(u_0)  +  I_1(\sigma_0) + I_2(\sigma_0),\]
where  $I_1$ and $I_2$ have been defined in \eqref{I1-I2}. 
Let ${\mathfrak L}_0$ be chosen such that $P(\tilde{\tau}_{{\mathfrak L}_0} \leq T) \leq {\epsilon}_0$. 
Then 
 \[ P(  \sigma_0 <T  ) 
 \leq  P(\tilde{\tau}_{{\mathfrak L}_0}\leq T) +
\sum_{j=1,2}   P\Big(  \sup_{s\in [0,T \wedge \tilde{\tau}_{{\mathfrak L}_0} ]} |I_j(s)|  \geq {\beta_0 H(Q)}/{2} \Big). 
 \] 
The Markov  and Davis inequalities, and the computations that  yield \eqref{upper-I1}, imply
\begin{align*}
P\Big( \sup_{s\in [0,T \wedge \tilde{\tau}_{{\mathfrak L}_0}  ]} & |I_1(s)|  \geq \frac12 {\beta_0 H(Q)} \Big) \leq \frac{6}{\beta_0 H(Q)}
   \EE\Big( \Big\{ \int_0^{T\wedge \tilde{\tau}_{{\mathfrak L}_0}}
 \sum_k \Big|   \int_{\RR^n}
|u(s,x)|^{\frac{n+2}{n-2} }  (\phi e_k)(x) dx \Big|^2  ds \Big\}^{\frac{1}{2}} \Big)    \nonumber \\
\; \leq &\; \frac{6 n}{\beta_0}  C_n^n  \EE\Big( \Big\{ \int_0^{ T \wedge \tilde{\tau}_{{\mathfrak L}_0}}   \sum_k \|u(s)\|_{L_x^{\frac{2n}{n-2}}}^{\frac{2(n+2)}{n-2}  }
\|\phi e_k\|_{L_x^{\frac{2n}{n-2} }}^2 ds \Big\}^{\frac{1}{2}} \Big) \nonumber \\
\leq &\;6 n \, \beta_0^{-1}\,  C(\phi)^{\frac{1}{2}}  C_n^{n+ \frac{n+2}{n-2} }   {{\mathfrak L}_0}^{\frac{n+2}{n-2}}   \sqrt{T} \leq \epsilon_0, 
\end{align*}  
if $C(\phi)$ is small enough. 
 
Again, using the Markov and Davis inequalities, and then the Cauchy-Schwarz inequality with respect to  $dx$, we obtain 
\begin{align*}
  P\Big(  \sup_{s\in [0,T \wedge \tilde{\tau}_{{\mathfrak L}_0}  ]} & |I_2(s)|  \geq \frac12{\beta_0 H(Q)}  \Big) \leq \frac{6}{H(u_0)}
   \EE\Big( \Big\{ \int_0^{T\wedge \tilde{\tau}_{{\mathfrak L}_0}}
  \sum_k \Big( \int_{\RR^n}  \nabla \overline{u(s,x)} \nabla \phi e_k(x) dx \Big)^2 ds \Big\}^{\frac{1}{2}} \Big)\\
  & \leq \frac{6}{\beta_0 H(Q) }  \EE\Big( \Big\{ \int_0^{T\wedge \tilde{\tau}_{{\mathfrak L}_0}} {\mathfrak L}_0 \; \sum_k \| \nabla (\phi e_k)\|_{L^2_x}^2 ds \Big\}^{\frac{1}{2}} \Big)\\
  &\leq 6 n \, C_n^n \,  \beta_0^{-1}\,  \sqrt{T {\mathfrak L}_0} \|\phi\|_{L^{0,1}_2} \leq {\epsilon}_0,
 \end{align*} 
 if $\|\phi\|_{L^{0,1}_2} $ is small enough. 
 Therefore,  we deduce  $P(\sigma_0\leq T) \leq 3\epsilon_0$ for $\| \phi\|_{L^{0,1}_2}$ and $C(\phi)$ small enough.  Since we have chosen
 $T$ large enough to have $P(T<\tilde{\sigma}_\gamma)\leq \epsilon_0$, we obtain 
 \[ 1 \leq P(T<\tilde{\sigma}_\gamma) + P( \sigma_0 < T) + P(\tilde{\sigma}_\gamma \leq \sigma_0\wedge T)
 \leq P(\tilde{\sigma}_\gamma \leq \sigma_0\wedge T)+4\epsilon_0. 
 \] 
Hence, the upper estimate \eqref{upp-gamma-beta} yields
\[ \gamma H(Q)  \leq \beta_0 H(Q)  + 
\frac{1}{2}  \| \phi\|_{L^{0,1}_2}^2 \big\{ \EE\big( (\tau \wedge T)^2\big) \big\}^{\frac{1}{2}} + 4 \epsilon_0 \gamma H(Q), 
\] 
which implies 
\[ \EE\big(
 (\tau_\delta \wedge T)^2\big) \geq \frac{4
\big(\gamma (1-4\epsilon_0) -\beta_0 \big)^2  }{ C_n^{2n}\,  n^2 \, \|\phi\|_{L^{0,1}_2}^4}.
\]
 Recall that we have chosen $\epsilon_0\in \big(0, \frac{1}{4}\big)$ such that $\gamma (1-4\epsilon_0) > \beta_0$. 
Therefore, if  $\|\phi\|_{L^{0,1}_2}$ is small enough, we deduce that 
$\EE\big(  (\tau_\delta\wedge T)^2\big) \geq 2 X_2^2$,
which completes the proof.  
\hfill $\square$
\bigskip

\subsection{Multiplicative noise}\label{S:M-bup}
We start with solutions to \eqref{NLS_Ito_critical}. 
The following lemma is an adaptation of \cite[Proposition~3.2]{deB_Deb_AnnProb} to our energy-critical case.

\begin{lemma}\label{lem_v-G-multi}
Let $u$ be the solution to \eqref{NLS_Ito_critical} and assume that $u_0\in \dot{H}^1\cap \Sigma$ a.s.
Suppose that  
$\phi$ satisfies \eqref{E:Radon-m} and \eqref{M_phi}. 
Then for any
stopping time $\tau<\tau(u_0)$ a.s., we have
\begin{equation}    \label{V(u)}
V(u(\tau)) = V(u_0) +4 \int_0^\tau G(u(s)) ds\quad a.s.
\end{equation}
and 
\begin{align} \label{G(u)}
G(u(\tau)) =  \; G(u_0) + & \frac{4n}{n-2}  \int_0^\tau H(u(s)) ds - \frac{4n}{n-2}  \int_0^\tau \|\nabla u(s)\|_{L^2}^2 ds \nonumber \\
&\; + \sum_{l\geq 0} \int_0^\tau \int_{\RR^n} |u(s,x)|^2 \, x\!\cdot\!\nabla(\phi e_l)(x) \, dx \, d\beta_l(s)\quad a.s.
\end{align}
\end{lemma}
\begin{proof} The equation \eqref{V(u)} is proved in \cite{deB_Deb_AnnProb}, see (3.4). To prove \eqref{G(u)}, we rewrite $G(u(\tau))$ using It\^o's formula as on page 1095 in \cite{deB_Deb_AnnProb}.  This yields
\begin{align*} 
G(u(\tau)) =  \; G(u_0) +& 2\int_0^\tau \int_{\RR^d} |\nabla u(s,x)|^2 dx ds - \frac{2 n}{n+2} \int_0^\tau \int_{\RR^n}
|u(s,x)|^{2\sigma +2} dx \\
&\; +\sum_{l\geq 0} \int_0^\tau \int_{\RR^n} |u(s,x)|^2 \, x\!\cdot\!\nabla(\phi e_l)(x) \, dx \, d\beta_l(s)\\
=  \; G(u_0) + & \frac{4n}{n-2}  \int_0^\tau H(u(s)) ds - \frac{4}{n-2}   \int_9^\tau \| \nabla u(s)\|_{L^2}^2 ds\\
&\; +\sum_{l\geq 0} \int_0^\tau \int_{\RR^n} |u(s,x)|^2 \, x\!\cdot\!\nabla(\phi e_l)(x) \, dx d\beta_l(s).
\end{align*}
This completes the proof of \eqref{G(u)}.
\end{proof}

The following result describes a sufficient condition on the initial condition and on some deterministic positive time with the blow-up occurring before that time with positive probability. 
Recall  that $f^1_\phi = \sum_{l\geq 0} |\nabla (\phi e_l)|^2$ and $M_\phi=\| f^1_\phi\|_{L^\infty}$.

\begin{theorem}\label{th_blowup2} 
Let the noise $W$ be real-valued and $\phi$ satisfy \eqref{E:Radon-m} and \eqref{M_phi}. 
Let $u_0 \in H^1_x \cap \Sigma$  a.s. be ${\mathcal F}_0$-measurable such that ~$\EE\big( M(u_0)^2\big)<\infty$, 
~ $\EE \big( V(u_0)\big) <\infty$, ~ $\EE\big( H(u_0)\big) <\infty$  and ~ $\EE \big( G(u_0)^2 \big) <\infty$. 
Let $\tau(u_0)$ be the stopping time  defined in Theorem \ref{th_lwp_mul}.

Suppose that for some constants $\beta_0$ and $\delta_0$ such that $0<\beta_0 < 1 < \delta_0$, we have 
\begin{equation} \label{cond-u_0-Q-blowup}
H(u_0)  \leq  \beta_0 H(Q) = \frac{\beta_0}{n C_n^n}  \; {\mbox{\rm a.s.} \quad \mbox{\rm  and}\quad 
\| \nabla u_0\|_{L^2} 
\geq  \delta_0 \|\nabla Q\|_{L^2}  = \delta_0 C_n^{\frac{n}{2}} \; \mbox{\rm a.s.} }
\end{equation}

Then for $M_\phi$ small enough, we have  
$P\big(\tau(u_0) < \infty\big) >0$.
\end{theorem}
As in the case of additive noise, the proof relies on the corresponding technical result.

\begin{Prop}\label{Prop-blow-multi}
Let $u_0$ and $\phi$ satisfy the assumptions of Theorem \ref{th_blowup2}, 
replacing $\delta_0$ by $\delta>1$ in the statement of \eqref{cond-u_0-Q-blowup}. Suppose that for some $\epsilon >0$ and $N>0$, $t$ and $M_\phi$ satisfy the
following conditions
\begin{align}
\EE\big(M(u_0) \big) M_\phi  \Big( \frac{32}{3(n-2)} t^3 + 4t \Big) 
&\leq \epsilon,  \label{cond_1}\\
   \frac{64 n}{15 (n-2) } \big\{  \EE\big( M(u_0)\big)\big\}^{ \frac{1}{2}}\sqrt{M_\phi} N t^{\frac{5}{2}} & \leq \epsilon.    \label{cond_2} 
\end{align}
For $1<\delta<\delta_0$ set
\begin{align}
\tau_\delta = &\; \inf\{ s\geq 0 \, :\, \| \nabla u(s)\|_{L^2_x} \leq \delta \| \nabla Q\|_{L^2_x} \} \wedge \tau(u_0),	\label{def-tau_delta-multi} \\
\tilde{\tau}_N= &\; \inf\{ s\geq 0 \, : \, \| \nabla u(s)\|_{L^2_x} \geq N\} \wedge \tau(u_0). 	\label{def-tildetauK-multi}
\end{align}
Let $\tau = \tau_\delta \wedge \tilde{\tau}_N$, and suppose that
\begin{align}    \label{cond-3}
   \EE\big( V(u_0) \big)& + 2\epsilon +4 \big\{ \EE\big( G(u_0)^2 \big) \big\}^{\frac{1}{2}} \big\{  \EE\big((t\wedge \tau)^2\big) \big\}^{\frac{1}{2}} 
       - \frac{8C_n^{-n} }{n-2}  (\delta^2-\beta_0) \EE\big( (t\wedge \tau)^2\big) <0. 
\end{align}
Then $P(\tau(u_0)\leq t)>0$. 
\end{Prop}
\begin{proof}
Assume that $t<\tau(u_0)$ a.s. and write $V(u)$ using \eqref{V(u)}, \eqref{G(u)} and the following equation for the energy $H(u(t))$  (see \cite[Proposition~4.5]{deB_Deb_H1}) 
\[ H(u(t)) = H(u_0) - Im  \int_{\RR^n} \! \int_0^t\bar{u}(s,x) \nabla u(s,x)\!\cdot\!\nabla dW dx + \frac{1}{2} \sum_{l\geq 0}\!\!  \int_0^t \!\! \int_{\RR^n}\!\!  |u(s,x)|^2 |\nabla (\phi e_l)(x)|^2 dx ds.
\] 
Then using \eqref{V(u)} and \eqref{G(u)}, we obtain
\begin{align*}
V(u(t&\wedge \tau)) = V(u_0) + 4 G(u_0) (t\wedge \tau)  + \frac{8n}{n-2}  H(u_0) (t\wedge \tau)^2  - \frac{16n}{n-2}  \int_0^{t\wedge \tau} ds \int_0^s \|\nabla u(r)\|_{L^2_x}^2 dr \\
& + \frac{8n}{n-2}  \int_0^{t\wedge \tau} \int_0^s \int_0^r \int_{\RR^n} |u(r_1,x)|^2 f^1_\phi(x) \,
dx dr_1 dr ds\\
&-\frac{16n}{n-2}  \int_0^{t\wedge \tau} \int_0^s \int_0^r {\rm Im}\sum_{l\geq 0} \int_{\RR^n}  \bar{u}(r_1,x) \nabla u(r_1,x)\!\cdot\!\nabla(\phi e_l)(x) \, dx
d\beta_l(r_1) dr ds\\
& +4 \int_0^{t\wedge \tau}  \int_0^s \sum_{l\geq 0} \int_{\RR^n} |u(r,x)|^2 \, x\!\cdot\!\nabla(\phi e_l)(x) \, dx d\beta_l(r) ds.
\end{align*}
Thus,  taking expected values, we deduce 
\begin{align}    
  & \EE\big(V(t\wedge \tau)\big)   \leq  \EE\big( V(u_0) \big) + 4   \EE\big( G(u_0) (t\wedge \tau) \big) 
\nonumber \\
  &  \qquad 
  + \frac{8n}{n-2}  \EE\big(   H(u_0)(t\wedge \tau) ^2\big)-\frac{16}{n-2}  \EE\int_0^{t\wedge \tau} ds \int_0^s \|\nabla u(r)\|_{L^2}^2 dr + \sum_{j=1}^3 T_j,  \label{E:MEV-1b}
\end{align}
where 
\begin{align*}
 T_1= &\frac{8n}{n-2}  \, \EE\Big(  \int_0^{t\wedge \tau} ds \int_0^s dr \int_0^r \int_{\RR^n}   |u(r_1,x)|^2 f^1_\phi(x) dx  \, dr_1 \Big), \\
T_2=&  -\frac{16n}{n-2}  \EE\Big(   \int_0^{t\wedge \tau}ds \int_0^s dr\int_0^r {\rm Im}\sum_{l\geq 0} \int_{\RR^n} 
\bar{u}(r_1,x)\,  \nabla u(r_1,x)\!\cdot\! \nabla(\phi e_l)(x) dx
d\beta_l(r_1) \Big),\\
T_3=&4  \EE\Big(\int_0^{t\wedge \tau}  ds\int_0^s \sum_{l\geq 0} \int_{\RR^n} |u(r,x)|^2 \, x\!\cdot\!\nabla(\phi e_l)(x) dx d\beta_l(r) \Big).
\end{align*}
Since for the time $r\leq \tau_\delta$ a.s., we have
\[  \|\nabla u(r)\|_{L^2_x}^2 \geq \delta^2  \|\nabla Q\|_{L^2_x}^2, \quad \mbox{\rm  while}\quad 
 H(u_0) \leq \beta_0  H(Q),\]
using  the identity \eqref{nabla-H-Q},
we deduce that  
\begin{align}\label{compensation}
\frac{8n}{n-2} & \EE\big(  H(u_0) (t\wedge \tau)^2\big) - \frac{16}{n-2} \,
\EE\Big(\int_0^{t\wedge \tau} ds \int_0^s \|\nabla (u(r))\|_{L^2}^2 dr \Big) 
\nonumber \\
& \leq \frac{8n}{n-2}  \beta_0 \|\nabla Q\|_{L^2_x}^2 \, \EE\big( (t\wedge \tau)^2\big)  - \frac{16}{n-2}  \EE\Big( \int_0^{t\wedge \tau} ds \int_0^s
\delta^2 \|\nabla Q\|_{L^2}^2 dr \Big) \nonumber \\
&\leq -\frac{8}{n-2} C_n^{-n} (\delta^2-\beta_0)  \, \EE\big( (t\wedge \tau)^2\big).
\end{align}
The upper estimates  \eqref{E:MEV-1b} and \eqref{compensation} yield 
\begin{align}\label{MEV-2}
  \EE\big( V(t\wedge \tau)\big)  \leq & \EE\big(   V(u_0)\big)  + 4 \EE\big( G(u_0)  (t\wedge \tau)\big)  
   -\frac{8}{n-2} C_n^{-n}  (\delta^2-\beta_0)  \EE\big( (t\wedge \tau)^2\big) + \sum_{i=1}^3 T_i.
\end{align}
We now estimate the last three terms. Since $M_\phi = \|f^1_\phi\|_{L^\infty}$, we have
\begin{equation}    \label{estim_T1}
T_1\leq   
\frac{8}{n-2} \EE\big( M(u_0)\big) \,  M_\phi \int_0^t ds \int_0^s dr \int_0^r dr_1 = \frac{32}{3(n-2)}  \EE\big( M(u_0)\big) \, M_\phi t^3. 
\end{equation}
Using Fubini's Theorem and the Cauchy-Schwarz inequality with respect to $dP$   and then to $dx$, we obtain
\begin{align}\label{estim_T2}
|T_2|\leq & \frac{16 n}{n-2}  \int_0^t ds \int_0^s dr \, \EE \Big( \Big| \int_0^{r\wedge \tau} \sum_{l\geq 0} \int_{\RR^n}
 \bar{u}(r_1,x) \nabla u(r_1,x)\!\cdot\!\nabla (\phi e_l) (x) dx d\beta_l(r_1)\Big| \Big) \nonumber \\
\leq & \;  \frac{16 n}{n-2}  \int_0^t ds \int_0^s dr \Big\{ \EE\Big( \int_0^{t\wedge \tau}  \sum_{l\geq 0}  \Big| 
\int_{\RR^n} \bar{u}(r_1,x) \nabla u(r_1,x)\!\cdot\!\nabla (\phi e_l) (x) dx\Big|^2 dr_1 \Big) \Big\}^{\frac{1}{2}}\nonumber \\
\leq &\;   \frac{16n}{n-2} \int_0^t ds \int_0^s dr \Big\{ \EE\Big[ \int_0^{t\wedge \tau}  \sum_{l\geq 0}  \Big( \int_{\RR^n} |\bar{u}(r_1,x)|^2 
|\nabla (\phi e_l)(x)|^2 dx \Big) \nonumber \\
&\hspace{2cm}  \times \Big( \int_{\RR^n} |\nabla u(r_1,x)|^2 dx\Big) dr_1 \Big) \Big] \Big\}^{\frac{1}{2}} \nonumber \\
\leq & \; \frac{16 n}{n-2}  \big\{ \EE\big( M(u_0)\big) \big\}^{ \frac{1}{2}} \sqrt{M_\phi} \int_0^t ds \int_0^s dr \Big\{ \EE\big( {N^2} (r\wedge \tau)\big) \Big\}^{\frac{1}{2}}
\nonumber \\
\leq &\;  \frac{16n}{n-2}\big\{ \EE\big( M(u_0)\big) \big\}^{ \frac{1}{2}}    \sqrt{M_\phi} N \frac{4}{15} t^{\frac{5}{2}}.
\end{align}
Similar computations yield
\begin{align*}       
 |T_3|\leq &\;  4 \int_0^t ds \,\EE\Big( \Big| \int_0^{s\wedge \tau} \sum_{l \geq 0}  \int_{\RR^n} |u(r,x)|^2 \,x\!\cdot\!\nabla(\phi e_l)(x) dx d\beta_l(r) \Big| \Big)
 \nonumber \\
 \leq & \; 4 \int_0^t ds \Big\{ \EE\Big( \int_0^{s\wedge \tau} \sum_{l\geq 0} \Big|\int_{\RR^n} |u(r,x)|^2 \, x\!\cdot\! \nabla(\phi e_l)(x) dx\Big|^2  dr  \Big)\Big\}^{\frac{1}{2}}
 \nonumber \\
 \leq & \; 4 \int_0^t ds \Big\{ \EE\Big( \int_0^{s\wedge \tau} \sum_{l\geq 0}  \Big( \int_{\RR^n} |u(r,x)|^2 |\nabla(\phi e_l)(x)|^2 dx\Big)
 \Big( \int_{\RR^n }|x|^2 |u(r,x)|^2 dx\Big) dr \Big] \Big\}^{\frac{1}{2}} \nonumber \\
 \leq &\;  4 \big\{ \EE\big( M(u_0)\big) M_\phi \big\}^{ \frac{1}{2}}  \int_0^t ds \Big\{ \EE\int_0^s V\big( u(r\wedge \tau)\big) dr \Big) \Big\}^{\frac{1}{2} }\nonumber \\
 \leq & \; 4 \big\{ \EE\big( M(u_0)\big)  M_\phi \big\}^{ \frac{1}{2}}  \sqrt{t} \Big\{ \int_0^t ds \int_0^s {\EE} \big( V(u(r\wedge \tau)) dr \big) \Big\}^{\frac{1}{2}}.
\end{align*}
Young's inequality implies that for $\bar{\epsilon} = t^{-1} $,   
\begin{equation}        \label{estim_T3}
    |T_3| \leq \int_0^t  \EE\big( V(u(r\wedge \tau))\big) dr + 4 \EE\big(M(u_0) \big) M_\phi \,t. 
\end{equation} 
Collecting estimates for $T_1, T_2, T_3$ and recalling the bounds \eqref{cond_1}-\eqref{cond_2}, we obtain 
$$
\Big| \sum_{i=1}^3 T_i \Big| \leq \int_0^t  \EE\big( V(u(s\wedge \tau))\big) ds + 2\epsilon. 
$$
Putting this into the estimate \eqref{MEV-2}, we get an inequality on the expected value of the variance, from which we would like to extract the bound on it. 
However, before proceeding (and applying Gronwall's inequality), we need to make sure that this expected value is bounded a.s.
For that we refer to \cite{deB_Deb_AnnProb}: since $\tau \leq \tilde{\tau}_N$ and $t<\tau (u_0)$ a.s., the upper estimate (6.2) on page 1095 in \cite{deB_Deb_AnnProb} 
implies that $u(s\wedge \tau)\in \Sigma$ a.s. for
every $s\leq t$ and 
\[ \sup_{s\leq t} V(s\wedge \tau)\leq [4N^2t+V(u_0)] e^t \quad \mbox{\rm a.s.}\]
Therefore, $\sup_{r\leq t} \EE(V(r\wedge \tau)) <\infty $, and Gronwall's lemma yields { for every $t>0$} 
\begin{align}       \label{estim_EV_Gronwall}
    \EE\big[ V\big( u(t\wedge \tau)\big) \big] \leq &\big[  \EE\big( V(u_0)\big) + 2\epsilon +4 \EE \big(G(u_0) (t\wedge \tau)\big) 
         - \frac{8C_n^{-n} }{n-2}  (\delta^2-\beta_0) \EE\big( (t\wedge \tau)^2\big)   ] e^t.
\end{align}
Using the assumption \eqref{cond-3}, we deduce that $\EE\big( V(u(t\wedge \tau))\big) <0$, which brings a contradiction, since
$V(u)\geq 0$ for every $u\in \Sigma$. Therefore,  $P(  \tau(u_0) \leq t)>0$, which completes the proof. 
\end{proof}
\medskip

\noindent {\it Proof of Theorem~\ref{th_blowup2}}. ~
We suppose that $\tau(u_0)=\infty$ a.s. and look for a contradiction. 

As before, we use the monotonicity properties of the function $f$ defined by  \eqref{def-f}.
Since $\|\nabla u(t)\|_{L^2_x}$ is a.s. continuous in $t$ and initially we have $x=\|\nabla u_0\|_{L^2_x} $
 greater than $\|\nabla Q\|_{L^2_x} $ (that is, $x > x_c$ implying that $x$ is located on the decreasing side of the graph of $f$); hence 
there exists $\delta \in (1,\delta_0)$ and $\gamma \in (\beta_0,1)$ such that $\gamma(1-2\epsilon)>\beta_0$ for some $\epsilon >0$, and for any $t>0$  we have a.s. 
\begin{equation}     \label{gamma-delta}
\sup_{s\in [0,t] }H(u(s))  \leq \gamma H(Q) \; \Longrightarrow 
\inf_{s\in [0,t]}\|\nabla u(s)\|_{L^2_x}  \geq \delta \|\nabla Q\|_{L^2_x}  .
\end{equation}

Let $\tau_\delta$ be as in Proposition \ref{Prop-blow-multi} and define the stopping time $\tilde \sigma_\gamma$ by 
\begin{equation}\label{sigma_gamma1} \tilde{\sigma}_\gamma = \inf\{ s\geq 0 : H(u(s))
\geq \gamma H(Q) \} . 
\end{equation} 

Then \eqref{gamma-delta} implies that $\tilde{\sigma}_\gamma \leq \tau_\delta$.
As $T\to \infty$, we have $T\wedge \tilde\sigma_\gamma \to \tilde\sigma_\gamma$, 
and  the monotone convergence theorem  implies $\EE\big((T\wedge \tau_\delta)^2 \big)
\to \EE\big(\tau_\delta^2\big)$. 
We consider two cases, depending on the size of the last quantity. 
\smallskip

$\bullet$ If $\EE\big(\tau_\delta^2\big)=\infty$, given any $M>0$ there exists $T$ large enough, call it $T_0>0$, to ensure $\EE\big((T_0\wedge \tau_\delta)^2\big) \geq M^2$.\smallskip

$\bullet$ If $\EE(\tau_\delta^2)<\infty$, given any $\lambda \in (0,1)$, we may choose $T_0$ large enough to have 
$\EE\big((T_0 \wedge \tau_\delta)^2\big) \geq \lambda^2 \EE(\tau_\delta^2)$. 

Furthermore,  as shown in the proof of Theorem \ref{th_blowup-add}, since $\|\nabla u(t)\|_{L^2_x} \to \infty$ a.s. when $t\to \infty$, 
we deduce that for $\tilde{\tau}_N$ defined by \eqref{def-tildetauK-multi},
 we have $\tilde{\tau}_N \to \infty$ a.s. when  $N \to \infty$. 
Hence, there exists $N_0$ such that for $\bar{\tau}_0=\tau_\delta \wedge \tilde{\tau}_{N_0}$ we have
\[ \EE \big((T_0 \wedge \bar{\tau}_0)^2\big) = \EE\big((T_0 \wedge \tau_\delta \wedge \tilde{\tau}_{N_0})^2\big) \geq 
\lambda^2 \EE\big((T_0 \wedge \tau_\delta)^2\big),\]
which implies that either $\EE\big((T_0 \wedge \bar{\tau}_0 )^2\big) \geq \lambda^2 M^2$
or $\EE\big((T_0 \wedge \bar{\tau}_0)^2\big) \geq \lambda^4 \EE(\tau_\delta^2)$. 

Now, that $T_0$ and $N_0$ have been defined, given $\epsilon >0$, choose $M_\phi$ small enough to ensure that
\begin{equation}    \label{cond1-2}
   4 \EE\big( M(u_0)\big) M_\phi T_0 \Big( 1+\frac{8 T_0^2}{3}\Big) \leq \epsilon \quad \mbox{\rm and} \quad \frac{64 n}{15 (n-2) } \big\{ \EE\big( M(u_0)\big) \big\}^{\frac{1}{2}}   \sqrt{M_\phi} N_0 T_0^{\frac{5}{2}} 
\leq \epsilon, 
\end{equation}
and thus, the conditions \eqref{cond_1} and \eqref{cond_2} are satisfied.  

We next show that 
there exists $\epsilon>0$ (close to 0) and $T>0$ (large enough) such that for $\bar{\tau}_0= \tau_\delta \wedge \tilde{\tau}_{N_0}$, the Cauchy-Schwarz inequality implies 
$$
\EE\big(  V(u_0) \big)  + 2\epsilon +4   \big\{\EE \big(G(u_0) ^2\big) \big\}^{\frac{1}{2}}  \big\{ \EE\big((T\wedge \bar{\tau}_0)^2\big) \big\}^{\frac{1}{2}} 
- \frac{8}{n-2} C_n^{-n}  (\delta^2-\beta_0)  \EE\big( (T\wedge \bar{\tau}_0)^2\big) <0,
$$
i.e., the expression in square brackets on the right-hand side of \eqref{estim_EV_Gronwall} is negative  for $T$ large. 

Set 
$$
a= \frac{8}{n-2} C_n^{-n}  (\delta^2-\beta_0), 
\quad b=4 \big\{ \EE\big( G(u_0)^2\big) \big\}^{\frac{1}{2}}  \quad \mbox{\rm and} \quad
c= \EE\big( V(u_0)\big)  +2\epsilon, 
$$
and denote $X:= \big\{ \EE\big( (T\wedge \bar{\tau}_0)^2\big) \big\}^{\frac{1}{2}}$. 
We look for $\epsilon>0$ and $T>0$ such that  $-a X^2 +bX+c <0$. 
Indeed, if $-aX^2+bX+c<0$,  the condition \eqref{cond-3} is satisfied with $T\wedge \bar{\tau}_0$ instead of $t\wedge \tau$, 
and Prop. \ref{Prop-blow-multi} implies that $P(T<\tau(u_0)) >0$. 
Let $X_1<X_2$ denote the roots of the polynomial $-aX^2+bX+c$. 
\smallskip

$\bullet$ If $\EE(\tau_\delta^2)=\infty$, we may choose $M$ large enough and $\lambda $ close to one to have $M\lambda >X_2$, which implies 
$X>X_2$. 

$\bullet$ If $\EE(\tau_\delta^2)<\infty$, then $\tilde{\sigma}_\gamma \leq \tau_\delta <\infty$ a.s.
Thus,  the a.s.  continuity  of the energy $H(u(\cdot))$ and  the definition of $\tilde{\sigma}_\gamma$ in \eqref{sigma_gamma1} imply
$H(u(\tilde{\sigma}_\gamma))   = \gamma H(Q) $.
Using \eqref{cond-u_0-Q-blowup} and $f^1_\phi = \sum_{l\in \NN} |\nabla (\phi e_l)|^2$,  
we deduce that for any stopping time $\tau $, we have 
\begin{align}  \label{Mu_0H(u)}
  H(u(\tau))= &\beta_0 H(Q) +\frac{1}{2}  \int_0^\tau ds \int_{\RR^n} |u(s,x)|^2 f^1_\phi(x) dx  -   {\rm  Im}\big(I(\tau)\big), 
\end{align}
where  
\[ I(\tau):=  \int_0^\tau \sum_{l\geq 0}  \int_{\RR^n} 
\bar{u}(s,x) \nabla u(s,x)\!\cdot\!\nabla(\phi e_l)(x) dx d\beta_l(s).
\]

Let $\sigma_0$ be the stopping time defined by
\[ \sigma_0 = \inf\{ s\geq 0 \; : \; H(u(s)) \leq 0\} .\] 
Then the Cauchy-Schwarz inequality implies 
\begin{align} \label{square-int-mart}
 \EE\big( \big|   {\rm Im}\big(  I(\tilde{\sigma}_\gamma \wedge \sigma_0 \wedge \tilde{\tau}_k) \big) \big|^2 \big)
\leq \EE\Big( \int_0^{\tilde{\sigma}_\gamma\wedge \sigma_0 \wedge \tilde{\tau}_k} \|f^1_\phi\|_{L^\infty_x}  M(u_0) \, \|\nabla u(s)\|_{L^2_x}^2 ds\Big) <\infty. 
\end{align}
Therefore, for any $k>0$  the above stochastic integral is square integrable, hence, centered; thus, the Cauchy-Schwarz inequality yields
\[  \EE\big(H(u(\tilde{\sigma}_\gamma\wedge  \sigma_0 \wedge \tilde{\tau}_k))\big) \leq \beta_0 H(Q)  + \frac{1}{2}  M_\phi
\big\{ \EE\big( M(u_0)^2 \big) \big\}^{\frac{1}{2}} \big\{  \EE\big((\tilde{\sigma}_\gamma)^2\big)\big\}^{\frac{1}{2}} .
\]
For a fixed $k>0$ recall that $\tilde{\tau}_k = \inf\{ s\geq 0 : \|\nabla u(s)\|_{L^2} \geq k\}$. Then as $k\to \infty$, we get $\tilde{\tau}_k\to \infty $ a.s.
By a.s. continuity  of $H(u(\cdot))$ we have  $ H(u(\tilde{\sigma}_\gamma \wedge \sigma_0 \wedge \tilde{\tau}_k))\to  H(u(\tilde{\sigma}_\gamma \wedge \sigma_0))$ 
as $k\to \infty $  and 
$$ 
 H(u(\tilde{\sigma}_\gamma \wedge \sigma_0)) \leq  \gamma H(Q)   1_{\{ \tilde{\sigma}_\gamma \leq  \sigma_0\} }
+ 0 \cdot   1_{\{ \sigma_0 < \tilde{\sigma}_\gamma\}} \qquad \mbox{\rm a.s.}
$$
Furthermore,   
$  H(u(\tilde{\sigma}_\gamma \wedge \sigma_0 \wedge \tilde{\tau}_k)) \in [0, \gamma\; H(Q)] $,  and hence, 
the dominated convergence theorem and the Cauchy-Schwarz inequality imply 
\begin{equation} 		\label{upper_1}
\gamma H(Q) \; P( \tilde{\sigma}_\gamma \leq  \sigma_0)  \leq \beta_0 H(Q)  + \frac{1}{2}
 M_\phi  \big\{ \EE\big( M(u_0)^2 \big) \big\}^{\frac{1}{2}} \big\{  \EE\big( (\tilde{\sigma}_\gamma)^2\big) \big\}^{\frac{1}{2}}.
\end{equation} 

We next prove that when the strength of the noise is small enough,  $P( \tilde{\sigma}_\gamma \leq  \sigma_0) $ can be made as close to one as desired. 
Let $\epsilon \in (0,\frac{1}{2})$ be chosen such that $\gamma (1-2\epsilon) >\beta_0$. 
We may choose $k_0$  and $T_0$ large enough 
to have $P(\tilde{\sigma}_\gamma
\geq  \tilde{\tau}_{k_0} \wedge T_0 ) \leq \epsilon$. 

Furthermore, on the set
$\{\sigma_0\leq \tilde{\sigma}_\gamma \wedge \tilde{\tau}_{k_0}\wedge T_0\}\subset \Omega$,  
the identity \eqref{Mu_0H(u)} implies  
$0\geq \beta_0 H(Q)  -  {\rm Im}\big( I(\sigma_0)\big)$, so we have the following inclusion:
\begin{equation}		
\label{inclusion}
 \{\sigma_0 \leq \tilde{\sigma}_\gamma \wedge \tilde{\tau}_{k_0} \wedge T_0\} \subset \Big\{ \sup_{s\leq \tilde{\sigma}_\gamma \wedge \tilde{\tau}_{k_0}\wedge T_0}
   {\rm Im}\big(I(s)\big) \geq
\beta_0 H(Q)  \Big\}.
\end{equation}
Then the Markov and Davis inequalities imply
\begin{align}		\label{estim-IM(u)}
P\Big(& \sup_{s\leq \tilde{\sigma}_\gamma \wedge \tilde{\tau}_{k_0}\wedge T_0}\; {\rm Im}\big( I(s)\big) \Big)  \leq \frac{1}{\beta_0 H(Q) }
 \EE\Big( \sup_{s\leq \tilde{\sigma}_\gamma \wedge \tilde{\tau}_{k_0}\wedge T_0}
 |I(s)| \Big)  \nonumber \\
\leq & \; \frac{3  }{\beta_0 H(Q) } \EE\Big(  \Big\{ \int_0^{\tilde{\sigma}_\gamma \wedge \tilde{\tau}_{k_0}\wedge T_0} 
\sum_l \Big( \int_{\RR^n} \bar{u}(s,x) \, \nabla u(s,x) \!\cdot\!\nabla (\phi e_l)(x) dx\Big)^2 ds 
\Big\}^{\frac{1}{2}} \Big) \nonumber  \\
\leq &\; \frac{3  }{\beta_0 H(Q) } \EE\Big( \Big\{ \int_0^{\tilde{\sigma}_\gamma \wedge \tilde{\tau}_{k_0}\wedge T_0}  M(u_0) \|f^1_\phi\|_{L^\infty_x} 
 \Big(\int_{\RR^n} |\nabla u(s)|^2 dx\Big) ds  \Big\}^{\frac{1}{2}} \Big) \nonumber \\
 \leq &\; \frac{3 \sqrt{ k_0 T_0}   \big\{\EE\big(  M(u_0)\big) \big\}^{\frac{1}{2}} }{\beta_0 H(Q) }  \sqrt{M_\phi}, 
\end{align}
where in the last bound we used the definition of $\tilde{\tau}_{k_0}$ and in the one before the last one, the Cauchy-Schwarz inequality with respect to $dx$, 
and  the conservation of mass  identity $M(u(s))=M(u_0)$ a.s. 
Therefore, if $M_\phi$ is small enough, we have $P(\sigma_0 \leq \tilde{\sigma}_\gamma \wedge \tilde{\tau}_{k_0} \wedge T_0) \leq \epsilon$.
Finally,  we obtain the upper bound
\[ P( \sigma_0 < \tilde{\sigma}_\gamma ) \leq P(\sigma_0 \leq \tilde{\sigma}_\gamma  \wedge \tilde{\tau}_{k_0} \wedge T_0) + P(\tilde{\sigma}_\gamma \geq \tilde{\tau}_{k_0} \wedge T_0) \leq 2\epsilon.
\] 
For $M_\phi$ small enough,  the inequality \eqref{upper_1}  implies 
\begin{align*}
 (1-2\epsilon) \gamma H(Q)  \leq&  \;   \beta_0 H(Q) + \frac{1}{2}
  M_\phi \,   \big\{ \EE\big( M(u_0)^2 \big) \big\}^{\frac{1}{2}} \big\{  \EE\big( (\tilde{\sigma}_\gamma)^2\big) \big\}^{\frac{1}{2}}.
\end{align*} 
{Since $\tilde{\sigma}_\gamma \leq \tau_\delta$ and $\EE\big( (T_0 \wedge \tau_\delta)^2\big) \geq \lambda^2 \EE(\tau_\delta^2)$, for $M_\phi$  small enough and $\lambda$
close enough to 1}, we have 
\[ X = \Big\{ \EE\big( (T\wedge \bar{\tau}_0)^2\big) \big\}^{\frac{1}{2}} \geq \frac{2\lambda^2 \big[ \gamma(1-2\epsilon) -\beta_0\big]  H(Q) }{
M_\phi \big\{ \EE\big( M(u_0)^2\big) \big\}^{\frac{1}{2}}}  > X_2.
\]
Hence, the condition \eqref{cond-3} is satisfied, which concludes the proof. 
\qed
\medskip




\begin{thebibliography}{99}

\bibitem{adams}  
R. A. Adams and J. J. F. Fournier,    
{\em Sobolev spaces}, 
Pure and Applied Mathematics Series, 2nd ed., Academic
Press, 2003. 

\bibitem{BRZ2016}
V. Barbu, M. R\"ockner and D. Zhang, 
Stochastic nonlinear Schr\"odinger equations, 
{\em Nonlinear Anal.}, 136, 168–194 (2016).


\bibitem{BRZ2017}
V. Barbu, M. R\" ockner and D. Zhang, 
Stochastic nonlinear Schr\"odinger equations: No blow-up in the non-conservative case,
{\em J. Differential Equations}, 263, 7919-7941 (2017). 


\bibitem{Bog}  
V. I. Bogachev,  
Gaussian Measures, 
{\em Mathematical Surveys and Monographs}, vol. 62,
AMS, Providence, R.I., 1998. 

\bibitem{Brz_1997} 
Z. Brze{\'z}niak,  
{On  stochastic  convolution  in  {B}anach  spaces  and applications}, 
{\em Stochastics and Stochastics Reports}, {61},  n. 3-4, 245--295 (1997).


\bibitem{Brz_Mil} 
Z. Brze{\'z}niak, and A. Millet,  
{On the stochastic {S}trichartz estimates and the stochastic
nonlinear {S}chr\"odinger equation on a compact {R}iemannian manifold}, {\em Potential Anal.}, {41}, n.2, 269--315 (2014).

\bibitem{Caz-book}
T. Cazenave, 
\emph{Semilinear Schr\"odinger equations}, 
Courant Lecture Notes in Mathematics, 10. NYU, Courant Inst.; AMS, Providence, RI, 2003. xiv+323 pp.

\bibitem{Caz_Wei} 
T. Cazenave and F. Weissler, 
The Cauchy problem for the critical nonlinear Schr\"odinger equation in $H^s$,
{\em Nonlinear Analysis, Theory, Methods \& Applications}, {14-10},  807--836  (1990).

\bibitem{CW1988}
T. Cazenave and F. Weissler, 
The Cauchy problem for the nonlinear Schr\"odinger equation in $H^1$,
{\em Manuscripta Math.}, {61-4},  {477–494} (1988).


\bibitem{CW1990}
T. Cazenave and F. Weissler,
{The Cauchy problem for the critical nonlinear Schr\"odinger equation in $H^s$}, 
{\em Nonlinear Anal.}, 14, no. 10, 807–836 (1990).


\bibitem{DaP-Zab} 
G. Da Prato and J. Zabczyk, 
{Stochastic equations in infinite dimensions}, 
{\em Encyclopedia of Mathematics and its Application},
Cambridge: Cambridge University Press, 1992. 


\bibitem{deB_Deb_CMP} 
A. de Bouard and A. Debussche, 
{A Stochastic Nonlinear Schr\"odinger Equation with Multiplicative Noise}, 
{\em Comm. Math. Phys.}, {205}, 161-181 (1999).

 
 
 
\bibitem{deB_Deb_PTRF} 
A. de Bouard and A. Debussche, 
On the effect of a noise on solutions of the focusing supercritical nonlinear Schr\"odinger equation,
{\em Probab. Theory Relat. Fields}, {123}, 76--96  (2002).  

\bibitem{deB_Deb_H1} 
A. de Bouard and A. Debussche, 
The Stochastic Nonlinear Schr\"odinger Equation in $H^1$, 
{\em Stochastic Analysis and Applications}, {21-1}, 97--126 (2003). 


\bibitem{deB_Deb_AnnProb} 
A. de Bouard and A. Debussche, 
Blow-up for the stochastic Nonlinear Schr\"odinger equation with multiplicative noise,
{\em The Annals in Probability}, {33-3}, 1078--1110 (2005).

\bibitem{DR2015}
T. Duyckaerts, J. Holmer and S. Roudenko, 
Going beyond the threshold: scattering and blow-up in the focusing {NLS} equation,
{\em Comm. Math. Phys.},
v. {334}, no.3, {1573--1615}  (2015).


\bibitem{DHR}
T. Duyckaerts, J. Holmer and S. Roudenko, 
{Scattering for the non-radial 3D cubic nonlinear Schr\"odinger equation},
{\em Math. Res. Lett.}, 
v. {15}, {no. 6}, {1233-1250} (2008).

\bibitem{Fa_Li_Po} 
L. Farah, F. Linares, and G. Ponce,  
The supercritical generalized KFD equation: global well-posedness in the energy space and below, 
{\em Math. Res. Lett.}, {18-2}, 357--377 (2011). 



\bibitem{GV1979}
J. Ginibre and G. Velo,
On a class of nonlinear Schrödinger equations. I. The Cauchy problem, general case.
{\em J. Functional Analysis}, {32-1},  1–32 (1979).

\bibitem{GV1985}
J. Ginibre and  G. Velo,
The global Cauchy problem for the nonlinear Schr\"odinger equation revisited.
{\em Ann. Inst. H. Poincaré Anal. Non Linéaire}, {2-4},  309–327  (1985).

\bibitem{Hol_Rou_2007} 
J. Holmer and S. Roudenko,  
On blow-up solutions to the 3D cubic nonlinear Schr\"odinger equation, 
{\em Applied Mathematics Research eXpress (AMRX)}, Article ID 004, 29 pages (2007).
 
\bibitem{Hol_Rou_2008} 
J. Holmer and S. Roudenko, 
A sharp condition for scattering of the radial 3{D} cubic nonlinear {S}chr\"odinger equation, 
{\em Commun. Math. Phys.}, {282}, 435--467 (2008).

\bibitem{K1987}
T. Kato,
On nonlinear Schrödinger equations, 
{\em Ann. Inst. H. Poincar\'e,  Phys. Th\'eor.}, {46 -1},  113–129  (1987).

\bibitem{KenMer} 
C. E. Kenig and F. Merle,  
Global well-posedness, scattering and blow-up for the energy-critical, focusing, non-linear Schr\"odinger equation in the radial case, 
{\em Invent. Math.},  {166}, 645-675 (2006).

\bibitem{Kwa_Szy} 
S. Kwapie\'n and B. Szyma\'nski, 
Some remarks on Gaussian measures in Banach spaces, 
{\em Probab. Math. Statist.}, {1-1}  59-65 (1980).

\bibitem{MilRou} 
A. Millet and S. Roudenko, Well-posedness of the focusing stochastic nonlinear  Schr\"odinger equation: $L^2$-critical and supercritical cases, 
{arXiv:2511.07072}.
 

\bibitem{Oh_Oka} 
T. Oh and M. Okamoto, 
On the stochastic Nonlinear Schr\"odinger equations at critical regularities, 
{\em Stoch. Partial Differ. Equ. Anal. Comput.},
{8-4}, 869--894 (2020).



\bibitem{Temam} 
R. Temam, 
Sur un probl\`eme non lin\'eaire, 
{\em J. Math. Pures Appl.}, {48}, 159-172 (1969).

\bibitem{Tao} 
T. Tao, 
Nonlinear Dispersive Equations: local and global analysis, 
{\em CBMS}, {106},
American Mathematical Society, 2006.

\bibitem{T1987}
Y. Tsutsumi,
$L^2$-solutions for nonlinear Schr\"odinger equations and nonlinear groups,
{\em Funkcial. Ekvac.}, {30-1},  115–125  (1987).


\bibitem{vanNee} 
J. van Neerven, 
$\gamma$-radonifying operators - a survey, 
{The {AMSI}-{ANU} {W}orkshop on {S}pectral {T}heory and {H}armonic {A}nalysis},
{\em Proc. Centre Math. Appl. Austral. Nat. Univ.}, {44}, {1--61} (2010).
      




\bibitem{Wein} 
M. I. Weinstein, 
Nonlinear Schr\"odinger equations and sharp interpolation estimates, 
{\em Comm. Math. Phys.}, {87}, 567-576 (1983). 

\bibitem{Zh}  
P. Zhidkov, Korteweg-de Vries and nonlinear Schr\"odinger equations: qualitative theory, 
{\em Lecture Notes in Mathematics}, 1756. Springer-Verlag, Berlin, 2001. vi+147 pp
\end{thebibliography}
\end{document}